\documentclass[letterpaper,10pt,reqno]{amsart}
\usepackage[usenames,dvipsnames]{xcolor}

\RequirePackage[OT1]{fontenc}
\usepackage{underscore}
\usepackage[portrait,margin=3.5cm]{geometry}
\usepackage{mathrsfs}
\usepackage[colorlinks,citecolor=blue,urlcolor=blue,linkcolor=blue,linktocpage=true]{hyperref}
\usepackage[foot]{amsaddr}
\usepackage{amssymb,amsthm,amsfonts,amsbsy,latexsym,dsfont}
\usepackage[textsize=small]{todonotes}       
\usepackage{graphicx}
\usepackage[english]{babel}
\usepackage{subfigure}
\usepackage[title, header]{appendix}
\usepackage{enumerate}
\newenvironment{enumeratei}{\begin{enumerate}[\upshape (i)]}{\end{enumerate}}

\usepackage{color}              
\usepackage{graphicx}
\usepackage[]{amsmath}
\usepackage{amsthm}
\usepackage{amssymb}
\usepackage{bbm}
\usepackage{hyperref}
\usepackage{comment}
\usepackage{microtype}

\definecolor{MyDarkblue}{rgb}{0,0.08,0.50}
\definecolor{Brickred}{rgb}{0.65,0.08,0}

\hypersetup{
colorlinks=true,       
    linkcolor=MyDarkblue,          
    citecolor=Brickred,        
    filecolor=red,      
    urlcolor=cyan           
}

\newtheorem{theorem}{Theorem}[section]
\newtheorem{lemma}[theorem]{Lemma}
\newtheorem{proposition}[theorem]{Proposition}
\newtheorem{corollary}[theorem]{Corollary}

\newtheorem{definition}[theorem]{Definition}

\newtheorem{remark}[theorem]{Remark}
\newtheorem{claim}[theorem]{Claim}

\newcommand{\Pv}{\mathbb{P}}

\newcommand{\Ev}{\mathbb{E}}

\newcommand{\set}[1]{\left\{#1\right\}}

\newcommand{\sss}{\scriptscriptstyle}

\newcommand{\CE}{{\mathcal{E}}}

\newcommand*{\CMD}{{\mathrm{CM}}_n(\boldsymbol{d})}

\newcommand{\e}{{\mathrm e}}

\numberwithin{equation}{section}

\newcommand{\R}{\mathbb{R}}
\newcommand{\N}{\mathbb{N}}
\newcommand{\Z}{\mathbb{Z}}

\newcommand{\indicator}[1]{\mathbbm{1}_{\set{#1}}}

\newcommand{\CA}{\mathcal {A}}
\newcommand{\CB}{\mathcal {B}}

\newcommand{\CL}{\mathcal {L}}

\newcommand{\CO}{\mathcal {O}}
\newcommand{\CP}{\mathcal {P}}

\newcommand{\CR}{\mathcal {R}}
\newcommand{\CS}{\mathcal {S}}

\newcommand{\CU}{\mathcal {U}}

\newcommand*{\wih}{\widehat}

\newcommand*{\vr}{\varrho}
\newcommand*{\la}{\lambda}

\newcommand*{\de}{\delta}

\newcommand*{\ve}{\varepsilon}

\newcommand*{\al}{\alpha}

\newcommand*{\Vv}{{\text{\bf Var}}}

\newcommand*{\be}{\begin{equation}}
\newcommand*{\ee}{\end{equation}}
\newcommand*{\ba}{\begin{aligned}}
\newcommand*{\ea}{\end{aligned}}
\newcommand*{\barr}{\begin{array}{c}}
\newcommand*{\earr}{\end{array}}
\def \toinp    {\buildrel {\Pv}\over{\longrightarrow}}
\def \toindis  {\buildrel {d}\over{\longrightarrow}}

\newcommand*{\Yrn}{Y_r^{\scriptscriptstyle{(n)}}}
\newcommand*{\Ybn}{Y_b^{\scriptscriptstyle{(n)}}}

\newcommand*{\Ybnla}{\big(Y_b^{\scriptscriptstyle{(n)}}\big)^\la}

\newcommand*{\CMBN}{\mathcal M^{\sss{(b)}}_n}
\newcommand*{\MBN}{M^{\sss{(b)}}_n}
\newcommand*{\wit}{\widetilde}
\newcommand*{\opt}{\mathcal{O}}
\newcommand*{\ops}{\mathrm{O}}

\newcommand*{\opss}{\mathrm{Opt}}
\newcommand*{\ind}{\mathbbm{1}}
\def\namedlabel#1#2{\begingroup
    #2%
    \def\@currentlabel{#2}%
    \phantomsection\label{#1}\endgroup
}

\begin{document}
	\title[Competition on Configuration model for $\tau\in (2,3)$]{Fixed speed competition on the configuration model with infinite variance degrees: unequal speeds}

	\date{\today}

	\author[Baroni]{Enrico Baroni}
	\author[van der Hofstad]{Remco van der Hofstad}
	\address{Department of Mathematics and
	    Computer Science, Eindhoven University of Technology, P.O.\ Box 513,
	    5600 MB Eindhoven, The Netherlands.}
\author[Komj\'athy]{J\'ulia Komj\'athy}
	\email{e.baroni@tue.nl, rhofstad@win.tue.nl, j.komjathy@tue.nl}

\begin{abstract}
We study competition of two spreading colors starting from single sources on the configuration model with i.i.d.\ degrees following a power-law distribution with exponent $\tau\in (2,3)$.
In this model two colors spread with a fixed but not necessarily equal speed
on the unweighted random graph.  We show that if the speeds are not equal, then the faster color paints almost all vertices, while the slower color can paint only a random subpolynomial fraction of the vertices.
We investigate the case when the speeds are equal and typical distances in a follow-up paper.
\end{abstract}

\maketitle

\section{Introduction and results}
\subsection{The model and the main result}
Let us consider the configuration model $\CMD$ on $n$ vertices, where the degrees $D_v, v\in \{1,2,\dots, n\}:=[n]$ are i.i.d.\ with a power-law tail distribution. That is, given the number of vertices $n$, to each vertex we assign a random number of half-edges drawn independently from a distribution $F$ and the half-edges are then paired randomly to form
edges. In case the total number of half-edges $\CL_n:=\sum_{v\in[n]} D_v$ is not even, then we drop one half-edge from $D_n$ (see below for more details).
 We assume that
\be\label{eq::F} \frac{c_1}{x^{\tau-1}}\le 1- F(x)= \Pv(D>x) \le \frac{C_1}{x^{\tau-1}},\ee
 with $\tau \in (2,3)$, and all edges have weight $1$. We assume $\Pv(D\ge 2)=1$ guaranteeing that the graph has almost surely a unique connected component of size $n(1-o(1))$ see e.g.\ \cite[Theorem 10.1]{H10} or \cite{MolRee95, MolRee98}.\\
 We further denote the mass function of $-1$ plus the \emph{size-biased version} of $D$ by
 \be\label{def::size-biased1}
f^\ast_{j}:=\frac{(j+1)\Pv(D=j+1)}{\Ev[D]}\text{, \ }j\geq 0.
\ee
We write $F^\ast(x)$ for the distribution function $F^\ast(x)=\sum_{j=0}^{\llcorner x\lrcorner}f^\ast_{j}$.

 Pick two vertices $\mathcal R_0$ (red source) and $\mathcal B_0$ (blue source) uniformly at random in $[n]$,
 and consider these as two sources of spreading infections.
 Each infection spreads \emph{deterministically} on the graph:
for color blue it takes $\la$ time units to pass through an edge, while color red needs $1$ unit of time for that.
  Without loss of generality we can assume that $\la>1$. Each vertex is painted the color of the infection that reaches it
  first, keeps its color forever, and starts coloring the outgoing edges at the speed of its color.  When the two colors reach a vertex at the same time, the vertex gets color red
  or blue with an \emph{arbitrary adapted rule}, i.e. a rule that is not depending on the future. One examples of such a rule is when it is painted red or blue with probability $1/2$ each, independently of everything else. Another natural adapted rule is that a vertex, when the two colours arrive at it at the same time, is painted red or blue with probability proportional to the number of previously red and blue-colored neighbors of the vertex.\\
	Let $\mathcal R_t:=\mathcal R_t(n)$ and $\mathcal B_t:=\mathcal B_t(n)$
  denote the number of red and blue vertices occupied up to time $t$, respectively. We denote by $\CB_{\infty}:=\CB_{\infty}(n)$ the number of vertices eventually occupied by blue. We emphasise that the randomness in this model is only coming from the structure or topology of the graph and the uniform choice of the source vertices
 for the two colors; once these are settled, the dynamics is completely deterministic.
 
   Roughly speaking, the first main result of this paper, Theorem \ref{thm::main} below, tells us that in the quenched setting, i.e., for almost all realizations of the graph $\CMD$ and for almost all initial vertices $\CR_0, \CB_0$, the faster color always wins, that is, it gets $n-o(n)$ many vertices. Furthermore, the number of vertices the slower color paints is a subpolynomial of $n$. More precisely, blue paints whp $\exp\{ (\log n)^{2/(\la+1)} H(n, Y_r, Y_b)\}$ many vertices, i.e.,\ a stretched exponential in $\log n$ with exponent $2/(\la+1)<1$, and where the coefficient  $H(n, Y_r, Y_b)$ is a random function that depends on $n, \la, \tau$, and two random variables $Y_r$ and $Y_b$, that can intuitively be interpreted as some measure of `how good' the neighbourhoods of the source vertices are: the faster the local neighbourhoods grow, the larger these variables are. Moreover, $H(n, Y_r, Y_b)$ does not converge: it has an oscillatory part that exhibits  `$\log\log$-periodicity'.  
   
  The other main result, Theorem \ref{thm::maxdegreeblue}, shows that the degree of the maximal-degree vertex that blue ever occupies obeys asymptotic behaviour similar to blue's total number, with a strictly smaller coefficient in the exponent, and the same $\log\log$-periodicity.
  This phenomenon is due to integer part issues coming from the fact that the edge weights are concentrated on a lattice. We emphasise again that these results are quenched.
  
  To be able to state the main theorem precisely, let us define the following random variables:
  \begin{definition}\label{def::limit-variables}
  Let $Z_k^{\sss{(r)}}, Z_k^{\sss{(b)}}$ denote the number
 of individuals in the $k$th generation of two independent copies of a Galton-Watson process described as follows: the size of the first generation has distribution $F$ satisfying \eqref{eq::F},
 and all the further generations have offspring distribution $F^\ast$ from \eqref{def::size-biased1}.
 Then, for a fixed but small $\rho>0$ let us define
\be\label{def::yrn-ybn}\Yrn:=(\tau-2)^{t(n^\rho)} \log (Z^{\sss{(r)}}_{t(n^\rho)}),  \quad \Ybn:=(\tau-2)^{\lfloor t(n^\rho)/\la \rfloor} \log (Z^{\sss{(b)}}_{\lfloor t(n^\rho)/\la \rfloor}),\ee
  where $t(n^\rho)=\inf_k\{Z_k^{\sss{(r)}} \ge n^\rho\}$.
Let us further introduce
  \be\label{def::Y} Y_r:= \lim_{k\to\infty} (\tau-2)^k \log (Z_k^{\sss{(r)}}), \quad Y_b:=\lim_{k\to\infty} (\tau-2)^k \log (Z_k^{\sss{(b)}}).\ee
  \end{definition}
  We will see below in Section \ref{sc::BP} that these quantities are well-defined and that $(\Yrn, \Ybn)\toindis (Y_r,Y_b)$ from \eqref{def::Y} as $n\to \infty$. With these notation in mind, we have the following theorem:
\begin{theorem}\label{thm::main}
Fix $\lambda>1$. Then, $\lim_{n \to \infty}\CR_\infty/n =1$ whp. Further, 
there exists a bounded and strictly positive random function $C_n(\Yrn,\Ybn)$ such that as $n\to\infty$
\be \label{eq::thm-main}   \frac{\log (\mathcal B_\infty)}{(\log n)^{\tfrac{2}{\la+1}} C_n(\Yrn,\Ybn)} \toindis  \left(\frac{Y_b^\la}{Y_r}\right)^{\frac{1}{\la+1}}. \ee

We identify $C_n(\Yrn, \Ybn)$ in \eqref{def::norm-constant} as a deterministic, oscillating (non-convergent) function of $\tau, \la, n,\Yrn, \Ybn$. $C_n(\Yrn,\Ybn)$ has the uniform (non-tight) bounds
\[     \left(\frac{(\tau-2)}{(\tau-1)^{2}}\frac{(3-\tau)}{1-(\tau-2)^\la}\right)^{\frac{1}{\la+1}} < C_n(\Yrn, \Ybn) < (\tau-2)^{-2}\cdot \frac{4-\tau}{1-(\tau-2)^{\la}}.  \]\end{theorem}
\begin{remark}\normalfont
We also give the accompanying tight bounds on $C_n(\Yrn, \Ybn)$, see \eqref{eq::tight-bounds}.
\end{remark}

   Let us denote
   \be\label{def:maxdeg-blue} D_{\max}^{\sss{(b,n)}}(\infty):= \max_{i\in \mathcal B_{\infty}} D_i\ee
   the degree of the maximal degree vertex eventually occupied by blue. As a side result of the proof of Theorem \ref{thm::main}, we get the following theorem:
   \begin{theorem}\label{thm::maxdegreeblue} Fix $\la>1$. There exists a bounded and strictly positive random function $C_n^{\max}(\Yrn, \Ybn)$ defined below in \eqref{def::dn}, such that as $n\to \infty$
\be\label{eq::maxdegreeatinfty} \frac{\log  D_{\max}^{\sss{(b,n)}}(\infty)}{(\log n)^{\frac{2}{\la+1}} \ C_n^{\max}(\Yrn, \Ybn)} \toindis \left( \frac{Y_b^{\la}}{Y_r}\right)^{\frac{1}{\la+1}} \ee
 Further,  $C_n^{\max}(\Yrn,\Ybn)$ is stochastically dominated by $C_n(\Yrn,\Ybn)$ whp, and
\be\label{eq::dnmaxbounds} \left(\frac{(\tau-2)^{2+\la}}{(\tau-1)^2}\right)^{\frac{1}{\la+1}} \le C_n^{\max}(\Yrn,\Ybn) \le \left(\frac{\tau-2}{4}\right)^{\frac{1}{\la+1}}\!\!\!\!. \ee
\end{theorem}

 \begin{remark}\normalfont
We emphasise that these results are valid for any adapted rule of decision when the two colors jump at the same time to a vertex. In case $\la$ is irrational, clearly, this rule will never be used.  If $\la$ is rational, then the normalisation random variables $C_n(\Yrn,\Ybn)$, $C_n^{\max}(\Yrn,\Ybn)$ depend on the rule -- they are slightly different if the rule is so that these vertices are always painted red, from the case when there is a positive chance that these vertices are painted blue, but the upper and lower bounds on $C_n(\Yrn,\Ybn), C_n^{\max}(\Yrn,\Ybn)$ remain the same.   On the other hand, when $\la=1$, this rule will play an important role in the outcome.
  \end{remark}

\begin{remark}[More than two colors]\normalfont
If there are a finite number of colors with edge passage-times $1\!=\!\la_1< \la_2\le \dots \le \la_k$, then the statements of Theorems \ref{thm::main} and \ref{thm::maxdegreeblue} stay valid for each $\la_i,\  2\le i\le k$, with limit variables $(Y_i^{\la_i}/Y_1)^{1/{(\la_i+1)}}$ on the right hand side of  \eqref{eq::thm-main}, where $Y_i$ are i.i.d.\ copies of $Y$.
 The reason for this is that with high probability each slower color only meets the fastest color and never meets the other slow ones. That is, the clusters of slower colors are separated from each other by the cluster of the fastest color.
\end{remark}

\subsection{Related work and discussion}
First we give a (non-complete) overview of the literature on competition on different graph models. Then we mention some more applied results.

In a seminal paper \cite{HP98} H\"aggstr\"om and Pemantle introduced competition on the grid $\Z^d$. The model is called the two-type Richardson's model,
and it describes the dynamics of two (red and blue) infections with single source vertices $v_{0}, v_{1}\in \Z^d$ that compete to conquer the grid $\Z^d$.  In this continuous-time model, a vertex of $\Z^d$ gets a given color with rate  proportional to the number of infected neighbours of that color; then, once a vertex is infected, it keeps its color forever. Note that the evolution of a single color without the presence of the other color has independent exponential passage times across edges, and a vertex gets infected at the time that equals the minimal length path from the source to the vertex. Hence, a single color process is often called  \emph{first passage percolation} in the literature.
Multiple colours then lead to the name \emph{competing first passage percolation}.

For two colors, we have two possible evolution scenarios: in the first, one of the growing clusters completely blocks the growth of the other color -- by surrounding it -- and then it infects all the remaining healthy vertices. In the second scenario the two clusters continue to grow unboundedly forever: this is called \emph{coexistence}.
The important question is: does coexistence occur with positive probability?
H\"aggstr\"om and Pemantle \cite{HP98} proved that this is the case for the $\Z^{2}$ grid with i.i.d.\ exponential passage times. Later this result has been extended by Garet and Marchand \cite{GarMar05} for $\Z^{d}$, $d\geq 2$ for a vast class of passage time distributions under mild hypothesis.
For further literature on the Richardson model see \cite{DeiHag06, DeiHag07, GarMar08, MR2114988, MR2462555}.

Recently, a noticeable scientific interest arose in understanding the structure of large but finite networks and the behaviour of spreading processes on these networks. Typically, results on these topics are called first passage percolation, see e.g. \cite{BHH10, BHH11, BHH14}. It is then natural to ask what happens when one considers competition of multiple spreading processes on these networks. When studying competitive spreading, one might also gain a more detailed understanding of the structure of these graphs.

The idea of competitive spreading on finite random graph sequences raises several questions. First and foremost, due to the finite size of the graphs the main questions about these models must be rephrased, since infinite growth can never happen. Thus, the definition of coexistence had to be modified in this setting. Consider two competing colors on a sequence of random graphs: is there an asymptotic coexistence of the two colors? That is, is it possible  that both colors paint a positive proportion of vertices with positive probability, as the size of the graph tend to infinity?
If this is not the case, can we determine the number of eventually occupied vertices for both colors in terms of the size of the graphs? What happens if we modify the passage dynamics so that the two infections have different rates of growth $\lambda_{1}$ and $\lambda_{2}$? Here we give a (non-complete) overview of the existing literature on these topics for different random graph models.

Antunovic, Dekel, Mossel and Peres \cite{ADMP11}
give a detailed analysis of competition on random regular graphs (degree at least $3$) on $n$ vertices with i.i.d. exponential edge weights. They analyse the number of eventually occupied vertices by both colors as a function of the speeds $\lambda_{1},\lambda_{2}$ and of the initial number of infected vertices, that might even grow with $n$. They show that asymptotically almost surely the color with higher rate occupies $n-o(n)$ vertices and the slower color paints approximately $n^\beta$ vertices for some deterministic function $\beta(\la_1, \la_2)$.  Their result include asymptotic coexistence for equal speeds $\la_1=\la_2$ for infections starting from single sources.

Next, van der Hofstad and Deijfen \cite{DH13} investigates competition with exponential spreading times on the configuration model with i.i.d.\ degrees coming from a  power-law distribution with exponent $\tau\in (2,3)$. They prove that even if the speeds are not equal, the `winner' color is \emph{random}, i.e. the color with slower rate can still take most of the graph. Moreover, the winning color paints all but a finite number of vertices. The randomness of the `winner' color comes from the fact that the underlying Markov branching process explodes in finite time, and the slower color has a positive chance to explode earlier than the faster color.

A slightly different, discrete time competition model is analysed by Antunovi\'c, Mossel and R\'acz in \cite{AntMosRac13}. There, the underlying random graph is the growing linear preferential attachment model, and vertices pick their color upon entering the network randomly from the colors of the vertices they attach to. The probability of picking a color is a (possibly linear) function of the number of neighbors with the given color, called the \emph{coloring function}. The authors analyse coexistence of colors in terms of the properties of the coloring function. Note that in this case the graph has power law $\tau=3$. The proofs are based on comparison to P\'olya urns.

Finally, this paper considers competition on the configuration model with i.i.d.\
power-law degrees with exponent $\tau\in(2,3)$, but with \emph{deterministic} unit edge-weights.
Theorem \ref{thm::main}  shows that the fact that the edge weights have a support separated from zero entirely changes the picture observed in \cite{DH13}:
 when the speeds are unequal, the faster color always paints $n-o(n)$ vertices, and the slower color can paint only subpolynomial  many vertices.

If the speeds are equal, then the phenomena is richer: as a side result of the analysis of the $\la=1$ case, we obtain precise distributional limits of the second order terms in typical distances in the graph. Further, we conjecture that there is still no coexistence with high probability, and the loser type can paint a polynomial many vertices with a random exponent that is less than $1$. However, this random exponent sensitively depends on the initial local neighbourhoods of the source-vertices and shows different behaviour if the corresponding random variables are within a very specific constant factor of each other or if they are not. Due to the length of the analysis of this case and to put more highlight on the rich phenomena that comes with it, we decided to put the equal-speed case  in a subsequent paper soon to be published.

From the more applied perspective, competition on networks is present in many aspects of our life.
To start with an example, in marketing, companies compete for customers who are connected via their acquaintance network, and they provide word-of-mouth recommendations and opinions about the services of the different companies, see \cite{FudBan04, FudEll95}.  For economic studies on the importance of word-of-mouth, see e.g.\ \cite{Arn67, Butt98, Dicht66}.
 Recently, `word-of-mouth' recommendations happen also on large scale on different social online media such as Facebook, and Twitter. For a survey on how online feedback mechanisms differ from original word-of-mouth recommendations and what challenges they pose, see \cite{Dell03}. The paper \cite{LesAdaHub07} analyses recommendation-based viral marketing on social media, where they use viral marketing also to identify communities of online networks. For recent economic studies of the importance of word-of-mouth recommendations, see e.g.\ \cite{Chev06, KumPetLeo07}.

 In epidemiology, viruses and bacterial infections spread through society.
 In this setting, competition can happen among different strains of a pathogen, see e.g.
 \cite{Lip09} for a study under which conditions coexistence can occur and references therein. In the physics community, \cite{AhnJeo06, KarNew11, Newman05} study the effect of the underlying network on co-existence of competing viruses.

The  epidemiological analogies have been further exploited by \cite{PraBeu12}, where they study a variation of susceptible-infectious-susceptible epidemic spread, where two epidemics are immune to each other, and the authors show that one of them completely takes over (similarly as in \cite{DH13}). Then, \cite{Beu12} studies how partial immunity can cause coexistence in the previous model.

\subsubsection*{Discussion and open problems}
 The analysis of competition on the configuration model is far from complete. One can for instance ask about different spreading dynamics (edge lengths) and different power-law exponents. Further, one can ask what happens if the colors have entirely different passage time distributions (e.g.\ one is explosive and the other is not), or what happens if one of the colours have a main advantage by starting from one or many initial vertices of very high degree. These can correspond to e.g.\ competition advantage of different product on the network or to different marketing strategies.
Here we list some conjectures for uniformly picked single vertex sources of infections on $\CMD$ with i.i.d. power law degrees of distribution $D$ with exponent $\tau$. We further assume that the time to passage times can be represented as i.i.d. random variables on edges, from distribution $I_r$, $I_b$ for red and blue, respectively.

\noindent 1.\emph{ $\tau\in (2,3)$:}

A.  If the spreading dynamics are so that the underlying
branching processes defined by $D,I_r$ and $D, I_b$ are both explosive, then we conjecture that there is \emph{never} coexistence and either
of the two colors can win. This is one of our ongoing research projects.

B. If the underlying branching process for one color has explosive spreading while the other one has not, than we suspect that the explosive one always wins.

C. If  both underlying branching processes are non-explosive, and further assume $I_r{\buildrel d \over = }\ \la I_b$, then we guess
that there is no coexistence if $\la\neq1$ (the fastest color wins). We suspect that the number of vertices the `loser' color paints
 depends sensitively on the weight distribution. The outcome in the $\la=1$ case might sensitively depend on the weight distribution.

\noindent2. \emph{$\tau>3$:}

D.  We suspect that if the transmission times $I_r, I_b$ both have continuous distribution, and the branching process approximations of them have different Malthusian parameters, then there is no coexistence, and the number of vertices painted by the slower color is $n^{\beta}$ for some $\beta\in(0,1)$.
When the Malthusian parameters agree, we suspect that there is asymptotic co-existence.

\noindent 3.\emph{ $\tau=3$:}

E. In this case $\Pv(D>x) = L(x)/ x^2$, with $L(x)$ a slowly varying function at infinity. We suspect that $L(x)$ and the transmission distributions $I_r, I_b$ jointly determine into which category among A, C, D above the spreading of the colours belongs to: if the $\mathbf{x \log x}$ criterion holds for the underlying age-dependent branching process with $D, I_r$ and $D, I_b$ then we expect that the model will show similar phenomena as in case D. If the underlying branching processes are explosive, then similar phenomena is expected as in case A, and if it is none of these two, then as in case C. Further, if the two colours have significantly different dynamics, i.e.\ one is explosive and the other one is not, then we conjecture that case B applies.

\subsection{Overview of the proof and structure of the paper}
The heuristic idea of the proof is as follows: we can start growing the two clusters simultaneously. The growth has six phases, each corresponding to a section below, described as follows:
\begin{enumeratei}
\item \label{ph::bp} \emph{Branching process phase.}\\ At first, whp, the two colored clusters do not meet and the
 growth of both clusters is characterised by the growth rate of the branching process (BP) to which they can be coupled.
  This we call the \emph{branching process phase}. The length of this phase is of order $\log\log n/ |\log(\tau-2)| + O(1)$. Then,
   the faster color (red) reaches the area where the coupling fails to remain valid: $\mathcal R_t$ reaches size $n^\vr$ for some $\vr>0$.
 \item \label{ph::montain_up}\emph{Mountain climbing phase.}\\
    At this point, we start making use of the structure of high-degree vertices in the graph: due to high connectivity, the subgraph formed by high-degree vertices can be represented as a `mountain' where the height function is linear in the $\log\log$-degree. Level sets of this mountain represent vertices with degree of the same order of magnitude, with the maximal degree in the graph at the top of the mountain. We partition this mountain into \emph{layers} -- that is, constant length intervals on a $\log\log$-scale -- and we show that every vertex in a given layer has at least one neighbour in one layer higher.
As a result, we show the existence of a path for red through
   these layers of vertices of higher and higher degree such that the path reaches some vertex with degree larger
     than  $n^{(\tau-2)/(\tau-1)}$ at the end.  This we call the \emph{mountain climbing phase}. The climbing phase lasts only finitely
     many steps, but the constants turns out to be important, so we perform a rather careful analysis. We denote the total time of the branching process phase and the climbing phase for red by $T_r$.
 \item \label{ph::peak}\emph{Crossing the peak of the mountain.} \\
    We handle how the color red goes through the peak of the `degree-mountain' very carefully.  Vertices of degree much larger than $\sqrt{n}$ form a subgraph that is a complete graph, hence it takes only one step to paint all the very high degree vertices,
     but the degree of vertices to which the faster color arrives at the end of this single step is delicately depending on the initial random growth rates of the branching processes and their integer and fractional part issues.
 \item \label{ph::mountain_down}\emph{Red avalanche from the peak.}\\
 After crossing the mountain, red starts sloping down to layers of vertices of smaller and smaller degree. Since it is still true that each vertex in a layer is connected to at least one vertex in one layer higher, this means that in each additional step, red paints
     \emph{all} the vertices in one layer lower. We call this the \emph{avalanche-phase of red}. (One can imagine
      this as red being a very careless climber who -- after crossing the peak of a mountain -- steps in the snow with a bucket of red paint and starts
      a huge painted avalanche.)
\item \label{ph::meeting}\emph{At the collision time.}\\
Now we turn our attention to the blue climber who does essentially the same as red except that it is slower:
after getting out of its local neighbourhood corresponding to the branching process, blue starts its mountain climbing phase as well.
Since it is slower, whp it will only reach some low layer of the degree-mountain when red starts its avalanche.
With this picture in hand, we can identify the maximal degree vertex eventually painted blue - this is the vertex in the highest layer blue can still reach.
 The idea of the proof is to determine the value $\ell$ such that during the total time $T_r+\ell$, blue has
 climbed up to the same layer as the red avalanche has sloped down to. Since red occupies \emph{every vertex} in a layer
 it reaches, it will necessarily bump into blue, who whp reaches only some vertices in that layer. This determines the time when red starts successfully blocking blue.
 \item \label{ph::after_meeting}\emph{Competing with the avalanche.}\\
 After the meeting time $T_r+\ell$, blue cannot go higher up on the mountain since red already occupies every
  vertex having degree higher than the maximal degree of blue.  Note that at this time most of the graph is still not reached by any color: we need to estimate the number of vertices that blue can still reach before the red avalanche closes up around the blue cluster. This is done in two steps: heuristically, every vertex that is close enough to a blue half-edge occupied at or before $T_r+\ell$ has a high chance to become blue later. Hence, first we calculate the size of the `optional cluster of blue', i.e.\ we calculate the size of the $k$-neighborhood of blue half-edges via path counting methods. The size of the optional cluster is convergent if $k\to\infty$: due to the presence of the red avalanche, the degrees in the blue paths get more and more restricted and finally the red avalanche reaches constant order vertices and then the procedure stops.

 It can still happen that some vertices in the optional cluster of blue are occupied by red simply because they are `accidentally' also close to some red vertex. Thus, in the second step we estimate the size of the intersection between the optional cluster of blue and the red cluster. The two steps together provide a matching upper and lower bound for the number of vertices that blue occupies after the intersection. This phase has a non-negligible impact on the order of magnitude of vertices painted blue since the constant $C_n$ in the exponent of \eqref{eq::thm-main} is influenced by this last phase.
  \end{enumeratei}

\subsubsection*{Notation}
We write $[n]$ for the set of integers $\{1,2,\dots, n\}$. We denote by the same name and add a superscript $(r), (b)$ to random variables, sets or other quantities belonging to the red and blue processes, respectively.  We write $E(\CMD)$ for the set of edges. For any set of vertices $S\subset [n]$, we write $N(S)$ for the set of their neighbors, i.e.,
\be \label{def::ns}N(S)=\{y\in [n]: \exists x\in S, (x,y) \in E(\CMD)\}.\ee
 For any event $A$, $\Pv_n(A):=\Pv(A| D_1, D_2, \dots, D_n)$. As usual, we write i.i.d.\ for independent and identically distributed, lhs and rhs for left-hand side and right-hand side. We write $\lfloor x\rfloor, \lceil x \rceil$ for the lower and upper integer part of $x\in \R$, and $\{x\}$ for the fractional part of $x\in \R$. Slightly misusing the notation, we use curly brackets around set elements, events and exponents as well.
 We say that a sequence of events $\mathcal E_n$ occurs with high probability (whp) when $\lim_{n\to \infty}\Pv(\mathcal E_n) = 1.$ In this paper, constants are typically denoted by $c$ in lower and $C$ in upper bounds (with possible indices), and their precise values might change even along lines. Typically, all the whp-events hold whp under the event $\{\CL_n\in [1/2 \Ev[D]n, 2 \Ev[D] n] \}.$

\section{The branching process phase}\label{sc::BP}
First we describe the exploration process of the local neighbourhood of a given vertex in order to relate it to a branching process.

The configuration model $\CMD$ (introduced in \cite{Boll80},  for more see \cite{Bollobas01, H10}) on $n$ vertices with i.i.d.\ degree distribution $D$ can be briefly described as follows:
for each vertex $i\in[n]$ we assign  an i.i.d. random variable $D_i\sim D$, and attach $D_i$ half-edges to that vertex.
If the total degree $\CL_n=\sum_{i=1}^n D_i$ is odd, then we add an extra half-edge to the vertex $n$.
Then we number the half-edges in an arbitrary way from $1$ to $\CL_n$,
and  start pairing them uniformly at random, i.e.\  we pick an arbitrary unpaired
half-edge and pair it to a uniformly chosen other unpaired half-edge to form an
edge. Once paired, we remove them from the set
of unpaired half-edges and continue the procedure until all half-edges are paired.
We call the resulting multi-graph $\CMD$. Since the choice of the half-edge to be paired is arbitrary,
we can start from any set of vertices, and explore their cluster simultaneously with the construction of the graph.
We call this procedure the \emph{exploration process}, which is a version of a Breadth First Search Algorithm on
the random graph $\CMD$. We describe the exploration process in more detail for the case when the initial set is a single uniformly chosen vertex $v\in [n]$, and relate it to a corresponding branching process as follows.

In each step of the exploration process, each vertex belongs to exactly one of three sets: it can be active ($\CA$), explored ($\CE$) or unexplored ($\CU$).
Initially $\CE_0=\varnothing$ and all vertices except $v$ are in $\CU_0$. We start setting the status of the initial vertex $v$ to active:
$\CA_0=\{v\}$, and we write $\CA_i$ for the set of active vertices after the $i$th step of the exploration.
In each step we pick a vertex $v_{i+1}$ from $\CA_i$ (we do this first-in-first-out way, i.e., we keep track of when a vertex enters the set $\CA$)
 and do three things: remove $v_{i+1}$ from $\CA_i$; add it to the explored vertices $\CE_i$;
 and put all its unexplored neighbors in the active set of vertices, i.e.,
\[ \ba \CA_{i+1}&:=\CA_{i}\setminus \{v_{i+1}\} \cup \{ N(v_{i+1})\cap \CU_i\},\\ \CU_{i+1}&:=\CU_i \setminus N(v_{i+1}), \ea\]
where $N(v_{i+1})$ denotes the neighbors of $v_{i+1}$ in $\CMD$. The explored vertices form the sequence $\CE_i=\{v=v_1, v_2, \dots, v_{i}\}$.

 Let $B_{i}$ stand for the \emph{forward-degree} of the vertex $v_{i}$ in the exploration process, so that $B_{i}=|N(v_{i})\cap \CU_i|$. We aim to determine the distribution of $B_i$. For this we note that in the construction of the random graph $\CMD$, an \emph{arbitrary} half-edge is chosen and paired to a uniformly chosen unpaired half-edge. Hence, we can do the construction of the graph together with the exploration process.
 Further, the probability of picking a half-edge which is belonging to a vertex with degree $j+1$ is proportional to $(j+1)f_{j+1}$, and as long as the size of the neighbourhood is small the probability that a vertex is connected to some vertex explored earlier vanishes. Hence, we get the size-biased distribution \eqref{def::size-biased1}
 as a natural candidate for the forward degrees of the vertices $v_i$ in the exploration process.
More precisely, we have the following result:
\begin{proposition}\label{prop::coupling}\emph{\cite[Proposition 4.7]{BHH10}}
There exists $0<\rho<1$ such that the random vector $\big(B_{i}\big)_{i=2}^{n^{\rho}}$ of forward degrees can be coupled
to an independent sequence of random variables $\big(\widetilde{B}_i\big)_{i=2}^{n^{\rho}}$ with probability
mass function given in \eqref{def::size-biased1} and $\big(B_{i}\big)_{i=2}^{n^{\rho}}=\big(\widetilde{B}_{i}\big)_{i=2}^{n^{\rho}}$ whp.
\end{proposition}
\begin{proof}
See \cite[Proposition 4.5]{BHH10} and the proof of Proposition 4.7 in \cite[Appendix A.2]{BHH10}.
\end{proof}
 In our case, we have two source vertices red and blue with different spreading speed, thus, we need a slight modification of this proposition. Namely, we need that a similar coupling remains
 valid for two exploration processes from two uniformly chosen vertices up to the time when the red (first) color reaches size $n^\rho$.
 Let us temporarily denote the number of vertices occupied by blue (the other) color by this time by $h(n,\rho)$.  This coupling is similar to \cite[Proposition 4.8]{BHH10}, but we state it for the reader's convenience:

\begin{lemma}\label{lem::coupling-for-two}Fix $\la>1$. Let $T(n^{\rho}):=\inf\{t: |\CR_{t}| \ge n^\rho \text{ or }  |\CB_{t}| \ge n^\rho \}$.
Then there exists a $\rho>0$ such that $\mathcal{R}_{t(n^{\rho})}\cap \mathcal{B}_{t(n^{\rho})}=\varnothing$ whp, and the forward degrees in both the red and the blue process can be coupled to i.i.d. sequences 
$\big(\widetilde B_{i}^{\sss{(r)}}\big)_{i=2}^{n^{\rho}}$ and $\big(\widetilde B_{j}^{\sss{(b)}}\big)_{j=2}^{ h(n,\rho)  }$, 
where $h(n, \rho)$ is the random number of vertices reached by blue up to time $t(n^\rho)$. (The same statement holds true for $\la=1$, but in this case the color to reach $n^\rho$ vertices first is random.)
\end{lemma}

\begin{proof}[Proof of Lemma \ref{lem::coupling-for-two}]

 First assume that $\la>1$:  consider the coupling given by Proposition \ref{prop::coupling} with forward degrees $\left(B_i^{\sss{(r)}}\right)_{i=2}^{n^\rho}$ for the red cluster. After this, connect the $n^\rho$-th chosen vertex
 to the blue source vertex $\CB_0$  with an imaginary edge. Then drop all the other active vertices from $\CA_{T(n^\rho)}$ and re-start the exploration process with only vertex $\CB_0$ being active.
 Since it takes time $\la$ to cover an edge for blue, up to time $T(n^\rho)$ blue reaches all the vertices which have graph distance at most $\lfloor T(n^\rho)/\la \rfloor$ from the source vertex $\CB_0$.
  Thus, continue the exploration process from the blue source up to finishing generation $\lfloor T(n^\rho)/\la \rfloor$. Since $\la >1 $, the total number of vertices found by this second phase
  has smaller order than $n^\rho$, so that the coupling still remains valid. Moreover, since each of the clusters have only at most $n^\rho$ many vertices, with high probability they do not meet each other.
Further, when $\la=1$, the proof is the same, the first cluster to reach $n^{\rho}$ vertices takes the role of red, and the other one takes the role of blue.
\end{proof}

An immediate consequence of Proposition \ref{prop::coupling} and Lemma \ref{lem::coupling-for-two} is that locally we can consider the growth
 of $\mathcal{R}_{t}$ and $\mathcal{B}_t $ as independent branching processes $(Z_{k})_{k>0}$  with offspring distribution $F^\ast$ for the second and further generations, and with offspring distribution given by $F$ for the first generation.

Let us now investigate the growth of these branching processes.
Since $\tau\in (2,3)$, the offspring distribution of this branching process has infinite mean for every individual in the second and larger generations.
 To understand the behavior of this BP, we first look at what happens in a BP where all the degrees are distributed as $F^\ast$, including the first generation.

The following theorem by Davies \cite{D78} describes the growth rate of such a branching process:

\begin{theorem}[Branching process with infinite mean \cite{D78}]\label{thm::davies} Let $\wit Z_k$ denote the $k$-th generation of a branching process with offspring distribution given by the distribution function $F^\ast$. Suppose there exists an $x_{0}>0$ and a function $x\mapsto\gamma(x)$ on $\R^+$ that satisfies the following conditions:
\begin{enumeratei}
\item  $\gamma(x)$ is non-negative and non-increasing,
\item $x^{\gamma(x)}$ is non decreasing,
\item $\int\limits_0^\infty\gamma\left(\mathrm e^{\mathrm e^x}\right)\mathrm d x<\infty$.
\end{enumeratei}
 Let us assume that for some $\tau\in(2,3)$, the tail of the offspring distribution satisfies that, for all $x\ge x_0$,
\begin{equation}\label{eq::davies_cond}
    x^{-(\tau-2)-\gamma(x)}\leq 1-F^\ast(x)\leq x^{-(\tau-2)+\gamma(x)}.
\end{equation}
Then $(\tau-2)^{k}\log(\wit{Z}_{k}\vee 1)$ converges almost surely to a random variable $\wit Y$. Further, the variable $\wit Y$ has exponential tails. \end{theorem}

To be able to apply this theorem to our setting, we need to show that the distribution function $F^\ast$ satisfies the condition \eqref{eq::davies_cond}. This is clearly the case since using the elementary re-arrangement of weights
\[ 1-F^\ast(x) = \sum_{j=x+1}^\infty \frac{(j+1)\Pv(D\!=\!j)}{\Ev[D]} = \frac{1}{\Ev[D]} \Big( (x+2)[1-F(x+1)] + \sum_{j=x+2}^{\infty} [1-F(j)]\Big),\]
combined with the bounds in \eqref{eq::F} and elementary estimates immediately yields that there exist constants $0<c_1^\ast\le C_1^\ast<\infty$, such that for $x$ large enough
\begin{equation}\label{eq::davies_appliedto_G}
    x^{-(\tau-2)} c_1^\ast \leq 1-F^\ast(x)\leq x^{-(\tau-2)} C_1^\ast.
\end{equation}
Since $\Pv(D\ge2)=\Pv(B\ge1)$, these BP-s cannot die out, i.e., we can write $\log \wit Z_k$ instead of $\log (\wit Z_k \vee 1 )$ and apply Davies' theorem  to obtain the a.s. convergence
 of $\wit{Y}_{k}=(\tau-2)^{k}\log(\wit{Z}_{k})$ to a random variable $\wit{Y}$. Recall that the degree of the first vertex in the exploration process is distributed as $F$ not $F^\ast$, hence we denote by $Z_k$ the corresponding BP and call it the \emph{delayed branching process}.
 The next lemma identifies the distribution of the limit of the properly scaled delayed branching process.
 We also identify the limit random variable $Y$ in terms of $\widetilde Y$.
 
\begin{lemma}\label{lem::convergence-non-delayed}
Let $Y$ be the limiting random variable $\lim\limits_{k\to\infty}(\tau-2)^{k}\log{Z_{k}}$ of the delayed BP. Then $Y$ satisfies the distributional identity
\be\label{eq::Y=max Y_{i}}
 Y \  {\buildrel {d}\over{\equiv}}\ (\tau-2)\max_{1\leq i\leq D	}\wit Y^{(i)},
\ee
where $\wit Y^{(i)}$ are i.i.d. copies of the limiting random variable of the original non-delayed BP. Further, \[ \wit Y {\buildrel {d}\over{\equiv}} (\tau-2) \max_{1\le i< D^*} \wit Y^{(i)}. \]
 \end{lemma}
\begin{remark}\normalfont An elementary calculation using \eqref{eq::Y=max Y_{i}} shows that $Y$ also has exponential tails with a parameter that is $(\tau-1)$ times the parameter of $\wit Y$.
\end{remark}
\begin{proof}[Proof of Lemma \ref{lem::convergence-non-delayed}]
Since the number of offsprings in the first generation is distributed as $D$, by the branching property the subtrees starting from the first generation up to level $k$
   are distributed as $\wit Z_{k-1}$ and are independent of each other. Thus, for every $k\ge 1$,
\be\label{eq::Zn_partition}
   Z_{k} \   {\buildrel {d}\over{\equiv}}\   \sum _{i=1}^{D}\wit{Z}_{k-1}^{(i)}
\ee
where $ \wit{Z}_{k-1}^{(i)}$ are i.i.d. copies of $\wit{Z}_{k-1}$.  Hence
\be\label{eq::sum_max*}
Y  =\lim_{k \to \infty}(\tau-2)^{k}\log(Z_{k}) =\lim_{k \to \infty}(\tau-2)^{k}\log\left(\sum_{i=1}^{D}\wit{Z}_{k-1}^{(i)}\right) \ee
We can bound the right hand side from both sides:
\be\label{eq::sum_bound}
(\tau-2)^{k}\log\left(\max_{i=1, \dots, D}\wit{Z}_{k-1}^{(i)}\right)\leq(\tau-2)^{k}\log\left(\sum_{i=1}^{D}Z_{k-1}^{(i)}\right)\leq(\tau-2)^{k}\log\left(\!D\!\cdot\! \max_{i=1,\dots,D}\wit{Z}_{k-1}^{(i)}\right).
\ee
Clearly $(\tau-2)^{k}\log D \overset{\Pv}{\to}0$, and by monotonicity we can exchange $\log$ and $\max$
and use Theorem \ref{thm::davies} for the convergence of $(\tau-2)^{k-1} \log (\widetilde Z_{k-1}^{(i)})$. Thus combining \eqref{eq::sum_max*} with \eqref{eq::sum_bound} yields
\be
Y=\lim_{k \to \infty}\max_{i=1,\dots, D} (\tau-2) \wit Y_{k-1}^{(i)}.
\ee
Exchanging the limit with the maximum finishes the proof. The second statement of the lemma can be proved analogously.
\end{proof}


\section{Mountain-climbing phase}\label{sc::climbup}
In this section we describe the mountain-climbing phase.
From now on we will concentrate on the growth of the \emph{red} (the faster) cluster, but the very same methods will later be used for blue as well. Thus, in this section we neglect the superscript $(r)$, and temporarily every quantity is belonging to the red cluster. We denote the set of red vertices at time $t$ by $\mathcal R_{t}$ and its size by $R_t$.
Since Proposition \ref{prop::coupling} only guarantees the coupling as long as the total number of explored vertices by red is at most $n^\vr$ for some $\vr>0$, let us first set some $\vr'<\vr (\tau-2)^2$ and define
\[ t(n^{\vr'})=\inf \{k:  Z_k  \ge n^{\vr'} \}.\]
Note that by Lemma \ref{lem::coupling-for-two}, and the fact that the total size of earlier generations are whp negligible compared to the last generation, $t(n^\rho)=T(n^{\vr'})$ whp.
Recall Definition \ref{def::limit-variables}, i.e.,
\be \label{def::Y_r*}\Yrn:=(\tau-2)^{t(n^{\vr'})}\log Z_{t(n^{\vr'})}. \ee
Note that $t(n^{\vr'})$ and thus $\Yrn$ is depending on $n$.
Then, an easy calculation yields that, with $\{x\}=x-\lfloor x \rfloor$,
\be\label{eq::an2} t(n^{\vr'}) = \frac{\log(\vr'/\Yrn) + \log\log n}{|\log(\tau-2)|}+ 1-a_n,\ee
where
 \be\label{eq::an} a_n=  \left\{ \frac{\log(\vr'/\Yrn) + \log\log n}{|\log(\tau-2)|}\right\}.  \ee
Note that $1-a_n$ is there to make the expression on the rhs of $t(n^{\vr'})$ equal to its upper integer part. Due to this effect, the last generation has a bit more vertices than $n^{\vr'}$, so let us introduce the notation $\vr''$ for the random exponent of the overshoot
\be \label{eq::rho_0} Z_{t(n^{\vr'})}= n^{ \vr' (\tau-2)^{a_n-1} }:= n^{\vr''},\ee
We get this expression by rearranging \eqref{def::Y_r*} and using the value $t(n^{\rho'})$ from \eqref{eq::an2}.
 The property $\vr'< \vr(\tau-2)^2$ guaranties that the coupling is still valid, i.e.\ we can also couple the \emph{degrees of vertices} in the $t(n^{\vr'})$th generation of the branching process to i.i.d.\ size biased degrees.

After time $t(n^{\vr'})$, we stop the coupling and focus on the graph: we start decomposing the graph to the following  nested sets of vertices, that we call \emph{layers}:
\be\label{def::Gamma_i} \Gamma_i:=\{ v: D_v>u_i \}, \ee
where $u_i$ is defined recursively by
\be\label{eq::ui_recursion} u_{i+1} =\left(\frac{u_{i}}{C\log n}\right)^{1/(\tau-2)}, \quad u_0:= \bigg(\frac{n^{\vr''}}{C\log n}\bigg)^{1/(\tau-2)} \ee
for a large enough constant $C>0$. We will see below that e.g. $C= 8/c_1$ is sufficient, where $c_1$ is from \eqref{eq::F}.
It is not hard to see that
\be\label{def::ui} u_{i}= n^{\vr'' (\tau-2)^{-(i+1)}} (C\log n)^{-e_i}\quad \mbox{with} \quad
 e_i =  \frac{1}{3-\tau}\bigg( \Big( \frac{1}{\tau-2}\Big)^{i+1}-1\bigg). \ee
Note that since $(\tau-2)^{-1}>1$, $u_i$ is growing, hence $\Gamma_0 \supset \Gamma_1 \supset \Gamma_2 \supset \dots$.
First we need to show that $Z_{t(n^{\vr'})}$ has a nonempty intersection with the initial layer $\Gamma_0$, and then we will build a path through the layers.
The following lemma is a general lemma about the maximum of i.i.d. power-law random variables. It guarantees that $\mathcal R_{t(n^{\vr'})}\cap \Gamma_0 \neq \varnothing$, and will also be repeatedly used to determine the maximum degree in a set of vertices:
\begin{lemma}\label{lem::maxdegree}
Let $X_i, \ i=1, \dots, m$ be i.i.d.\  random variables  with power-law tail exponent $\alpha$, i.e. the distribution function of $X_i$ satisfies \eqref{eq::F} with $\tau-1$ replaced by any $\alpha>0$.
Then there exists a random variable $\xi_\alpha$ with
\[ \max_{i=1,\dots, m} X_i/(m^{1/\alpha}) \buildrel {d}\over{\longrightarrow}  \xi_\alpha,  \]
and for $K>0$,
\be\label{eq::logwhp} \Pv\bigg(\max_{i=1,\dots, m} X_i < \Big(\frac{ m}{K\log n}\Big)^{1/\alpha} \bigg) \le \frac{1}{n^{c_1K}}, \ee
where $c_1$ arises from \eqref{eq::F}.
\end{lemma}
\begin{proof}Elementary calculation.
\end{proof}
Note that the distribution $F^\ast$ satisfy the condition of the lemma with $\al=\tau-2$, see \eqref{eq::davies_appliedto_G}. So, we can apply this lemma (specially \eqref{eq::logwhp}) in the following setting: the i.i.d.\ variables $X_i$ are the forward degrees  $(B_i)_{i=1, \dots, Z_{t(n^{\vr'})}}\sim F^\ast$ in the last generation of the branching process, thus $m:=Z_{t(n^{\vr'})}=n^{\vr''}$ and $\alpha=\tau-2$. Note that the bound we get when applying  \eqref{eq::logwhp} states that whp there is at least one vertex with degree at least $u_0$ (defined in \eqref{eq::ui_recursion}).
Hence,  we get that $\Gamma_0\cap \mathcal R_{t(n^{\vr'})}\neq \varnothing$ whp.

We will repeatedly use concentration of binomial random variables of the following form
\begin{lemma}[Concentration of binomial random variable]\label{lem::chernoff}
Let $X$ be a binomial random variable with parameters $n, p_n$.
Then
\be \ba \Pv(X \ge 2 \Ev[X]) &\le \exp\{ - \Ev[X]/8\}, \\
\Pv(X \le 1/2 \Ev[X])&\le \exp\{ - \Ev[X]/8\}. \ea \ee
\end{lemma}
\begin{proof} Follows from standard estimates, see e.g. \cite[Theorem 2.19]{H10} or \cite{Hage90} \end{proof}

In what follows, we will build a path from $\Gamma_0\cap \mathcal R_{t(n^{\vr'})}$ to the highest-degree vertices through successive layers $\Gamma_i$.
The following lemma guarantees the existence of such a path. Recall that $N(S)$ stands for the neighbors of the set $S$ in $\CMD.$
\begin{lemma}\label{lem::gamma_i_connectivity}With $u_i$ and $\Gamma_i$ defined as in \eqref{def::ui} and \eqref{def::Gamma_i}, for \emph{every} $v\in \Gamma_i$, whp there is a vertex $w\in \Gamma_{i+1}$, such that $(v,w)\in E(\CMD)$. Shortly, 
\[ \Gamma_i \subset N(\Gamma_{i+1}) \quad whp.\]
Furthermore, the previous statement can be applied \emph{repeatedly} to build a path from $\Gamma_0$ to $\Gamma_{i}$ as long as $u_i=o(n^{1/(\tau-1)})$, which is
\be\label{eq::ibound}i < -\frac{\log((\tau-1)\vr'')}{|\log(\tau-2)|}.\ee
\end{lemma}
\begin{proof}
Let us denote the total number of half-edges in $\Gamma_i$ by $\mathcal S_i$. Then, since the degrees are i.i.d.\!\,, we have $|\Gamma_{i+1}| \sim {\sf Bin}(n, 1-F(u_{i+1}))$, and each vertex $w\in \Gamma_{i+1}$ has degree at least $u_{i+1}$. Thus by Lemma \ref{lem::chernoff},
\be\label{eq::si-gammai} \Pv\left( \mathcal S_{i+1} < \frac{u_{i+1} n \left[1-F(u_{i+1})\right]}{2}\right) \le \exp\left\{ -\frac{ n \left[1-F(u_{i+1})\right]}{8} \right\},
\ee
Recall that $\CL_n$ denotes the total number of half-edges in the graph. Then, the probability that there is a vertex $v\in \Gamma_i$ not connected to $\Gamma_{i+1}$ can be bounded from above by
\be\ba \label{eq::error_gamma_i}\Pv_n&\left( \exists v\in \Gamma_i, v \nleftrightarrow \Gamma_{i+1}\right)\le
|\Gamma_i| \left( 1- \frac{\mathcal S_{i+1}}{\CL_n}\right)^{u_i/2} \\
&\le n \exp\!\Big\{\!-\frac{ u_i u_{i+1} n \left[1-F(u_{i+1})\right]}{8 n \Ev[D]} \Big\}+ \exp\{ - n \left[1-F(u_{i+1})\right]/8 \},
\ea \ee
where we recall that $\Pv_n(\cdot):=\Pv(\cdot|D_1, \dots,D_n)$. We have used that $\CL_n< 2\Ev[D]n$ whp by the Law of Large Numbers, $|\Gamma_i|<n$, and the estimate $S_{i+1}$ in \eqref{eq::si-gammai}. The factor $1/2$ in the exponent $u_i/2$ comes from the worst-case scenario estimate when we connect all  the first $u_i/2$ half-edges back to $v$. Similar calculations (with indices of $u_i$ and $u_{i+1}$ exchanged) are worked out in more detail in  \cite[Volume II., Chapter 5]{H10}.
Then, using the defining recursion \eqref{eq::ui_recursion}, it is easy to see that
\[ u_i u_{i+1} \left[1-F(u_{i+1})\right] = C \log n (u_{i+1})^{\tau-1} \left[1-F(u_{i+1})\right] \ge C c_1 \log n,  \]
that is, the error term in \eqref{eq::error_gamma_i} is bounded by
\[ \ve_i := \exp\left\{\Big(1- \frac{C c_1}{8\Ev[D]}\Big)\log n \right\}+ \exp\{ -n \left[1-F(u_{i+1})\right]/8 \}. \]
The assertion of the lemma follows if $\ve_i$ small: the first term is small when picking $C$ large enough. For the second term
we need $n \left[1-F(u_{i+1})\right]>O(1)$, which exactly translates to the condition $u_i=o(n^{1/(\tau-1)})$ and to \eqref{eq::ibound} using \eqref{def::ui}. Note that as long as \eqref{eq::ibound} is satisfied, even  $\sum_{j=1}^i \ve_i = o(n^{-1})$. This means that we can apply the lemma consecutively for the layers $(\Gamma_i)$'s and build a path $(v_0,v_1,\dots, v_i)$ such that for all $0\le j \le i $ we have $v_j\in \Gamma_j$ and $(v_j, v_{j+1})\in E(\CMD)$ whp, as long as $i$ satisfies \eqref{eq::ibound}. This finishes the proof of the second statement of the lemma.
 \end{proof}
 \begin{figure}\label{fig::mountain}
\includegraphics[width=0.5\textwidth]{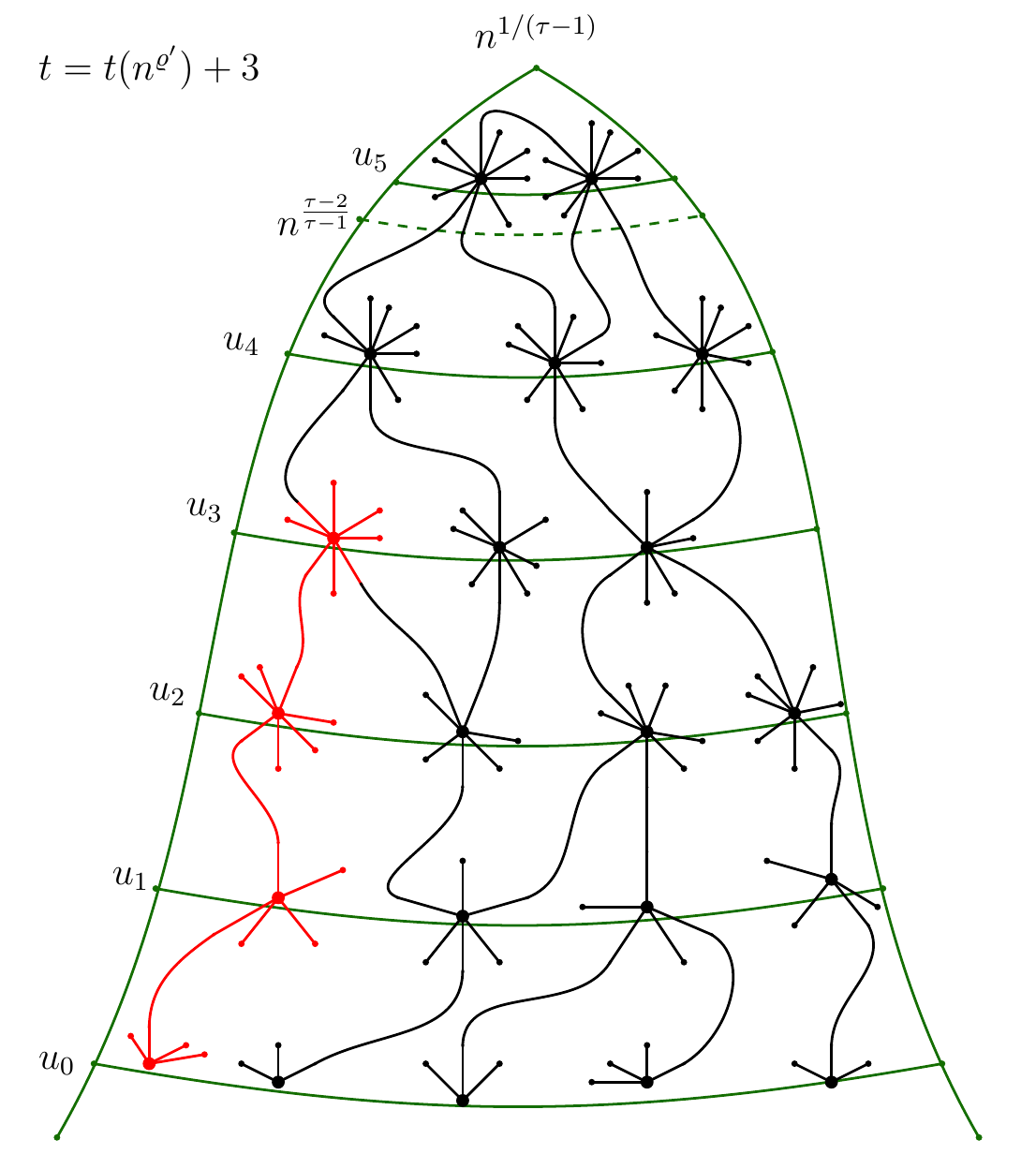}
\caption{An illustration of the layers and the mountain climbing phase at time $t(n^{\vr'})+3$. Disclaimer: the degrees on the picture are only an illustration.}
\end{figure}
With Lemma \ref{lem::gamma_i_connectivity} in hand we can determine how long it takes to climb up through the layers $\Gamma_i$ to the highest-degree vertices. Lemma \ref{lem::maxdegree} with $X_i=D_i\sim F$, $\alpha=\tau-1$ shows that the maximal degree in $\CMD$ is of order $n^{1/(\tau-1)}$. We write $i^\ast$ for the last index when $\Gamma_{i}$ is whp nonempty, i.e.,
\be\label{def::i*}i^\ast:=\inf \{ i: u_i \le n^{1/(\tau-1)} < u_{i+1}\}. \ee
An easy calculation using \eqref{def::ui} shows that
\be\label{eq::value_i*} i^\ast= -1+ \frac{-\log ((\tau-1)\vr'')}{|\log(\tau-2)|}-b_n,
\quad \mbox{ with }\quad
 b_n=\left\{ \frac{-\log ((\tau-1)\vr'')}{|\log(\tau-2)|}\right\}. \ee
Note that $i^\ast$ satisfies \eqref{eq::ibound}, thus all the error terms up to this point stay small.
Using the value of the overshoot exponent $\vr''$ in \eqref{eq::rho_0} and then the value $a_n$ in \eqref{eq::an}, plus the fact that $\{x - 1+\{y\}\}=\{x+y\}$,  we get that
\be\label{eq::bn} b_n = \left\{ \frac{-\log ((\tau-1)\vr')}{|\log(\tau-2)|}+ a_n-1\right\} = \left\{ \frac{-\log ((\tau-1)\Yrn)+\log \log n}{|\log(\tau-2)|}\right\}.\ee
From \eqref{def::ui} one can easily calculate that
\be\label{eq::ui*} \ba u_{i^\ast}&=n^{\frac{(\tau-2)^{b_n}}{\tau-1}} (C \log n)^{-e_{i^\ast}},  \quad \mbox{ with} \\
e_{i^\ast}&=\frac{1}{3-\tau}\left(\frac{(\tau-2)^{b_n}}{(\tau-1)\vr''} -1\right)\le\frac{1}{(3-\tau)}\left(\frac{1}{(\tau-1)\vr''}-1\right). \ea \ee
We will repeatedly need the total time to reach the top, so let us introduce the notation
\be\label{eq::k*+i*}  T_r:=t(n^{\vr'})+i^\ast=\frac{\log\log n-\log \left((\tau-1) \Yrn\right)}{|\log(\tau-2)|} -1-b_n,\ee
which only depends  on $\vr'$ via the approximating $\Yrn$, and $b_n$ is exactly the fractional part of the expression on the rhs of $T_r$. Since also $\Yrn\to Y_r$ irrespective of the choice of $\vr'$, this establishes that the choice of $\vr'$ is not relevant in the proof.

\section{Crossing the peak of the mountain}\label{sc::peak}
Next we investigate what happens when the path through the layers reaches the highest degree vertices.
We have just seen that the exponent of $n$ in $u_{i^\ast}$ is $\frac{(\tau-2)^{b_n}}{\tau-1} \in \left( \frac{\tau-2}{\tau-1}, \frac{1}{\tau-1}\right)$.
Recall that the maximum degree in the graph has exponent $\frac{1}{\tau-1}$ whp, i.e. $\Gamma_{i^\ast+1}= \varnothing$ whp, meaning the path can not jump `up' one more step.
 On the other hand, we can make use of the following lemma from \cite[Volume II., Chapter 5]{H10}:
 \begin{lemma}\label{lem::direct_connect} Consider two sets of vertices $A$ and $B$. If for the number of half-edges $\mathcal S_A=o(n)$ and $\mathcal S_B$ satisfy
 \[ \frac{\mathcal S_A \mathcal S_B}{n}> h(n),\]
 for some function $h(n)$, then conditioned on the degree sequence with $\CL_n\le 2\Ev[D] n$, the probability that the two sets are not directly connected can be bounded from above by
 \[ \Pv_n(A\nleftrightarrow B)< \e^{-\tfrac{h(n)}{4\Ev[D]}}.\]
  \end{lemma}
 \begin{proof}
 When pairing the half-edges coming out from $A$, the probability that the $i$-th one paired is not directly connected to a half-edge in $B$ is $(1-\mathcal S_B/(\CL_n-2i-1))$. Thus,
 \[ \ba \Pv_n(A \nleftrightarrow B) &= \prod_{i=0}^{\lceil\mathcal S_A/2\rceil} \left(1-\frac{\mathcal S_B}{\CL_n-2i-1}\right) \le \exp\left\{ -\frac{\mathcal S_A \mathcal S_B}{2(\CL_n-2\mathcal S_A)}\right\}\\
  &\le \exp\left\{-\frac{h(n)}{4 \Ev[D]}\right\}.  \ea \]
 The product only goes until $\mathcal S_A/2-1$, since in the worst case scenario the first $\lfloor \mathcal S_A/2 \rfloor$ half-edges are all paired back to another half-edge in $A$, thus the last $\lfloor\mathcal S_A/2\rfloor$ half-edges are not used anymore. In both cases, we can pair at least $\lceil \CS_A/2 \rceil$ many half-edges.
 \end{proof}
 Let us introduce
 \be\label{def::alpha} \alpha:=1-\frac{(\tau-2)^{b_n}}{\tau-1}, \quad
 \beta:=1+\frac{1}{(3-\tau)}\left(\frac{1}{(\tau-1)\vr''}-1\right), \ee
and
\be\label{eq::wideu1} \widetilde u_1:= (C \log n) n/u_{i^\ast}=n^{\alpha} (C\log n)^\beta,\ee
and the following layer:
\be\label{def::witgamma}\widetilde\Gamma_1:=\{v\in \CMD, D_v > \wit u_1\}.\ee
The next lemma helps us describe how the process goes through the highest-degree vertices:
\begin{lemma}\label{lem::cross_the_peak}
All the vertices in $\wit \Gamma_1$ are occupied by red at time $T_r+1$, i.e.,
\be\label{eq::slopestart}  \widetilde\Gamma_1\subset \mathcal R_{T_r+1} \quad \mbox{whp.}\ee
\end{lemma}
\begin{proof}
 By Lemma \ref{lem::gamma_i_connectivity}, there is a blue path up to $\Gamma_{i^\ast}$, and hence, blue is occupying some vertices in layer $\Gamma_{i^\ast}$ at time $T_r$. Hence, $\mathcal R_{\sss{T_r}}\cap \Gamma_{i^\ast}\neq \varnothing$, and we have at least one vertex $v_{i^\ast}$ in $\mathcal R_{\sss{T_r}}$ for which the degree is at least $u_{i^\ast}$, see \eqref{eq::ui*}. We claim that this vertex is whp connected to \emph{every} vertex in $\wit \Gamma_1$. To see this, let us set $A:=\{v_{i^\ast}\}$ and $B:=\{w\}$, that is, any single vertex in $\wit\Gamma_1$ with degree at least $(C\log n)n/u_{i^\ast}$. Then apply Lemma \ref{lem::direct_connect}  with this setting to see that $v_{i^\ast}$ is whp connected to $w$. Further, note that $\CS_A \CS_B /n=C\log n$ by the definition of $\wit u_1$. Hence, using the error bound in Lemma \ref{lem::direct_connect} and a union bound,
\be\label{eq::error-peak} \Pv_n( \exists w \in \wit\Gamma_1, (v_{i^\ast},  w) \notin E(\CMD) | \wit\Gamma_1) \le |\wit\Gamma_1| \frac{1}{n^{C/4\Ev[D]}}. \ee
Clearly $|\wit \Gamma_1|<n$:  picking a large enough $C$, we see that the error probability tends to zero.
Calculating  $C\log n \cdot n/u_{i^\ast}$ yields the formula for $\wit u_1$.
\end{proof}

 It is important to note that vertices with degree larger than $\wit u_1$ do whp exist in $\CMD$ by Lemma \ref{lem::maxdegree}. Moreover, $i^\ast$ is the first index when we can apply Lemma \ref{lem::direct_connect}, since for all smaller values $i< i^\ast$, there are whp no vertices with degree at least $n/u_i$ by Lemma \ref{lem::maxdegree}.

This completes the \emph{crossing the peak of the mountain} phase.

\section{red avalanche from the peak and the blue climber}\label{sc::slopedown}
Using the value $\widetilde u_1$ in \eqref{eq::wideu1}, let us again recursively define
\be \label{eq::wideui_recursion} \widetilde u_{\ell+1}=C \log n\!\cdot\!(\widetilde u_{\ell})^{\tau-2}. \ee
and also the increasing sequence of sets
\[ \widetilde \Gamma_\ell:=\{ v: D_v > \widetilde u_\ell\},\]
i.e., now $\wit \Gamma_1 \subset \wit \Gamma_2 \subset \dots$ holds.
Since \eqref{eq::wideui_recursion} is the very same as the recursion in \eqref{eq::ui_recursion}
with indices exchanged, we can apply Lemma \ref{lem::gamma_i_connectivity} to $\big(\widetilde \Gamma_\ell\big)_{\ell\ge 1}$, now yielding that for any $\ve> 0$, for all $\ell< (1-\ve)\log\log n / |\log(\tau-2)|$
\be\label{eq::widetildegamm} \widetilde\Gamma_{\ell+1}\subset N(\widetilde \Gamma_\ell)\quad\mbox{ and }\quad \widetilde \Gamma_{\ell} \subset \mathcal R_{\sss{T}_r+\ell} \quad \mbox{whp.} \ee
This means that in the `sloping down' phase, whp red occupies \emph{all} vertices in $\wit\Gamma_\ell$ at time $T_r+\ell$.
Solving the recursion \eqref{eq::wideui_recursion} yields that
\be\label{eq::ul} \widetilde u_\ell= n^{\alpha(\tau-2)^{\ell-1}} (C\log n)^{\beta (\tau-2)^{\ell -1} + \frac{1}{3-\tau}\left(1-(\tau-2)^{\ell-1}\right)},\ee
where $\alpha$ and $\beta$ were defined in \eqref{def::alpha}. Note that the exponent of $C\log n$ stays bounded even when $\ell\to \infty$.
Hence this procedure can be continued even to reach lower degree vertices, for every fixed $\ve>0$ up until $\ell< (1-\ve) \log\log n / |\log(\tau-2)|$.

In what follows, we determine the point where red and blue meet. More precisely, we calculate the value $\ell$ such that during the time $T_r+\ell$,
the maximum degree vertex in the cluster of blue is of the same order as $\widetilde u_\ell$.
Since at time $T_r+\ell$, red occupies whp almost every vertex with degree at least $\wit u_\ell$,
the growing cluster of blue bumps into the occupied vertices and cannot spread to higher-degree vertices anymore.

The following proposition about the maximal degree of blue is our main building block for the proof of Theorem \ref{thm::maxdegreeblue}:
\begin{proposition}\label{claim::maxdegblue}
Let us denote by $D_{\max}^{\sss{(b,n)}}(t)$ the forward  degree in $\CMD$ of the maximal degree vertex in the blue
cluster at time $t$. Then, at time $T_r+t$ and for any real $1\le t \le O^\diamond(1)$,
\be \label{eq::maxdegreeatTr+l}D_{\max}^{\sss{(b,n)}}(T_r+t)=\exp \bigg\{ \Ybn \left(\frac{1}{\tau-2}\right)^{\left\lfloor\frac{T_r+t}{\la}\right\rfloor+1}(1+o_\Pv(1))\bigg\},  \ee
as long as $t$ is so that the quantity on the rhs is less than $\wit u_{[t]}$,  and
where $\Ybn$ is defined in Definition \ref{def::Y}.
\end{proposition}
\noindent Before the proof we need some important definitions that will be used also outside the proof.
Similarly as in \eqref{def::Gamma_i}, let us define:
\be\begin{array}{ll}\label{eq::uibar}
   \widehat u_0^{\sss(b)}:= (Z^{\sss{(b)}}_{\lfloor t(n^{\vr'})/\la\rfloor}\cdot C \log n)^{1/(\tau-2)}, &\quad u_0^{\sss(b)}:= (Z^{\sss{(b)}}_{\lfloor t(n^{\vr'})/\la\rfloor}/C\log n)^{1/(\tau-2)} \\

  \widehat u_{i+1}^{\sss{(b)}}:= (\widehat u_i^{\sss{(b)}}\cdot C \log n)^{1/(\tau-2)}, &\quad  u_{i+1}^{\sss{(b)}}: = \left(u_i^{\sss{(b)}}/ C \log n\right)^{1/(\tau-2)} \\

  \widehat\Gamma_i^{\sss{(b)}}:=\{v \in \CMD: d_v \ge \widehat u_i^{\sss{(b)}}\}, &\quad \Gamma_i^{\sss{(b)}}:=\{v \in \CMD: d_v \ge  u_i^{\sss{(b)}}\}. \\
\end{array}\ee
Note that $\Gamma_i^{\sss{(b)}}$ grows exactly as $\Gamma_i$ while $\widehat\Gamma_i^{\sss{(b)}}$ grows faster: there is always an extra $(C\log n)^2$ factor causing an initial `gap' of order $(\log n)^2$ between $u_0^{(b)}, \widehat u_0^{(b)}$ and `opening up' as $i$ gets larger.

Further, let us say that a quantity  $Q\asymp O^\diamond(x)$ for $x\in R^+$, if $Q$ satisfies
\be\label{def::idiamond} Q= x \frac{\log\log n}{|\log (\tau-2)|} + O_\Pv(1).\ee
We will see below in \eqref{eq::tc} that blue cannot make more jumps than $O^\diamond\big(\tfrac{\la-1}{\la(\la+1)} \big)$ in its climbing phase. In order to show Proposition \ref{claim::maxdegblue}, we need a lower and an upper bound on the maximal degree in each step. The next lemma handles the upper bound, but first some definitions.

 We say that a sequence of vertices and half-edges $(\pi_0, s_0, t_1, \pi_1, s_1, t_2,  \dots,  t_k, \pi_k)$ forms a path in $\CMD$, if for all $0< i\le k$, the half edges $s_{i-1}, t_i$ form an edge between $\pi_{i-1},\pi_i$.
Let us denote the vertices in a path starting from a half-edge in $Z^{\sss{(b)}}_{\lfloor t(n^{\vr'})/\la\rfloor}$ by $\pi_0, \pi_1, \dots $. We say that a path is \emph{good} if $\deg(\pi_i)\le\widehat u_i^{\sss{(b)}}$ holds for every $i$. Otherwise we call it \emph{bad}. We decompose the set of bad paths in terms of where they turn bad, i.e.\  we say that  a bad path is belonging to  $\CB ad \CP_k$ if it turns bad at the $k$th step:
\[ \ba  \CB ad\CP_k := &\{ (\pi_0, s_0, t_1, \pi_1, s_1 \dots, t_k, \pi_k) \text{ is a path, } \\
 &\quad \pi_0\!\in\!\CB_{ t(n^{\vr'}) },\   \deg(\pi_i)\!\le\! \widehat u_i^{\sss{(b)}} \ \forall i\le  k-1,\ \deg(\pi_k)\!\ge\!\widehat u_k^{\sss{(b)}}   \}.\ea\]
The following lemma tells us that the probability of having a bad path is tending to zero:

\begin{lemma}\label{lem::badpaths}
Fix $0<x\le (\la-1)/\la(\la+1)$. Then for any $k_0 \le O^\diamond(x)$, the following bound on the probability of having any bad paths holds:
\be \Pv( \exists k\le k_0:  \CB ad\CP_k \neq \varnothing) \le \frac{2}{C\log n}.\ee
\end{lemma}
\begin{proof}
The proof uses path counting methods that we describe in the appendix. Hence we put the proof there.
\end{proof}

 \begin{proof}[Proof of Proposition \ref{claim::maxdegblue}]
 Since the method for the lower bound is very much the same as for red, plus we will need a more detailed analysis of this process below in Lemma \ref{lem::numberofverticesinGamma}, we just sketch the proof (read further to the proof of Lemma \ref{lem::numberofverticesinGamma} for more details).
 First, Lemma \ref{prop::coupling} ensures that we can couple both the blue and the red cluster to their BP approximation until time $t(n^{\vr'})$ given in \eqref{eq::an}.
Since it takes $\la>1$ unit of time to cover an edge for blue, the number of generations covered by the branching process approximation $Z^{\sss{(b)}}$ of blue is $\lfloor t(n^{\vr'})/\la\rfloor$.
The size of the last generation in the blue BP is thus
\be\label{eq::bluestart}Z_{\lfloor t(n^{\vr'})/\la\rfloor}^{\sss{(b)}}= \exp\bigg\{ \Ybn \left(\frac{1}{\tau-2}\right)^{\lfloor t(n^{\vr'})/\la\rfloor}\bigg\}.\ee
We start applying the method in the \emph{Mountain climbing} phase for blue from this point on.
With the same technique as we used to show that $\CR_{t(n^\vr)}\cap\Gamma_0\neq \varnothing$ using Lemma \ref{lem::maxdegree},
we define $u_0^{\sss{(b)}}$ and a corresponding layer $\Gamma_0^{\sss{(b)}}$ in \eqref{eq::uibar} and can show that $\CB_{t(n^{\vr'})}\cap \Gamma_0^{\sss{(b)}}\neq\varnothing$.
Then, we define the nested sequence of layers  $\left(\Gamma^{\sss{(b)}}_i\right)_{i\ge 0}$ with $u^{\sss{(b)}}_i$-s such that at time $\la ( \lfloor t(n^{\vr'})/\la \rfloor + i) $, blue occupies at least $1$ vertex in $\Gamma^{\sss{(b)}}_i$.
by Lemma \ref{lem::gamma_i_connectivity}.
Note that from $\Gamma^{\sss{(b)}}_i$ to $\Gamma^{\sss{(b)}}_{i+1}$, the exponent of $1/(\tau-2)$ on the right hand side of \eqref{eq::bluestart} is increased by on. Further, there is an extra $+1$
in the exponent for the initial maximization of the degrees in $u_0^{\sss{(b)}}$ similarly as in \eqref{eq::ui_recursion}.

The total number of layers $\Gamma_i^{\sss{(b)}}$ jumped by blue at time $T_r+t$ is then $\lfloor(T_r+t)/\la\rfloor-\lfloor t(n^{\vr'})/\la\rfloor$, that, combined with \eqref{eq::bluestart}, yields formula \eqref{eq::maxdegreeatTr+l}.

We still need to check that the term arising from $C\log n$ in the definition of $u_i^{\sss{(b)}}$'s can be put in a $(1+o_\Pv(1))$ factor in the exponent. For this, write $Z^{\sss{(b)}}_{\lfloor t(n^{\vr'})/\la\rfloor}:=m$, then
\[ u_i^{\sss{(b)}}= m^{(\tau-2)^{-(i+1)}}( C\log n)^{(1-(\tau-2)^{-(i+1)})/(3-\tau)}\]
and the last layer before time $T_r+t$ is reached after climbing $i=\lfloor (T_r+t)/\la\rfloor-\lfloor t(n^{\vr'})/\la \rfloor$ many $\Gamma_i^{\sss{(b)}}$ layers, so by \eqref{eq::an2} and \eqref{eq::k*+i*} we calculate
\be\label{eq::calc} \lfloor (T_r+t)/\la\rfloor-\lfloor t(n^{\vr'})/\la \rfloor=(i^*+t)/\la (1+o(1)) \ee
Thus,  if $t\le O^\diamond(1)$, when taking the logarithm, then the term corresponding to $(C\log n)^{(\tau-2)^{-i}/(3-\tau)}$ in $\log (u_{i}^{\sss{(b)}})$ at time $T_r+t$ is of order $(\tau-2)^{-t/\la} = o((\tau-2)^{-(T_r+t)/\la})$
 Hence, these terms vanish when taking out $(\tau-2)^{-\lfloor T_r+t/\la \rfloor}$ in the statement of the lemma.
We will see below in \eqref{eq::tc} that in fact the procedure stops at $t=O^\diamond(\frac{\la-1}{\la (\la+1)})$ since after that red will block the growth of blue entirely.

For the upper bound, according to  Lemma \ref{lem::badpaths}, whp  $\{\CB ad\CP_k=\varnothing \ \forall k\le k_0\}$, and on this event the maximal degree of blue at time $\la \lfloor t(n^{\rho'})/\la\rfloor + \la i$ is at most $\widehat u_i^{\sss{(b)}}$. Since the exponent of $C \log n$ in $\widehat u_i^{\sss{(b)}}$ is  exactly  $(-1)$ times the exponent of $C\log n$ in $u_i^{\sss{(b)}}$,  these terms can also be put in the $(1+o_{\Pv}(1))$ factor by the same argument as for the lower bound.
\end{proof}

We will later need more information than the maximal degree of blue, namely, we also need an upper bound on how many vertices blue occupies in each layer. For this, first, we will show that the probability that blue goes above $\widehat u_i^{\sss{(b)}}$ at time $\la \lfloor t(n^{\vr'})/\la \rfloor+\la i$ is small, then we estimate the number of vertices blue paints in each layer based on this bound. We carry these out in a claim and a lemma.

Let us denote the total number of half-edges attached to vertices with degree larger than $y_n$ by $\CE_{\ge y_n}$.
Then, we have the following tail bound for $\CE_{\ge y_n}$:
\begin{claim}\label{claim::Sbound} For a sequence $y=y_n$, and a large enough constant $C<\infty$,  and for some constant $0<c<\infty$,
\be\label{eq::Sbound} \Pv ( \CE_{\ge y_n} \ge C\!\cdot\! n\!\cdot\! y_n^{2-\tau}) \le \exp\{ - c \!\cdot\! n \!\cdot\! y_n^{1-\tau} \}. \ee
\end{claim}
\begin{proof}
Since the degrees are i.i.d.\  in $\CMD$, we write
\[ \CE_{\ge y_n} \  {\buildrel {d}\over{=}}\  \sum_{i=1}^{n} D_i \ind_{\{ D_i \ge y_n\}} \le  \sum_{i=1}^{n}  \sum_{k=1}^\infty  2^k y_n \ind{\{ 2^{k-1}y_n \le D_i < 2^k y_n\}}.   \]
Now, exchanging sums,
\[ \CE_{\ge y_n} = \sum_{k=1}^\infty 2^k y_n \left(\sum_{i=1}^{n}  \ind{\{ 2^{k-1}y_n \le D_i < 2^k y_n\}}\right):=  \sum_{k=1}^\infty 2^k y_n X_k^{\sss{(n)}} .\]
The variables $(X_k^{\sss{(n)}})_{k\ge 1}$ form a multinomial random variable, each marginal is a binomial, and hence large deviation type concentration bounds can be used. Lemma \ref{lem::chernoff} combined with a union bound yields  \be\label{eq::chernoff-for-s} \Pv(\exists k\ge 1:  X_k^{\sss{(n)}} \ge 2 \Ev[X_k^{\sss{(n)}}] ) \le \sum_{k=1}^\infty \exp\{ - \frac{1}{8} \Ev[X_k^{\sss{(n)}}]\}. \ee
Now, by \eqref{eq::F},
\[\Ev[X_k^{\sss{(n)}}]\le 2 C_1\!\cdot\!  n\!\cdot\!  y_n^{1-\tau} \cdot 2^{(1-\tau)(k-1)}.\]
 Note that $2^{1-\tau}<1$, hence summing up terms in $k$ on the right hand side of \eqref{eq::chernoff-for-s}, we get that for an $c\le C_1 /12$, the error term is bounded by
 \[ \Pv(\exists k\ge 1:  X_k^{\sss{(n)}} \ge 2 \Ev[X_k^{\sss{(n)}}] ) \le \exp\{ -  c \!\cdot\! n \!\cdot\! y_n^{1-\tau}\}.\]
Since the event $\{\forall k\ge 1:  X_k^{\sss{(n)}} \le 2 \Ev[X_k^{\sss{(n)}}] \}$ also implies that
\[ \CE_{\ge y_n} \le \sum_{k=1}^\infty  2^k y_n 2 \Ev[X_k^{\sss{(n)}}] \le C\!\cdot\! n\!\cdot\! y_n^{2-\tau},\]
combining this fact with the previous error estimate finishes the proof.
\end{proof}

Let us  denote the set and number of blue vertices in the $i$th layer $\Gamma_i^{\sss{(b)}}$ right at the time when blue reaches it by
\be\label{def::Ai} \CA_i:=\CB_{\la \lfloor t(n^{\vr'})/\la \rfloor+\la i}\cap\Gamma_i^{\sss{(b)}}, \qquad A_i:=|\CA_i|. \ee

\begin{lemma}\label{lem::numberofverticesinGamma}
Let  $k_0\le O^\diamond((\la-1)/\la(\la+1))$.  On the event $\{\CB ad\CP_k=\varnothing \ \forall k\le k_0 \}$, whp for all $ i\le O^\diamond(\tfrac{\la-1}{\la(\la+1)})$,
 \be \label{eq::aifinal} \ A_i\le\exp\left\{ \log (C\log n) \cdot \frac{2 (\tau-2)^{-i}}{(3-\tau)^2} \right\}. \ee
 Hence, for $x\le (\la-1)/\la(\la+1)$, for some constant $K_2$, whp
 \be \log A_{O^\diamond(x)}\le K_2 (\log n)^{x} \log \log n \cdot (1+o(1)).  \ee
\end{lemma}
\begin{proof}  First, Lemma \ref{lem::badpaths} guarantees that $\{\CB ad\CP_k=\varnothing \ \forall k\le k_0\}$ holds whp, and on this event $\widehat u_{i}^{\sss{(b)}}$ serves as  an upper bound on the maximal degree of blue at time $\la \lfloor t(n^{\rho'})/\la \rfloor +\la i$. So, we can give a recursive upper bound on the number of vertices reached by blue in a given layer $\Gamma_i^{\sss{(b)}}$ by using $u_i^{\sss{(b)}}$ as a lower and $\widehat u_{i}^{\sss{(b)}}$ as an upper bound on the degrees.

Let us condition on the number of blue vertices $A_i$ in layer $\Gamma_i^{\sss{(b)}}$.
Then, we have at most $A_i$ half-edges in $\Gamma_i^{\sss{(b)}}\cap \mathcal B_{\la \lfloor t(n^{\vr'})/\la\rfloor+\la i}$, with degree at most $\widehat u_{i}^{\sss{(b)}}$, hence  we get the stochastic domination
\be\label{eq::bindom} A_{i+1}=|N(\CA_i)\cap\Gamma_{i+1}^{\sss{(b)}}|\  {\buildrel {d}\over{ \le }}\  {\sf Bin}\left(A_i \widehat u_{i}^{\sss{(b)}}, \frac{\CE_{\ge u_{i+1}^{\sss{(b)}}}}{\CL_n(1+o(1))}\right).\ee
random variable.
Then, we can use Claim \ref{claim::Sbound} with $y= u_{i+1}^{\sss{(b)}}$ to bound $\CS_{\ge u_{i+1}^{\sss{(b)}}}$.
 Hence, on the event that $\Ev[D] n/2<\CL_n<2 \Ev[D] n$ holds, by Lemma \ref{lem::chernoff} applied on the binomial variable in \eqref{eq::bindom},
\be\label{eq::chernofferror}\Pv\left( A_{i+1}> \frac{8 C_1}{\Ev[D] (2-\tau)} A_i  \widehat u_i^{\sss{(b)}} (u_{i+1}^{\sss{(b)}})^{2-\tau}\Big| A_i\right) \le    \exp\left\{ - A_i  u_i^{\sss{(b)}} (u_{i+1}^{\sss{(b)}})^{2-\tau} \frac{C_1}  {4 c_1 \Ev[D] }    \right\}. \ee
Thus, with the error probability in the previous display, whp
\be\label{eq::airec}  A_{i+1}\le A_i \widehat u_i^{\sss{(b)}} (u_{i+1}^{\sss{(b)}})^{2-\tau}\frac{8 C_1}{\Ev[D] (2-\tau)}< A_i \frac{ \widehat u_i^{\sss{(b)}}}{u_{i}^{\sss{(b)}}} C'\log n\ee
with $C':=8 C_1C/(\Ev[D] (2-\tau))$.
Using \eqref{eq::uibar}, we get the recursion $A_{i+1}\le A_i C'\log n \cdot \widehat u_{i}^{\sss{(b)}}/u_i^{\sss{(b)}} $. The bound in \eqref{eq::aifinal} is nothing but
\be\label{eq::Aibound-with-ui/ui} A_i \le A_0 \prod_{j=0}^{i-1} \left(C'\log n \cdot \widehat u_{j}^{\sss{(b)}}/u_j^{\sss{(b)}}\right).\ee
Solving the recursions for $u_{i}^{\sss{(b)}}$ and $\widehat u_{i}^{\sss{(b)}}$ in \eqref{eq::uibar} we get that
\be\label{eq::uiquotient} \widehat u_i^{\sss{(b)}}/u_i^{\sss{(b)}}= (C\log n)^{2\left((\tau-2)^{-(i+1)}-1 \right)/(3-\tau)}. \ee
 Initially $A_0 \le 2C \log n$ whp. This can be seen as follows: by the coupling of the exploration process to the branching process in Section \ref{sc::BP}, the last generation has size $Z_{\lfloor t(n^{\vr'})/\la\rfloor}^{\sss(b)}$, and the degrees are i.i.d. of distribution $D^\star$.  Hence, the number of vertices in this last generation that have degree at least $u_0^{\sss{(b)}}$ has distribution $A_0\sim \mathrm{Bin}( Z^{\sss{(b)}}_{t(n^{\vr'})}, \Pv(D^\star > u_0^{\sss{(b)}}))$.  Note that by the choice of $u_0^{\sss{(b)}}$, $\Ev[A_0]\le C C_1 \log n$, and the Lemma \ref{lem::chernoff} implies that  $A_0< 2 C C_1 \log n$ holds with probability at least $\exp\{ C C_1 \log n/8 \}$, which is small when $C$ is large enough.

 Using \eqref{eq::uiquotient} and evaluating \eqref{eq::Aibound-with-ui/ui} finishes the proof of \eqref{eq::aifinal}. We dropped some negative terms in the exponent in \eqref{eq::aifinal}.
If we set $i\le O^\diamond(x)= x \log \log n / |\log(\tau-2)|+ O_\Pv(1)$,
then by picking a large enough $C$, the error terms are $o(n^{-A_i})$ in \eqref{eq::chernofferror}.
Thus we can also iterate the argument up to time $O^\diamond(x)$ to see that at time $\la \lfloor t(n^{\vr'})/\la\rfloor+\la O^\diamond(x)$,  in $\Gamma_{O^\diamond(x)}^{\sss{(b)}}$,
the number of vertices blue occupies is bounded by the right hand side of \eqref{eq::aifinal}.
\end{proof}
\section{At the collision time - the maximal degree of blue}\label{sc::meetingtime}
\subsection{The maximum degree of blue}
In this section we analyse how red and blue collide and prove Theorem \ref{thm::maxdegreeblue}, i.e., we determine the degree of the maximum degree vertex that blue ever occupies. There are two different processes running at time $T_r+\ell$: the red process is in its \emph{avalanche} phase and occupies every vertex that has degree higher than $\wit u_\ell$, while the slower blue process is still in its \emph{mountain-climbing} phase and keeps increasing its maximal degree (but does not occupy all vertices of this degree). To obtain a good approximation when the two processes meet let us first neglect integer part issues and
determine the collision time of the two processes (the red avalanche and the blue climber) in continuous time.
 Neglecting the terms containing $C\log n$, we can compare the maximal degree
 \eqref{eq::maxdegreeatTr+l} to $\wit u_\ell$ in \eqref{eq::ul} and define $t_c$ as the solution of the equation
\be\label{def::intersection} \exp \left\{ \Ybn (\tau-2)^{-\frac{T_r+t_c}{\la}-1}\right\} = \exp\left\{ \log n \cdot \alpha\cdot(\tau-2)^{t_c-1}\right\}. \ee
Note that the left-hand side is approximately equal to the maximum degree $D_{\max}^{\sss{(b,n)}}(T_r+\ell)$ of blue,
 while the right-hand side is the approximate value of $\wit u_\ell$. Thus $t_c$ is the (non-integer valued) time left till the intersection of these two functions after time $T_r$.

\begin{figure}
\centering
\subfigure[blue jumps first, $d(t_c)<1$ ]{\label{fig::bluedsmall}
\includegraphics[keepaspectratio,width=5cm]{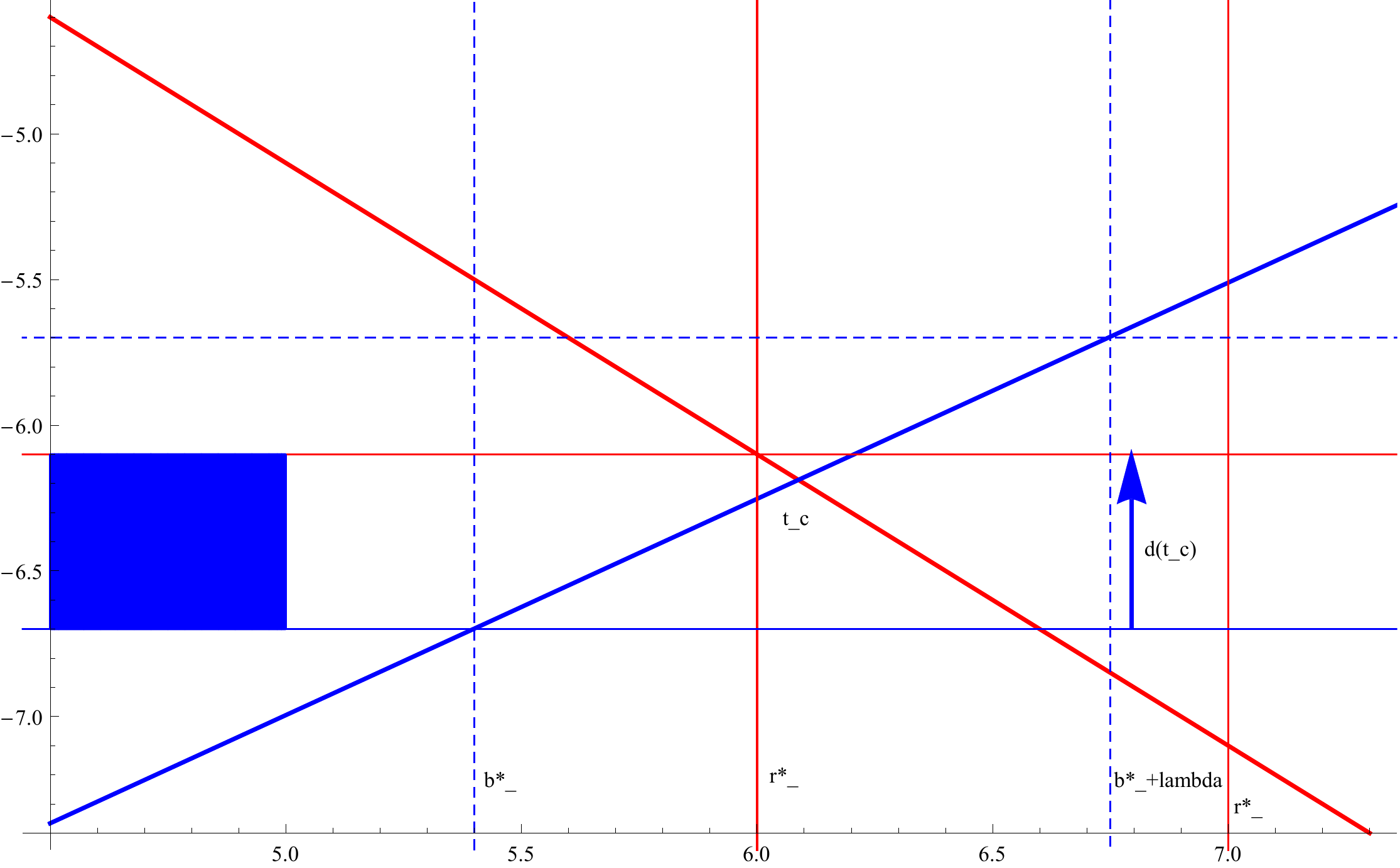}}
\subfigure[blue jumps first, $d(t_c)>1$]{\label{fig::bluedbig}
\includegraphics[keepaspectratio,width=5cm]{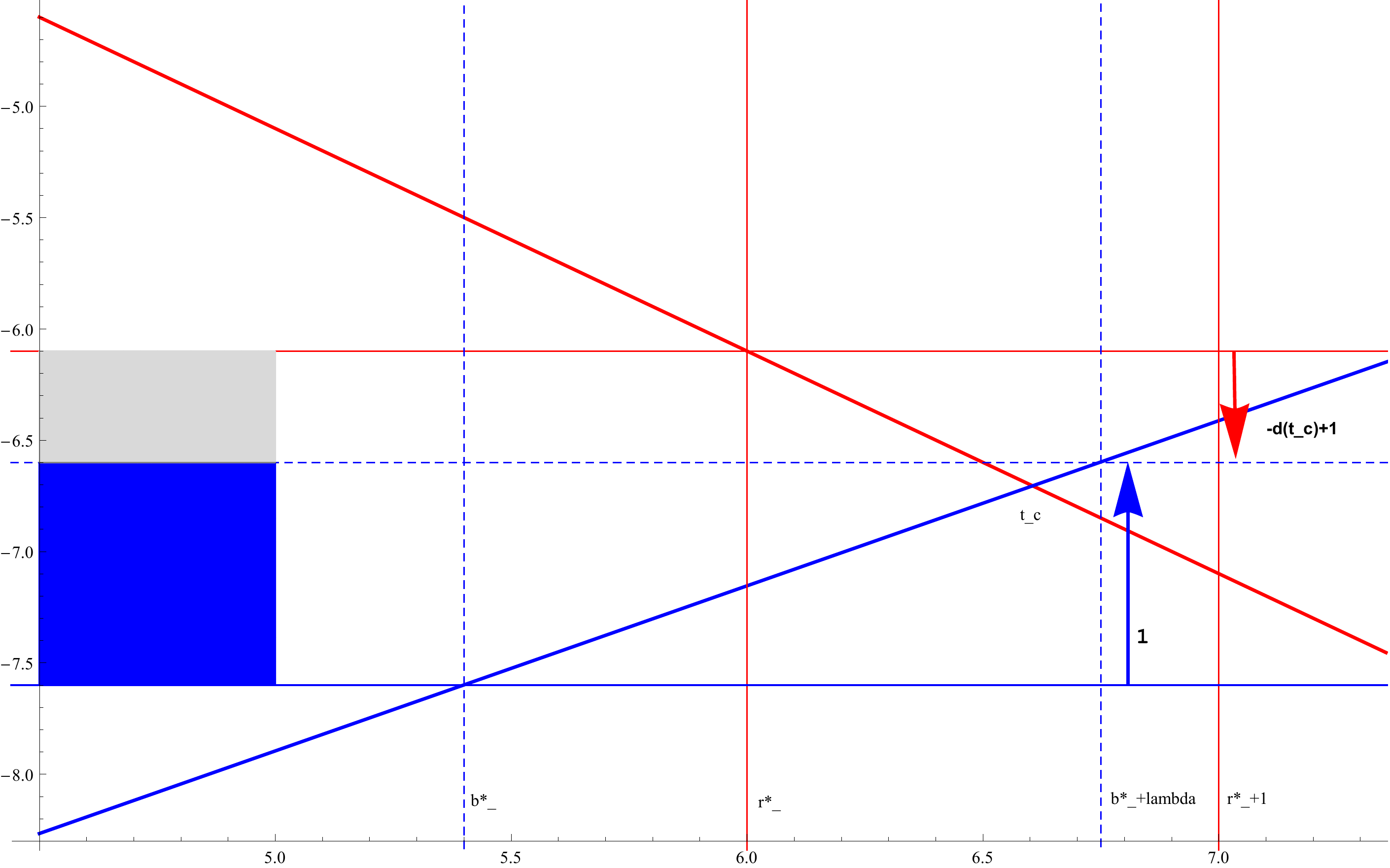}}
\caption{In these pictures the red and blue linear functions indicate the exponent of $1/(\tau-2)$ for $\log$(degree)/$\log n$ occupied by red and blue, where the horizontal axes corresponds to time. The vertical lines indicate the time of the jumps of red and blue before and after the intersection, while the horizontal lines indicate where the two processes are (or would be) at the time of the jumps. In these pictures, blue jumps first after the intersection, thus it can occupy more vertices: the exponent it can reach depends on how large the distance is between red and blue at their last jump before the intersection. The first picture shows a case where the distance before the jump is less than 1, the second when it is larger than 1. The colored regions indicate what degree-region can red and blue occupy at the end.}\label{fig::bluefirst}
\end{figure}

We will soon see that neglecting the integer part does have an influence on the highest degree vertex  blue can occupy.
 To get a more precise picture, we should compare which color is first and second to jump after $T_r+t_c$, since if it is blue, it can still increase its exponent. So, let us introduce the time of the last jump of red and blue before time $T_r+t_c$:
\be\label{def::rstar_bstar}r_-^\ast:=\lfloor t_c\rfloor, \qquad b^\ast_-:=\la \cdot \left\lfloor\frac{T_r+t_c}{\la}\right\rfloor.\ee
In words, red jumps at times $T_r+ r_-^\ast, T_r+r_-^\ast+1,\dots$ while blue jumps at times $b^\ast_-, b^\ast_-+\la, \dots$ and $t_c$ satisfies $T_r+ r_-^\ast\le t_c< T_r+ r_-^\ast+1$ and $b^\ast_-\le t_c<b^\ast_-+\la$.
We need to determine who jumps first after time $T_r+t_c$, (that is, $r_-^\ast+1<b^\ast_-+\la$ or the other way round), so let us also introduce the remaining times till the next jump \emph{after} the intersection for both colors:
\be\label{def::JrJb}\ J_r:=r_-^\ast\!+\!1-t_c=1-\{t_c\},\quad J_b:=b^\ast_-\!+\!\la-(T_r+t_c)= \la\left(1-\left\{\frac{T_r+t_c}{\la}\right\}\right).  \ee
$J_r$ and $J_b$ stands for the \emph{additional time} needed for red and blue till their next jump after time $T_r+t_c$.

\begin{remark}\normalfont
Note that given the values $\Ybn, T_r$ and  $\alpha$, with each additional jump,
 red decreases the exponent of $1/(\tau-2)$ by $1$ and blue increases its exponent by $1$.
 (Here we again neglect the terms including $C\log n$.)  Thus, for red, when plotting
 the exponents of $1/(\tau-2)$ of $\log \wit u_\ell/\log n$ one gets a line of slope $-1$,
 starting from time $T_r+1$ from the value $\alpha$. The exponent of $1/(\tau-2)$ in
  $\log u_i^{(b)}/\log n$ in the cluster of blue is a line of slope $1/\la$, since it increases by one with every additional $\la$ time units, see \eqref{eq::maxdegreeatTr+l}. These lines can be seen in Fig.\ \ref{fig::bluefirst} and Fig.\ \ref{fig::redfirst}.
\end{remark}

Intuitively, the final exponent of the maximal degree of blue depends on two things:
 which color jumps first after the intersection time $T_r+t_c$ and how large the difference $d(t_c)$ is
 between the exponents of $1/(\tau-2)$ in the $\log$(degree)$/\log n$ of red and blue before time $T_r+t_c$. Since with each jump the exponent of $1/(\tau-2)$ of the jumping color is changed by one, it is crucial whether this difference is less than or larger than $1$.
 Let us temporarily postpone the calculations and believe that this difference is
 \be\label{def::dtc} d(t_c):= \{t_c\}+\left\{\frac{T_r+t_c}{\la}\right\}.\ee We will later analyse this difference in detail around equation \eqref{eq::dtc2}.
Since $d(t_c)$ is the sum of two fractional parts, it is at most $2$.
 Recall also that $J_r, J_b$ stands for the time till the next jump of red and blue after time $T_r+t_c$,
  respectively (see \eqref{def::JrJb}). With these notations in mind, there are five cases (compare them to Fig.\ \ref{fig::bluefirst} and Fig.\ \ref{fig::redfirst}).
\begin{enumerate}[\upshape (B1)]
\item\label{caseBsmall} $J_b<J_r$ and $d(t_c)<1$. Blue jumps first after the intersection and occupies some vertices up to
 $\wit\Gamma_{r^\ast_-}$, i.e.\ blue can increase the exponent by a factor $(\tau-2)^{-d(t_c)}$. (Vertices with higher degree than
 that are already red). See Fig~\ref{fig::bluedsmall}.
\item\label{caseBlarge} $J_b<J_r$ and $d(t_c)>1$. Blue jumps first after the intersection
 and occupies some vertices one layer higher, namely the total exponent of $1/(\tau-2)$
 in \eqref{eq::maxdegreeatTr+l} reached by blue is $\left\lfloor\tfrac{T_r+t_c}{\la}\right\rfloor+1$. However,
  since $1<\la$, the next jump after this must be a red jump, hence red occupies every vertex with higher degree than this value. See Fig.\ \ref{fig::bluedbig}.
\end{enumerate}
\begin{enumerate}[\upshape (R1)]
\item\label{caseRsmall} $J_r<J_b$ and $d(t_c)<1$. Red jumps first after the intersection, and occupies every not-yet
 blue vertex down to $\wit\Gamma_{r_-^\ast+1}$, which means that blue cannot increase its exponent anymore.
 Thus the exponent of $1/(\tau-2)$ in \eqref{eq::maxdegreeatTr+l} of the maximal degree reached by blue is $\left\lfloor\tfrac{T_r+t_c}{\la}\right\rfloor$.  See Fig.\ \ref{fig::reddsmall}.
\item\label{caseRonce}$J_r<J_b<J_r+1$ and $d(t_c)>1$. Red can make only one jump after
 the intersection and occupies every vertex in $\wit\Gamma_{r_-^\ast+1}$, while blue jumps after this and
  can reach some vertices with degree up to $\wit\Gamma_{r_-^\ast+1}$ with its next jump. Thus the maximal degree of blue in this case is
   determined by $\wit\Gamma_{r_-^\ast+1}$, see Fig.\ \ref{fig::reddlarge1}.
\item\label{caseRtwice} $J_r+1<J_b$ and $d(t_c)>1$. Red can make at least two consecutive
jumps after the intersection and occupies every not-yet occupied vertex in $\wit\Gamma_{r_-^\ast+2}$,
which means that blue can not increase its exponent. The exponent of $1/(\tau-2)$ in \eqref{eq::maxdegreeatTr+l}
 of the maximal degree reached by blue is again $\left\lfloor\tfrac{T_r+t_c}{\la}\right\rfloor$,  see Fig.\ \ref{fig::reddlarge2}.
\end{enumerate}

Note that above we only handle the cases when $J_b\neq J_r$: this can be ensured by restricting $\la$ to be irrational. If $\la=p/q, \ p,q \in \N$ is rational with $p$ and $q$ co-primes,  then every vertex that is $q t$ away from the blue source and $p t$ away from the red source for arbitrary $t\in \N$ might be occupied at the same (i.e, at time $p t$). In this case, the color of such a vertex  is chosen with probability $1/2$ independently of everything else. For the meeting time of the red avalanche and blue climber, a rational $\la$ implies cases when $J_b=J_r$ or $J_b=J_r+1$, i.e. the two processes jump at the same time after $t_c$. Here we list what happens in these cases, to be able to merge them in the cases above. We assume here that the adapted rule is so that there is a positive probability that a vertex becomes blue upon co-occupation.
\begin{enumerate}[\upshape (BR1)]
\item\label{case=small} $J_b=J_r$ and $d(t_c)<1$. Since there are lots of vertices just slightly smaller than $\wit u_{r^\ast_-}$, blue whp occupies some vertices up to that point, i.e. blue can increase the exponent by a factor $(\tau-2)^{-d(t_c)}$ again. This case can be  merged into Case B\ref{caseBsmall}.  
\item\label{case=large1} $J_b=J_r$ and $d(t_c)>1$. In this case, blue can occupy some of the vertices up to one $\Gamma_i^{\sss{(b)}}$ higher. This case can be merged into Case B\ref{caseBlarge}.
\item\label{case=large2} $J_b=J_r+1$ and $d(t_c)>1$. In this case, red jumps first and occupies all the vertices down to $\Gamma_{r^\ast_-+1}$, and then the two processes jump together, so blue can occupy some vertices right below that. This case can be merged into Case R\ref{caseRonce}. 
\end{enumerate}

\begin{remark}\normalfont
If the adapted rule is so that the probability that a vertex is going to be red with probability one upon co-occupation, then Case BR\ref{case=small} merges into Case R\ref{caseRsmall}, case BR\ref{case=large1} merges into Case B\ref{caseRonce}, and Case BR\ref{case=large2}  merges into Case R\ref{caseRtwice}. We see that the adapted rule only influences the place where the strict and non-strict inequality signs appear inside the indicators in $f(d(t_c), J_r, J_b)$ in \eqref{eq::fdtc} below. Hence, the main result still holds true with a slightly different $f(d(t_c), J_r, J_b)$. For other adapted rules, the function $f$ can be determined similarly.

\end{remark}
\begin{figure}
\centering
\subfigure[red jumps once, $d(t_c)>1$]{\label{fig::reddlarge1}
\includegraphics[keepaspectratio,width=5cm]{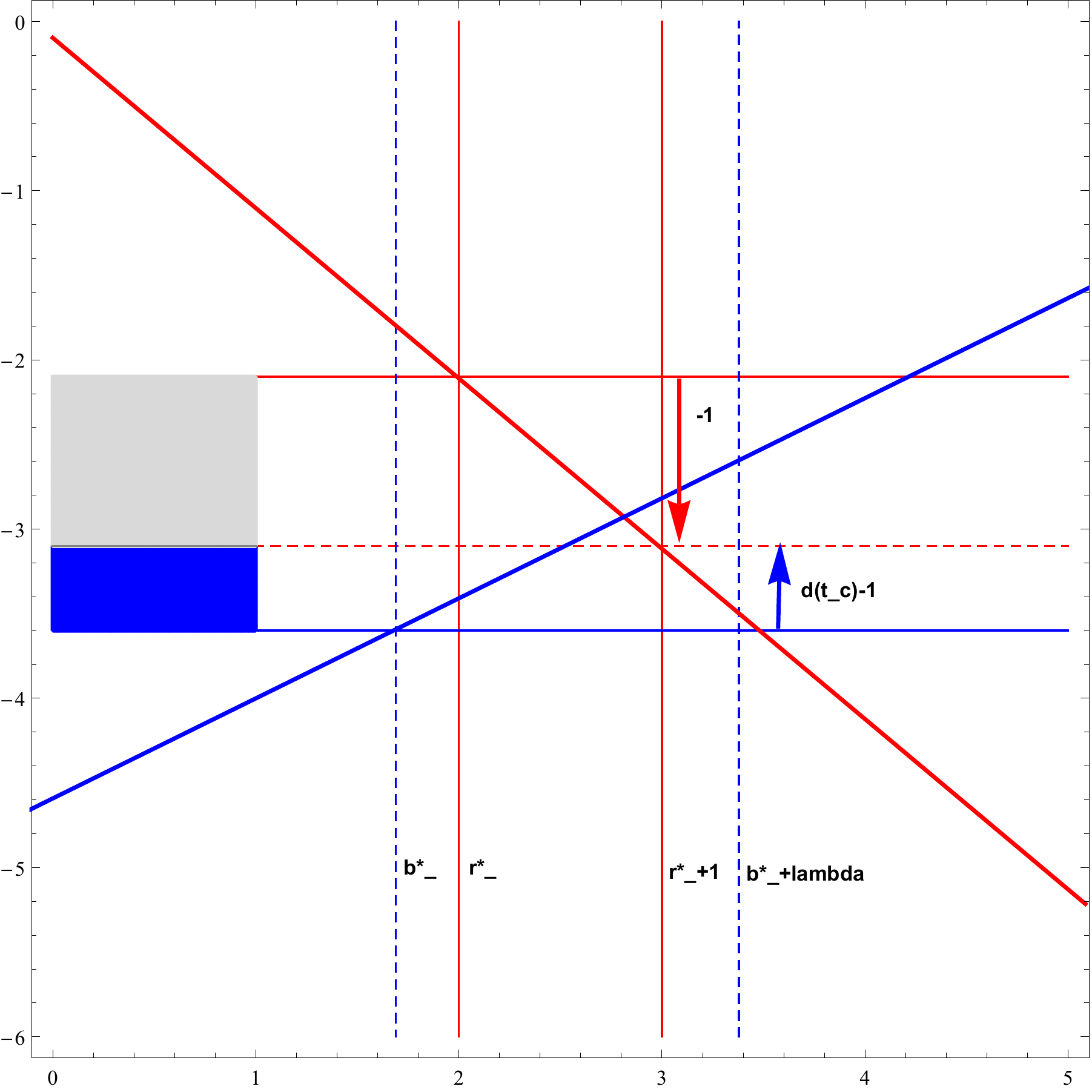}}
\subfigure[red can jump twice, $d(t_c)>1$]{\label{fig::reddlarge2}
\includegraphics[keepaspectratio,width=5cm]{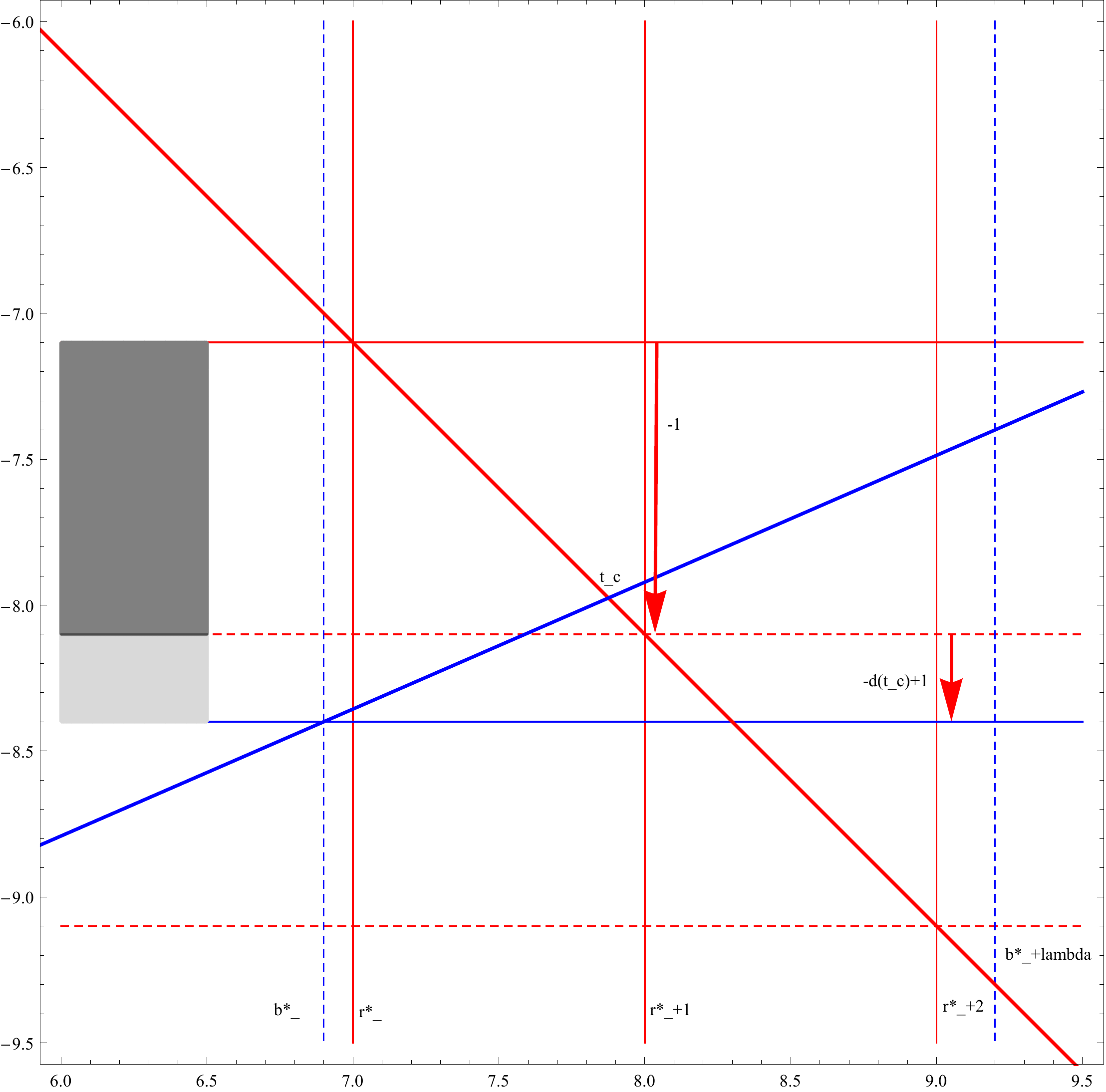}}
\subfigure[red jumps first, $d(t_c)<1$ ]{\label{fig::reddsmall}
\includegraphics[keepaspectratio,width=6cm]{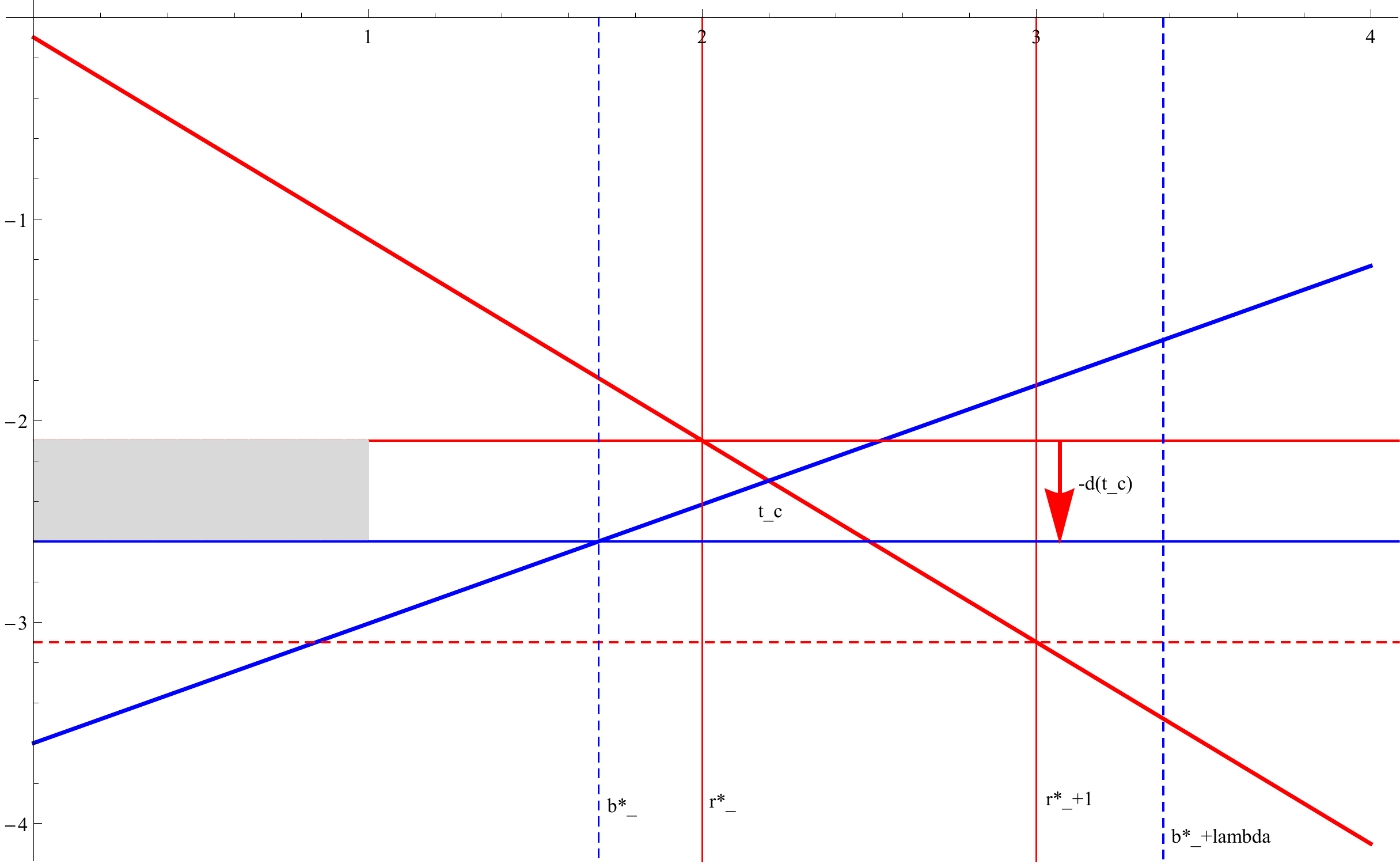}}
\caption{In these pictures the red and blue linear functions indicate the exponent of $1/(\tau-2)$ for the $\log$ (degrees)/$\log n$ of red and blue. The vertical lines indicate the time of the jumps before and after the intersection for the two colors, while the horizontal lines indicate where the two processes are (or would be) at the time of the jumps. In these pictures, red jumps first after the intersection, thus it can occupy more vertices: the exponent it can reach depends on how large the distance is between red and blue at their last jump before the intersection. The first two pictures show the two cases where the distance before the jump is more than one, and red can jump only once or at least twice after the intersection, respectively. The third pictures shows the case when distance before the jump is smaller than one.  The colored regions illustrate the maximal degree blue can reach. (In Fig~\ref{fig::reddlarge2} and \ref{fig::reddsmall} blue cannot increase its maximal degree anymore.)} \label{fig::redfirst}
\end{figure}

Now we formalize these heuristics by finishing the proof of Theorem \ref{thm::maxdegreeblue}. An elementary calculation is to solve \eqref{def::intersection} yielding
\be\label{eq::tc}\ba t_c&= \frac{\la}{\la+1} \frac{\log\log n +\log (\alpha/\Ybn)}{|\log (\tau-2)|} - \frac{T_r}{\la+1}\\
&=\frac{\la-1}{\la+1}\frac{\log\log n}{|\log (\tau-2)|}+\frac{ \log(\alpha^\la (\tau-1) \Yrn/ \Ybnla ) }{(\la+1)|\log (\tau-2)|} +\frac{1+b_n^{(r)}}{\la+1};\\
\frac{T_r+t_c}{\la}&=\frac{1}{\la+1} \frac{2\log\log n-\log (\Yrn \Ybn (\tau-1)/ \alpha )}{|\log (\tau-2)|}-\frac{1+b_n^{(r)}}{\la+1}.  \ea\ee
Hence, the red avalanche right before the intersection occupies every vertex with degree larger than $\wit u_{r_-^\ast}$, where $r_-^\ast=[t_c]=t_c-\{t_c\}$. Combining the formula of $t_c$ with $\wit u_\ell$ in \eqref{eq::ul}, we get
\be\label{eq::logul*} \log \big(\wit u_{r_-^\ast}\big)=  (\log n)^{\frac{2}{\la+1}} \left( \frac{\Ybnla}{\Yrn}\frac{\alpha}{\tau-1}\right)^{\frac{1}{\la+1}} (\tau-2)^{\frac{1+b_n^{(r)}}{\la+1}-1-\{t_c\}}(1+o(1)). \ee
On the other hand, since the last jump of blue before time $T_r+t_c$ is at time $\la[ (T_r+t_c)/\la]$, blue could do $(T_r+t_c)/\la - \{(T_r+t_c)/\la\}$ many up-jumps, hence right before the intersection, blue occupies some vertices that satisfy
\be\label{eq::logMbeforetc} \log \left( D_{\max}^{\sss{(b,n)}}(T_r+t_c)\!\right)=\!(\log n)^{\frac{2}{\la+1}}\! \left( \frac{\Ybnla}{\Yrn} \frac{\alpha}{\tau-1}\right)^{\frac{1}{\la+1}}\! (\tau-2)^{\frac{1+b_n^{(r)}}{\la+1}-1+\left\{\frac{T_r+t_c}{\la}\right\}}(1+o(1)),       \ee
where we have used \eqref{eq::tc} for $(T_r+t_c)/\la$ combined with \eqref{eq::maxdegreeatTr+l} at time $\left\lfloor(T_r+t_c)/\la\right\rfloor$.

Note that the formulas \eqref{eq::logul*} and \eqref{eq::logMbeforetc} only differ in the exponents of $1/(\tau-2)$, and this difference is exactly $d(t_c)$, introduces in defined in \eqref{def::dtc}. More precisely,
\be\label{eq::dtc2} \frac{\log\left( \frac{   \log \big(\wit u_{r_-^\ast}\big)   }{ \log \left( D_{\max}^{\sss{(b,n)}}(T_r+t_c)\!\right) }  \right) }{|\log (\tau-2)|} =  \{t_c\} + \left\{\frac{T_r+t_c}{\la}\right\}+o(1):=d(t_c)+o(1).\ee

Recall from \eqref{def::JrJb} that the remaining time to the next jump for red and blue after the intersection at time $T_r+t_c$ is denoted by $J_r$ and $J_b$, respectively.

Since \eqref{eq::logMbeforetc} is the exponent of the maximal degree vertex that blue occupies before the intersection, to determine the maximal degree of blue, we need to investigate whether blue can jump once more before the red avalanche reaches lower degrees than \eqref{eq::logMbeforetc}. If yes, then blue can gain an additional factor to the rhs of \eqref{eq::logMbeforetc}.

Obviously, if $d(t_c)<1$, then even though blue jumps first, it cannot increase its exponent by a whole factor $(\tau-2)^{-1}$, since vertices with degree larger than $(D_{\max}^{\sss{(b,n)}}(T_r+t_c))^{(\tau-2)^{-1}}$  are already all red. It is not hard to see that blue in this case will occupy some vertices `right below' $u_{[t_c]}$ (that is, say, higher than $\widetilde u_{[t_c]}/(C\log n)$), hence blue in this case can increase its exponent by $(\tau-2)^{d(t_c)}$.

This case illustrates that the additional factor that we need to add to the rhs of \eqref{eq::logMbeforetc} depends on two things: (1) which color jumps first (and possibly second) after the intersection and (2) whether $d(t_c)>1$ or not. There are five cases, described above (after formula \eqref{def::dtc}). As a result, the gain in the exponent for blue can be summarized by multiplying \eqref{eq::logMbeforetc} by the following function containing indicators for these five cases
(the order is Case B\ref{caseBsmall}, R\ref{caseRsmall}, B\ref{caseBlarge}, R\ref{caseRonce}, R\ref{caseRtwice} here, and the cases where $\la$ rational are also included):
\be\label{eq::fdtc} \ba f&\left(d(t_c), J_r, J_b\right):=\ind_{\left\{d(t_c)<1\right\}}\left(\ind_{\{J_b\le J_r\}}(\tau-2)^{-d(t_c)}+ \ind_{\{J_r<J_b\}}\right)\\
+ &\ind_{\left\{d(t_c)>1\right\}}\left( \ind_{\{J_b\le J_r\}}(\tau-2)^{-1} \!+\ind_{\{J_r<J_b\le J_r+1\}}(\tau-2)^{1-d(t_c)}+\ind_{\{J_r+1<J_b\}}\right).\ea\ee
Note that every expression in this formula,  i.e.,  $d(t_c), J_r, J_b$ are simple functions of the (random) fractional parts   $\{t_c\}, \left\{\frac{T_c+t_c}{\la}\right\}$ and  $b_n^{(r)}$. Further,  $\{t_c\}, \left\{\frac{T_c+t_c}{\la}\right\}$ and  $b_n^{(r)}$ depend only on $n$ and also on $\Yrn, \Ybn$. Hence knowing these values determines $f$ uniquely.

 Recall the value $\alpha$ from \eqref{def::alpha}, combine \eqref{eq::logMbeforetc} with the additional factor $f\left(d(t_c), J_r, J_b\right)$, so that we can introduce the `oscillation-filtering' random variable
\be\label{def::dn} C_n^{\max}(\Yrn, \Ybn)\!:=\! \left(\!\frac{(\tau-1)-(\tau-2)^{b_n^{(r)}}}{(\tau-1)^2}\!\right)^{\!\frac{1}{\la+1}}\!\!\!\!\!\!(\tau-2)^{\frac{-\la+b_n^{(r)}}{\la+1}+\left\{\frac{T_r+t_c}{\la}\right\}}\! f\left(d(t_c), J_r, J_b\right), \ee
which is oscillating with $n$ and is random, but is depending on the same randomness as $\Yrn, \Ybn$, i.e., they are defined on the same probability space. 

At this point we have shown that

\be\label{eq::logd}\log D_{\max}^{\sss{(b,n)}}(\infty) = \!(\log n)^{\frac{2}{\la+1}}\! \frac{\Ybnla}{\Yrn} C_n^{\max}(\Yrn, \Ybn) (1+o_{\Pv}(1)). \ee
To obtain the statement of Theorem \ref{thm::maxdegreeblue}, we have to check the conditions of Lemma \ref{claim::maxdegblue}. For this, note that the last layer of blue is reached after climbing $i_{\max}=\lfloor (T_r+t_c)/\la\rfloor-\lfloor t(n^{\vr'})/\la \rfloor+\ind_{\{ \text{Case B} \ref{caseBlarge}\}}$ many $\Gamma_i^{\sss{(b)}}$ layers,
and by \eqref{eq::tc} and \eqref{eq::an} we can calculate
\be\label{def::imax} i_{\max}:=\lfloor (T_r+t_c)/\la\rfloor-\lfloor t(n^{\vr'})/\la \rfloor+\ind_{\{ \text{Case B} \ref{caseBlarge}\}}=\frac{\la-1}{\la(\la+1)} \frac{\log\log n}{|\log (\tau-2)|} + O_\Pv(1). \ee
Thus,  $i_{\max}=O^\diamond(\frac{\la-1}{\la(\la+1)})$, hence the conditions of the lemma hold.  Thus, we get the first statement of Theorem \ref{thm::maxdegreeblue} by noting that $(\Yrn, \Ybn)\toindis (Y_r, Y_b)$.

By maximizing and minimizing the constants and the fractional parts in \eqref{def::dn} and in
the indicators in $f(d(t_c), J_r, J_c)$ we get the bounds in \eqref{eq::dnmaxbounds}.
This finishes the proof of Theorem \ref{thm::maxdegreeblue}.

\subsection{Number of maximum degree vertices}
With the last up-jump of blue, blue occupies some vertices of degree of order $D_{\max}^{\sss{(b,n)}}(\infty)$.
In this section we investigate \emph{how many} maximum degree vertices are reached by blue. We show that in some cases (namely, Cases B\ref{caseBsmall}, R\ref{caseRonce}) the number of these vertices is so large that it corresponds to an additional factor for the total number of half-edges in maximum degree vertices of blue.

More precisely, let us denote the set of outgoing half-edges from these maximal degree vertices by $\CMBN$, and its size by $\MBN$.
 Later we will determine how many vertices blue can occupy after this phase, and to be able to count that we need to know how many half-edges are in the highest layer of blue.
 \begin{lemma}\label{lem::verticeswithmaxdegree}
For $\MBN$, the number of outgoing half-edges from the set of maximal degree vertices, i.e. the sum of the forward degrees reached by blue for which \eqref{eq::maxdegreeatinfty} holds, we have
\[ \frac{\log\MBN}{(\log n)^{\frac{2}{\la+1}} C_n^{\mathrm{half-edge}}( \Yrn, \Ybn)  } \toinp \left( \frac{Y_b^\la}{Y_r}\right)^{1/(\la+1)},  \]
where  $C_n^{\mathrm{half-edge}}( \Yrn, \Ybn)\ge C_n^{\max}( \Yrn, \Ybn)$ is a bounded random variable given below in formula \eqref{def::D3}.
\end{lemma}

\begin{proof}
Recall that $A_i$ denotes the number of vertices blue occupies in layer $\Gamma_i^{\sss{(b)}}$ upon reaching it, see \eqref{def::Ai}.
In the cases where blue finishes its last jump at a certain layer $\Gamma_i^{\sss{(b)}}$, that is,
in Case R\ref{caseRsmall} (Fig~\ref{fig::reddsmall} and Case R\ref{caseRtwice} (Fig~\ref{fig::reddlarge2}) and also in Case B\ref{caseBlarge} (Fig~\ref{fig::bluedbig})
the statement is a direct consequence of Lemma \ref{lem::numberofverticesinGamma}, since blue is stuck with its maximal degree at a given layer $\Gamma_{i_{\max}}^{\sss{(b)}}$, and hence $\MBN= A_{i_{\max}} D_{\max}^{\sss{(b,n)}}(\infty)(1+o(1))$. Taking logarithm we get
\be \label{eq::logsum}\log \MBN \le \log D_{\max}^{\sss{(b,n)}}(\infty)(1+o(1))+ \log  A_{i_{\max}}.\ee By \eqref{def::imax}, $i_{\max}=O^\diamond(\frac{\la-1}{\la(\la+1)})$ in Lemma \ref{lem::numberofverticesinGamma}, so we can use the bound in \eqref{eq::aifinal} with $x=(\la-1)/\la(\la+1)$.
 Hence, the last term in \eqref{eq::logsum} disappears when we divide by $(\log n)^{2/(\la+1)}$.

We are left with handling the cases where the last jump of blue is not a full layer, i.e., Cases B\ref{caseBsmall} and R\ref{caseRonce}.
 In these cases, after reaching layer $\Gamma_{i_{\max}}^{\sss{(b)}}$, blue still jumps up, but not a full layer:
 due to the presence of red the forward degrees are truncated at $\wit u_{r^\ast_-}$ in Case B\ref{caseBsmall} and at $\wit u_{r^\ast_-+1}$ in Case R\ref{caseRonce}.

First, we apply Lemma \ref{lem::numberofverticesinGamma} to see that $\log A_{i_{\max}}$ in the last `full' layer $\Gamma_{i_{\max}}^{\sss{(b)}}$ is small.
Let us recall the notation $u_{i_{\max}}^{\sss{(b)}}=D_{\max}^{\sss{(b,n)}}(T_r+t_c)$, and  introduce the extra factor of the $\log$(degrees) reached at
  the last up-jump of blue by \[ \gamma:=(\tau-2)^{-d(t_c)\mathbb \ind\{B\ref{caseBsmall}\}+(1-d(t_c)) \ind\{R\ref{caseRonce}\} }. \]  Then we introduce a new layer
\[ \Gamma^\diamond:=\left\{ v\in \CMD:   d_v \ge  \frac{(u^{\sss{(b)}}_{i_{\max}})^\gamma}{(\log n)^{1/(\tau-2)} } \right\},\]
and we denote the number of half-edges in this set by $\CE_{\gamma}$.

By Lemma \ref{lem::numberofverticesinGamma}, whp blue is not reaching higher degrees than $\widehat u_{i_{\max}}^{\sss{(b)}}$ at time $i_{\max}$. Recall that  there are
 $A_{i_{\max}}$ many blue vertices in layer $\Gamma_{i_{\max}}^{\sss{(b)}}$. Hence, the total number of blue half-edges in this layer is at most $A_{i_{\max}}\widehat u_{i_{\max}}^{\sss{(b)}}$. Thus, the number of vertices in $\Gamma^\diamond$ to which blue is connected is dominated by
 \be\label{eq::bindom3}\mathcal B \cap \Gamma^\diamond \  {\buildrel {d}\over{ \le }}\     {\sf Bin} \left( A_{i_{\max}} \widehat u_{i_{\max}}^{\sss{(b)}}, \frac{\CE_{\gamma}}{\CL_n (1+o(1))}  \right). \ee
Using Claim \ref{claim::Sbound}, $\CE_\gamma\le (u_{i_{\max}}^{\sss{(b)}})^{-\gamma(\tau-2)} $ whp.
 Thus, conditioned on $A_{i_{\max}}$, the expected value of the Binomial variable in \eqref{eq::bindom3} is bounded above by
 \[  \frac{2 C_1^2}{c_1} A_{i_{\max}} \widehat u_{i_{\max}}^{\sss{(b)}}  (u_{i_{\max}}^{\sss{(b)}})^{-\gamma(\tau-2)}
  = \frac{2 C_1^2}{c_1} A_{i_{\max}}\frac{\widehat u_{i_{\max}}^{\sss{(b)}}}{u_{i_{\max}}^{\sss{(b)}}} (u_{i_{\max}}^{\sss{(b)}})^{1-\gamma(\tau-2)}.  \]
Since red occupies every vertex with degree larger than $(u^{\sss{(b)}}_{i_{\max}})^\gamma$, the previous formula bounds the number of vertices with degree in the interval $[(u^{\sss{(b)}}_{i_{\max}})^\gamma/C\log n, (u^{\sss{(b)}}_{i_{\max}})^\gamma)$. Thus, the total number of half-edges going out from maximal degree vertices can be bounded by
 \be\label{eq::addgamma} \frac{2 C_1^2}{c_1} A_{i_{\max}}\frac{\widehat u_{i_{\max}}^{\sss{(b)}}}{u_{i_{\max}}^{\sss{(b)}}} (u_{i_{\max}}^{\sss{(b)}})^{1+\gamma(3-\tau)}.\ee
  Since $i_{\max}=O^\diamond(\frac{\la-1}{\la(\la+1)})$, we can use \eqref{eq::aifinal} and the calculations in the proof of Lemma \ref{lem::numberofverticesinGamma} to see that  \[ A_{i_{\max}}\frac{\widehat u_{i_{\max}}^{\sss{(b)}}}{u_{i_{\max}}^{\sss{(b)}}}\le (A_{O^\diamond(\frac{\la-1}{\la(\la+1)})})^2\] is still small, i.e., it disappears when taking logarithm and dividing by $(\log n)^{2/(\la+1)}$. Hence, the main contribution comes from $( u_{i_{\max}}^{\sss{(b)}} )^{1+\gamma(3-\tau)}=(D_{\max}^{\sss{(b,n)}}(T_r+t_c))^{1+\gamma(3-\tau)}$.

  Hence, in Cases  B\ref{caseBsmall} and R\ref{caseRonce}, blue can get more half-edges than of order $D_{\max}^{(b,n)}(\infty)$. To get the total number of half-edges at the last up-jump, we need to modify the function $f(d(t_c), J_r,J_b)$.
 An elementary rearranging of the indicators of the cases and the constants shows that the extra factor needed for \eqref{eq::logMbeforetc} to get $\MBN$ is
 \be\label{def::gfunction} \ba &g(d(t_c), J_r, J_b):=\ind_{\left\{d(t_c)<1\right\}}\left(\ind_{\{J_b\le J_r\}}(1+(3-\tau)(\tau-2)^{-d(t_c)}+ \ind_{\{J_r<J_b\}}\right)\\
&+ \ind_{\left\{d(t_c)>1\right\}}\left( \ind_{\{J_b\le J_r\}}(\tau-2)^{-1} +\ind_{\{J_r<J_b\le J_r+1\}}(1+(3-\tau)(\tau-2)^{1-d(t_c)})\right)\\
&+\ind_{\left\{d(t_c)>1\right\}}\ind_{\{J_r+1<J_b\}}.\ea\ee
Then the normalizing constant for $\MBN$ is given by
\be\label{def::D3}C_n^{\text{half-edge}}(\Yrn, \Ybn)\!:=\! \left(\!\frac{(\tau-1)-(\tau-2)^{b_n^{(r)}}}{(\tau-1)^2}\!\right)^{\!\frac{1}{\la+1}}\!\!\!\!\!\!(\tau-2)^{\frac{-\la+b_n^{(r)}}{\la+1}+\left\{\frac{T_r+t_c}{\la}\right\}}\! \ g\left(d(t_c), J_r, J_b\right). \ee
This finishes the proof of Lemma \ref{lem::verticeswithmaxdegree}.
\end{proof}

Before moving on to the next section, let us introduce the time when the maximal degree is reached, which is nothing else but the time of the last possible up-jump of blue, i.e.,
\be\label{def::t_b}\ba t_b&:=\la \left(\left\lfloor\frac{T_r+t_c}{\la} \right\rfloor +\ind_{\{J_b<J_r\}}+\ind_{\{J_r<J_b<J_r+1\}}\ind_{\{d(t_c)> 1\}} \right)\\
&= T_r+t_c + \la \left(\ind_{\CE} - \left\{\frac{T_r+t_c}{\la}\right\}\right) \\
&= \frac{\la}{\la+1} \frac{2\log\log n-\log (\Yrn \Ybn (\tau-1)/ \alpha )}{|\log (\tau-2)|}-\frac{1+b_n^{(r)}}{\la+1} +
\la \left(\ind_{\CE} - \left\{\frac{T_r+t_c}{\la}\right\}\right),
\ea \ee
where $\CE$ stands for the event that blue has an additional up-jump after time $t_c$, i.e.\ Case B\ref{caseBsmall}, B\ref{caseBlarge} or R\ref{caseRonce} happens.

\section{Path counting methods for blue} By time $t_b$, only $o(n)$ vertices are reached by red and blue together -- most of the vertices are still not colored. Thus, it still remains to determine how many vertices blue can reach after time $t_b$. We do this via giving  matching upper and  lower bounds on how many vertices blue occupies in this last phase.

For the upper bound, the idea is that we count the size of the local neighborhood of the half edges that are just occupied at time $t_b$. Since the red avalanche continues to be in its avalanche phase and occupies all vertices of smaller and smaller degrees as time passes, the spreading of blue is more and more restricted, so this local neighborhood is quite small. We call this the \emph{optional cluster} of blue. Since its size is random, we give a concentration result on its size, i.e.,\ we give a concentrated upper bound on what blue can get.

For the lower bound, we estimate how much the red color might `bite out' of this optional cluster. This can happen since even a constant degree vertex might by chance be close to both colors. We show that this intersection of the clusters is negligible compared to the size of the optional cluster.

We start describing the first step -- the optional cluster of blue -- in more detail. At time $t_b$, the half-edges in the set $\CMBN$ start their own \emph{exploration clusters}, i.e.,\ an exploration process from the half-edge to not-yet
 occupied vertices. At time $t_b+\la j$, we color every vertex $v$, whose distance is exactly $j$ from some half-edge $h$ in $\CMBN$, and the degrees of vertices on the path from $h$ to $v$ are less than what red occupies at that moment, blue.  That is, the degree of the $j$th vertex on the path must be less than $\wit u_{\lfloor t_b-\la i+\la j-T_r\rfloor}$.
 We do this via estimating the number of paths with degree restrictions from $\CMBN$ and call this the \emph{optional cluster of blue}, denote the set by $\opt_{\max}$ and its size by $\ops_{\max}$. Corollary \ref{cor::chebisev} below determines its asymptotic behavior.

On the other hand, not just the half-edges in $\CMBN$ can gain extra blue vertices: from half-edges in $\CA_{i_{\max}-z},\  z=0,1,2\dots$ the explorations start a bit earlier (at time $t_b-\la z$) towards small degree vertices. Let us denote the vertices reached via half-edges from layer $\CA_{i_{\max}-z}\setminus \CA_{i_{\max}-z+1}$ by $\opt_{-z}, z\ge 0.$ At time $t_b-\la z+\la j$, we color a vertex $v$ blue if its distance is exactly $j$ from a half-edge $h$ in $\CA_{i_{\max}-z}$, and the degrees of vertices on the path from $h$ to $v$ are less than $u_{i_{\max}-z+j}^{(b)}$ and also what red occupies at that moment, i.e., the degree of the $j$th vertex on the path must be less than $\min\{u_{i_{\max}-z+j}^{(b)}, \wit u_{\lfloor t_b-\la z+\la j-T_r\rfloor}$\}.  This extra truncation is needed since we want to avoid \emph{double counting}, that is, we do not want to count vertices explored from $\CA_{i_{\max}-z}$ towards $\CA_{i_{\max}-z+1}$, hence the additional restriction. We show that the total number of optional blue vertices in lower layers, $\sum_{z\ge 0}\ops_{-z}$ with these additional explorations is at most the same order as $\ops_{\max}$ in Lemma \ref{lem::opt-z-lemma}.

For the lower bound of what blue can occupy after time $t_b$, note that not every vertex in $\opt_{\max}$ will be occupied by blue: red can still bite out some parts of these vertices by simply randomly being close to some parts of the blue cluster. We estimate the number of vertices in the intersection of $\opt_{\max}$ and red, and then subtracting the gained estimate from the lower bound on $\ops_{\max}$ gives a lower bound on what blue occupies from the graph after $t_b$, see Lemma \ref{lem::red-blue-intersection}. Now we turn to the calculations.

We introduce the expected truncated degree of a vertex that is distance $j$ away from the set $\CMBN$
\be\label{def::nuj} \nu_j:=\Ev\left[B \ind_{\left\{B< \wit u_{\lfloor t_b+\la j-T_r\rfloor}\right\}} \right]= \frac{1}{\Ev[D]}\int_0^{\wit u_{\lfloor t_b+\la j-T_r\rfloor}}\!\!\!\![1-F^\ast(x)] \mathrm dx .  \ee
Then, by \eqref{eq::F},
\be\label{eq::nujbound}  \frac{c_1}{\Ev[D]} \left(\wit u_{\lfloor t_b+\la j-T_r\rfloor}\right)^{3-\tau}\le \nu_j \le  \frac{C_1}{\Ev[D]} \left(\wit u_{\lfloor t_b+\la j-T_r\rfloor}\right)^{3-\tau}.\ee
Let us also define
\[ \kappa_j:=\frac{1}{\Ev[D]}\Ev[D(D-1)(D-2)\ind_{\left\{D< \wit u_{\lfloor t_b+\la j-T_r\rfloor}\right\}}] \]
Then again by \eqref{eq::F},
\[ \frac{c_1}{\Ev[D]} \left(\wit u_{\lfloor t_b+\la j-T_r\rfloor}\right)^{4-\tau}\le \kappa_j \le  \frac{C_1}{\Ev[D]} \left(\wit u_{\lfloor t_b+\la j-T_r\rfloor}\right)^{4-\tau}.\]
Let us call a path of length $k$ from $\CMBN$ with vertices $\left(\pi_j\right)_{j\le k}$ \emph{good} if $\pi_j\le \wit u_{\lfloor t_b+\la j-T_r\rfloor}$, and \emph{good-directed}  if  $ u_{\lfloor t_b+\la {j+1}-T_r\rfloor}\le \pi_j\le \wit u_{\lfloor t_b+\la j-T_r\rfloor}$.
\begin{lemma}\label{lem::wihZk} For $k\ge 0$, denote by $\ops_{\max}(k), \ops_{\max}^{\text{d}}(k)$ the number of vertices that are on good and good-directed paths of distance $k$ away from $\CMBN$, respectively.  Then there exist positive constants $0<c_2\le C_2 < \infty$ such that
\be \label{eq::wihZk_expected}   \MBN \cdot \prod_{j=1}^{k} \nu_j \le \Ev[\ops_{\max}(k)\mid \MBN ]\le  \MBN \cdot \prod_{j=1}^{k} \nu_j \cdot \left(1+O\left(\frac{k^2}{n}\right)\right)\ee
and
\be \label{eq::wihZk_expected2}    \MBN \cdot \prod_{j=1}^{k} (\nu_j-\nu_{j+1}) \le \Ev[\ops_{\max}^{\text{d}}(k)\mid \MBN ]\le  \MBN \cdot \prod_{j=1}^{k} (\nu_j-\nu_{j+1}) \cdot \left(1+O\left(\frac{k^2}{n}\right)\right)\ee
while for the variance of the latter:
\be \ba \label{eq::wihZk_variance}&\Vv[\ops_{\max}^{\text{d}}(k)|\MBN] \le \Ev[\ops_{\max}^{\text{d}}(k)|\MBN] \\
&\ + \overline{\Ev[\ops_{\max}(k)|\MBN]}^2 \cdot \left(\frac{\nu_{k-1}}{(\nu_{k-1}-1)}\frac{\kappa_1}{\nu_{1}^2}\left(\frac{1}{\MBN} + \frac{2}{\CL_n}\right)+ \frac{\nu_{k-1}^2}{(\nu_{k-1}^2-1)^2}\frac{\kappa_1^2}{\nu_{1}^4} \frac{2}{\MBN  \CL_n}+e_{k,n}\right), \ea\ee
where  $\overline{\Ev[\ops_{\max}^{\text{d}}(k)|\MBN]}$ means the upper bound on $\Ev[\ops_{\max}^{\text{d}}(k)|\MBN]$ in \eqref{eq::wihZk_expected2}, and the error term $e_{k,n}$ is
\be\ba\label{def::ekn} e_{k,n}&=\left( \prod_{i=1}^{k} \frac{\CL_n-2i+1}{\CL_n-2i-2k+1}-1\right)\\
+&   \left(1+\frac{\kappa_1\nu_{k-1}}{\nu_1^2}\frac{1}{\MBN} \right)\left(1+\frac{\kappa_1 \nu_{k-1}}{\nu_1^2 }\frac{1}{c\CL_n}\right)\frac{k}{\nu_{k-1}-1}\left(e^{k^2 \kappa_1^2 \nu_{k-1}/(\nu_{1}^4\CL_n)}-1\right).\ea\ee
\end{lemma}
 The proof of this lemma uses path counting methods and is similar to that of \cite[Lemma 5.1]{Janson10}. Similar techniques can also be found in \cite[Section 10.4.2]{H10}. Since our case is slightly different than the cases handled there, we work out the details in Appendix \ref{app::paths}.

Now we state the immediate corollary of Lemma \ref{lem::wihZk}. Recall the definition of $t_b$ from \eqref{def::t_b}.
\begin{corollary}[Chebyshev's inequality for blue vertices]\label{cor::chebisev} Take $c_3\le \tfrac{2-\ve}{\la+1} |\log(\tau-2)|^{-1}$ and any $k\le c_3\log\log n.$
Then, conditioned on the number of blue half-edges $\MBN$ at time $t_b$, the number of vertices optionally occupied by blue up to time $t_b+ \la k$ satisfies that, conditionally on $\MBN$,
\[\frac{\log (\ops_{\max}(k))}{\log \MBN + \sum_{i=1}^{k-1} \log \nu_i} \toinp 1. \]
\end{corollary}
\begin{proof}
Let us write $\ops_{\max}^{\text{non-d}}(k)$ for paths that are good but not good-directed. We show that they have a negligible contribution, while $\ops_{\max}^{\text{d}}(k)$ is well-concentrated. In this proof below, all expectations and probabilities are conditional wrt.\ $\MBN$. Let us write
\[\ba J:=\Pv&\left( \big|\ops_{\max}(k) -\Ev[\ops_{\max}(k) ]\big|\ge \frac12 \Ev[\ops_{\max}(k)] \right)\\
\le\Pv&\left( \big|\ops_{\max}^{\text{d}}(k) -\Ev[\ops_{\max}(k) ]\big|\ge \frac13 \Ev[\ops_{\max}(k)] \right)+
\Pv\left( \ops_{\max}^{\text{non-d}}(k) \ge \frac16 \Ev[\ops_{\max}(k)] \right). \ea\]
Now we can apply Chebyshev's inequality on the first term while Markov's inequality on the second term (both conditioned on $\MBN$), using  Lemma \ref{lem::wihZk}:
\be\label{eq::J}\ba J&\le \frac{9 \Vv[\ops_{\max}^{\text{d}}(k)| \MBN]}{\Ev[\ops_{\max}(k)|\MBN]^2}  + \frac{\Ev[\ops_{\max}(k)|\MBN]-\Ev[\ops_{\max}^{\text{d}}(k)|\MBN]}{6^{-1}\Ev[\ops_{\max}(k)|\MBN]}\\
&\le \left(1+ O(\tfrac{k^2}{n})\right)\left( \frac{1}{\MBN } \frac{\gamma_1}{\nu_1^2}\frac{\nu_{k-1}}{(\nu_{k-1}-1)} + \frac{2k^4 \nu_{k-1}}{\nu_{k-1}-1}\frac{ \gamma_1^2}{\nu_{1}^4 \CL_n} \Big(1+ O\big(\frac{1}{\MBN } \frac{\gamma_1}{\nu_1^2}\big)\Big) \right)\\
&\quad+ 6\left(1- \prod_{j=1}^k \left(1-\frac{\nu_{j+1}}{\nu_j}\right)\right) \ea\ee
The term containing $\gamma_1^2/\nu_{1}^4 \CL_n$ is coming from the Taylor expansion of the exponential factor in the formula for $e_{k,n}$.
We only have to verify that the rhs of the previous display is tending to $0$. For this we need $\gamma_1/(\nu_1^2 \MBN) \to 0$ and also $\gamma_1^2/(\nu_1^4 \CL_n) \to 0$.
For the first term, note that $\MBN\ge D_{\max}^{\sss{(b,n)}}(\infty)$, since it counts the number of half-edges with maximal degree $D_{\max}^{\sss{(b,n)}}(\infty)$. Further, $\gamma_1/\nu_1^2= \wit u_{\lfloor t_b+\la -T_r\rfloor}=o( D_{\max}^{\sss{(b,n)}}(\infty))$, since it is not hard to see that at time $t_b+\la$, the degree above which red occupies everything (i.e., $\wit u_{\lfloor t_b+\la -T_r\rfloor}$) is already less than $D_{\max}^{\sss{(b,n)}}(\infty)$,
 otherwise blue could have still increased its maximal degree at $t_b+\la$ by an extra jump.
  (Technically, this was the definition of $t_b$.
 Alternatively, compare the exact values of $D_{\max}^{\sss{(b,n)}}(\infty)$ in \eqref{eq::maxdegreeatinfty} and \eqref{def::dn}, and compare it to that
 of $\wit u_{\lfloor t_b+\la -T_r\rfloor}$, which can be derived from \eqref{eq::logul*} by adding the appropriate number of $(\tau-2)$ factors in the exponent corresponding to the five different cases. \
 This calculation is left to the reader.)

Similarly,  the second term, $\gamma_1^2/(\nu_1^4 \CL_n)=\wit u_{\lfloor t_b+\la -T_r\rfloor}^2/\CL_n$ is less than of order  $D_{\max}^{\sss{(b,n)}}(\infty)^2/n $ and hence is small as long as $D_{\max}^{\sss{(b,n)}}(\infty)=o(\sqrt{n})$. Note that this is the case by Theorem \ref{thm::maxdegreeblue} since $\la>1$.

Finally, we show that the last term in \eqref{eq::J} is also small. Since $\la>1$, $[t_b+\la (j+1)-T_r] \ge [t_b+\la j-T_r] +1$, and $\nu_i\in (c_1, C_1)\times \wit u_{[t_b+ \la i -T_r]}^{3-\tau}$ hence the last term is less than 6 times
\[ \sum_{j=1}^k \frac{\nu_{j+1}}{\nu_{j}} \le \sum_{\ell=[t_b+\la -T_r]}^{[t_b+\la k -T_r] }\frac{C_1}{c_1} \left(\frac{\wit u_{\ell+1}}{\wit u_\ell}\right)^{3-\tau} \le \sum_{\ell=[t_b+\la -T_r]}^{[t_b+\la k -T_r]} \frac{C_1(C\log n)^{3-\tau}}{c_1\wit u_\ell},\]
where we have used the recursion $\wit u_{\ell+1}=C\log n\,\wit u_{\ell}^{2-\tau}$ in \eqref{eq::wideui_recursion}.
Again, by the same recursion, for some large enough constant $C'$, the sum on the rhs is at most
\[ \frac{C_1(C\log n)^{3-\tau}}{c_1}\frac{C'}{\wit u_{[t_b+\la k -T_r]}},\]
which is small as long as $\log \wit u_{[t_b+\la k -T_r]}$ is of larger order than $\log ((C\log n)^{3-\tau})$. Note that this holds for an appropriate choice of $k$, since using \eqref{eq::logul*} and the recursion for $\wit u_\ell$ again, $\log \wit u_{[t_b+\la k -T_r]}$ is of order
\[\log \wit u_{[t_b+\la k -T_r]}= O((\log n)^{2/(1+\la)} (\tau-2)^k). \]
Note that if we now pick $k=o(\log \log n)$, then the exponent $(\log n)^{2/(1+\la)}$ stays unchanged and the expression is much larger order than $\log\log n$.
\end{proof}

Recall that $\CA_i, A_i$ stands for the set and number of blue vertices in layer $\Gamma_{i}^{\sss{(b)}}$ at the time when blue reaches the layer -- at time $\la [t(n_\rho')/\la]+\la i$.  Also recall that $i_{\max}$ stands for the index of the last $\Gamma_{i}^{\sss{(b)}}$ layer ever reached by blue, see \eqref{def::imax}. Further, $\ops_{-z}(k)$ is the number of vertices explored via a path of length $k$ starting from a half-edge in $\CA_{i_{\max}-z}$ that are not explored via a half-edge from $\CA_{i_{\max}-z+1}$.
Next we show that $\sum_{z\ge 0}\ops_{-z}(k)$ is at most the same order of magnitude as  $\ops_{\max}(k)$:
\begin{lemma}\label{lem::opt-z-lemma} With the notation introduced before,
\[ \log \Big(\sum_{z\ge 0}\ops_{-z}(k) \Big)\le \log (\ops_{\max}(k)) (1+o_{\Pv}(1)). \]
\end{lemma}

\begin{proof}
Let us denote the number of half-edges in $\CA_{i_{\max}-z}$ that are \emph{not connected} directly to $\Gamma_{i_{\max}-z+1}^{\sss{(b)}}$ by $H_{-z}$. From Lemma \ref{lem::numberofverticesinGamma} we have a bound on the  number of vertices $A_i$ in layer $\Gamma_i^{\sss{(b)}}$, and Lemma \ref{lem::badpaths} says that the maximal degree in $\CA_{i}$ is at most $\wih u_i^{\sss{(b)}}$ whp.

First, let us describe the following construction of the blue cluster spreading through the layers $\Gamma_i^{\sss{(b)}}.$ After an extra time unit $\la$,  $A_{i+1}$ half-edges out of the at most $A_i \wih u_i^{(b)}$  half-edges of blue are connected to half-edges in $\Gamma_{i+1}^{\sss{(b)}}$, while the other half-edges are not.
In the construction of $\CMD$ in Section \ref{sc::BP}, each half-edge is paired to a uniformly chosen other half-edge. The uniform distribution restricted to a set is still uniform on that set,
thus we can think of this procedure by picking  $A_{i+1}$ many of the half-edges out of the at most
 $A_i \wih u_i$ half-edges uniformly at random and connecting them to uniformly chosen half-edges in $\Gamma_{i+1}^{\sss{(b)}}$.

 The rest of the half-edges in $\Gamma_{i}^{\sss{(b)}}$ are connected to lower degree vertices, i.e., we can simply pair these half-edges to lower degree vertices than $u^{\sss{(b)}}_{i+1}$, and apply the path counting method similar as for $\CMBN$ in Lemma \ref{lem::wihZk},
 with the restriction that the degree of the $j$-th vertex on such a path must be less than the degree in $\Gamma_{i_{\max}-z+j}^{\sss{(b)}}$ if $j\le z$ and less than the degree where the red avalanche is at the current time when $j>z$, respectively. The restriction for $j\le z$ is needed to avoid double counting.

Clearly, $H_{-z}\le A_{i_{\max}-z} \wih u_{i_{\max-z}}$.
Then the degree truncation for this process at $\la j$ time unit later is at $u_{i_{\max}-z+j}^{\sss{(b)}}$ if $j\le z$ and $\wit u_{ t_b+\la (j-z) - T_r   }$ if $j>z$.

A simple modification of Lemma \ref{lem::wihZk} gives the  number of vertices found from these half-edges. Moreover, to show that vertices reached from $\CA_{i_{\max}-z}$, for $z\ge1$ are of less order than that reached via $\CMBN$, we can use Markov's inequality:

Similarly as in \eqref{eq::prodnuj1},
\be\label{eq::log0} \log \ops_{-z}(k)\le \log \left(A_{i_{\max}-z}\wih u_{i_{\max}-z}\right) + \sum_{j=1}^{z}\log ((\wih u_{i_{\max-z+j}})^{3-\tau})+\sum_{j=1}^{k-z-1} \log \nu_{j}, \ee
where the exponent $3-\tau$ comes from a similar calculation than that in \eqref{eq::nujbound}.
We claim that the maximum of this quantity is at $z=0$.

Since $\log \ops_{\max}(k)= \log \MBN +\sum_{j=1}^{k-1} \log \nu_{j}$, to prove that $\ops_{-z}(k)\le \ops_{\max}(k-z)$, we need to show that the sum of the first two terms in \eqref{eq::log0} are less than $\log \MBN$.

By the recursive definition of $\wih u_i$ in \eqref{eq::uibar}, $\log \wih u_{i_{\max-z}} = (\tau-2)^z\log\wih u_{i_{\max}}(1+o(1))$.
We can also use the fact from Lemma \ref{lem::numberofverticesinGamma} that $A_{i}=o(\wih u_i)$ for any $i\le i_{\max}$. Hence
\be\label{eq::log1}\log \left(A_{i_{\max}-z}\wih u_{i_{\max}-z}\right)= (\tau-2)^z\log (\wih u_{i_{\max}})(1+o(1)) \text{ whp } \ee
and the second term in \eqref{eq::log0} is $(1+o(1))$ times
\be\label{eq::log2}  \sum_{j=-z+1}^0 (3-\tau)\log(\wih u_{i_{\max}-z+j}) =  \log(\wih u_{i_{\max}}) (3-\tau)\sum_{j=0}^{z-1}(\tau-2)^j=\log(\wih u_{i_{\max}}) \left(1-(\tau-2)^z\right).  \ee
We see that the sum of the right hand sides of \eqref{eq::log1} and \eqref{eq::log2}
 is exactly $\log(\wih u_{i_{\max}})$.
 Thus, returning to \eqref{eq::log0},
 \[\log \ops_{-z}(k) \le \log \wih u_{i_{\max}} + \sum_{j=1}^{k-z-1} \log \nu_j  + o(1).\]
 The right hand side is indeed maximal for $z=0$, for which we have
\[ \log \ops_{-0}(k) \le \log \wih u_{i_{\max}} + \sum_{j=1}^{k-1} \log \nu_j  + o(1).\]
Compare this quantity to $\log \ops_{\max}(k)$ in Corollary \ref{cor::chebisev}.
Since $\wih u_{i_{\max}}(1+o(1))\le D_{\max}^{\sss{(b,n)}}(\infty)\le \MBN$, this finishes the proof of Lemma \ref{lem::opt-z-lemma}, since \[\sum_{z=0}^{O(\log\log n) } \ops_{-z}(k)\le O(\log\log n) \max_z \ops_{-z}(k),\] and the $\log\log n$ factor becomes a negligible additive term when taking logarithm.
\end{proof}

Having analysed the size of the optional cluster of blue,  we are ready to finish the upper bound of Theorem \ref{thm::main} by combining the previous results.
\begin{proof}[Proof of the upper bound in Theorem \ref{thm::main}]
First, fix $k=k(n)\to \infty$ so that $k(n)=o(\log\log n)$.
Then, Lemma \ref{lem::opt-z-lemma} implies that the logarithm of the  total number of vertices that blue paints in the last phase is at most $\log \ops_{\max}(k) (1+o_{\Pv}(1))$. Corollary \ref{cor::chebisev} says that the order of magnitude of $\log (\ops_{\max}(k))=\log \MBN + \sum_{j=1}^{k-1}\log  \nu_j + o_{\Pv}(1)$, where $\MBN$ is the number of blue half-edges in the highest layer that blue can reach.
Further, Lemma \ref{lem::verticeswithmaxdegree} determines the order of magnitude of $\log \MBN$, which is
\be\label{eq::order-mbn}\log \MBN=\left(\Ybnla/\Yrn\right)^{1/(\la+1)} (\log n)^{\frac{2}{\la+1}} C_n^{\text{half-edge}}(\Yrn, \Ybn) (1+o_{\Pv}(1))\ee and hence converges in distribution to $(Y_b^\la/Y_r)^{1/(\la+1)}$ when divided by the second two factors.

Thus, to get the asymptotic of $\log (\ops_{\max}(k))$, it remains to calculate  $\sum_{j=1}^k \log \nu_j$ and compare it to the order of $\log \MBN$. For this recall  the definitions of $\nu_j$ in \eqref{def::nuj}, $t_b$ in \eqref{def::t_b}, $t_c$ in \eqref{eq::tc}, $\wit u_\ell$ in \eqref{eq::ul}, and the upper bound on $\nu_j$ in \eqref{eq::nujbound}.
With $\de_j:=\{t_b-T_r+\la j\}$,
 \begin{align} \label{eq::prodnuj1}\sum_{j=1}^{k} \log \nu_j \le& \sum_{j=1}^k \log\left( C_1 \left(\wit u_{\lfloor t_b+\la j-T_r\rfloor}\right)^{3-\tau}\right)\\
  \le& \sum_{j=1}^{k} \left\{\left(\alpha\log n \!+\! \beta \log (C\log n)\right)(\tau-2)^{-1+t_b-T_r+\la j-\de_j}  (3\!-\!\tau)\right\} \!+\!k \log  (C_1 C\log n).\nonumber \end{align}
  We rewrite $t_b-T_r=t_c + \la \left(\ind_{\CE} - \left\{\frac{T_r+t_c}{\la}\right\}\right)$ in the exponent using \eqref{def::t_b}, and then 
use formula \eqref{eq::tc} to see that $t_c=O^\diamond(\frac{\la-1}{\la (\la+1)})$. Hence the main order term in $(\tau-2)^{t_c}$ is $(\log n)^{(\la-1)/(\la(\la+1))}$. This implies that the two smaller order terms $k \log  (C_1 C\log n)$ and $\beta \log (C\log n)(\tau-2)^{t_c}$ are  $o( \log n (\tau-2)^{t_c})$ and can be put in a $(1+o_\Pv(1))$ factor of the main term.
Using the exact value of $t_c$ in \eqref{eq::tc} we obtain then
\be\ba\label{eq::orderoflognuj} \sum_{j=1}^{k} \log \nu_j\le&\ (1+o_{\Pv}(1))\cdot  (\log n)^{\frac{2}{\la+1}} \!\! \left(\frac{\Ybnla}{\Yrn}\right)^{1/(\la+1)}\\
&\ \times\left(\frac{\alpha}{\tau-1}\right)^{\frac{1}{\la+1}}\!\!\!\cdot\!\!\!\! \left(\sum_{j=1}^k (\tau-2)^{\la j-\de_j}(3-\tau)\right)
(\tau-2)^{\frac{-\la+b_n^{(r)}}{\la+1}+ \la \left(\ind_{\CE} - \left\{\frac{T_r+t_c}{\la}\right\}\right) }.
\ea\ee
Note that we have used the formula in \eqref{eq::ul} for $\wit u_\ell$, but this is valid only until $\ell< O^\diamond(1-\ve)$ steps, so
$k=k(n)$ must satisfy $t_b-T_r+\la k \le O^\diamond(1-\ve)$. Again, by \eqref{def::t_b}, $t_b-T_r=O^\diamond(\frac{\la-1}{\la (\la+1)})$, hence any $k:=k(n)\to \infty$ with $k(n)=o(\log\log n)$ is still a good choice.

Similarly as in \eqref{eq::order-mbn}, we get
\be\label{eq::o-nuj}
\sum_{j=1}^{k} \log \nu_j \le \left(\Ybnla/\Yrn\right)^{1/(\la+1)} (\log n)^{\frac{2}{\la+1}} C^{\text{paths,k}}_n( \Yrn, \Ybn),\ee
where we introduce
 \be \ba \label{def::dnY} C^{\text{paths,k}}_n( \Yrn, \Ybn):=&\left(\frac{\tau-1-(\tau-2)^{b_n^{(r)}}}{(\tau-1)^2}\right)^{\frac{1}{\la+1}}\!\!\!\! (\tau-2)^{\frac{-\la+b_n^{(r)}}{\la+1}+ \la \left(\ind_{\CE} - \left\{\frac{T_r+t_c}{\la}\right\}\right) }\\
&\quad\times (3-\tau)\sum_{j=1}^{k} (\tau-2)^{\la j-\de_j}, \ea\ee
and where we used that $\alpha=1-(\tau-2)^{b_n^{\sss(r)}}/(\tau-1)$, see \eqref{def::alpha}, where $b_n$ is replaced by $b_n^{\sss(r)}$.
Let \be \label{def::Cpath}C^{\text{paths}}_n( \Yrn, \Ybn):=\lim_{k\to \infty} C^{\text{paths,k}}_n( \Yrn, \Ybn).\ee Note that in this formula, $\delta_j=\{t_b-T_r+\la j\}\in[0,1)$, where both $t_b,T_r$ are integers. This implies that if $\la = p$ for some $p\ge 2$ integer, then all $\de_j\equiv 0$ and the last sum can be carried out and tends to $1/(1-(\tau-2)^\la).$ Otherwise, the sum is less than $1/(1-(\tau-2)^\la).$ 

  Now, recall again that $\log \ops_{\max}(k)=\log \MBN+ \sum_{j=1}^k \nu_j+ o_{\Pv}(1)$ by Corollary \ref{cor::chebisev}, and combine \eqref{eq::o-nuj} with \eqref{eq::order-mbn}
  \[  \frac{\log \ops_{\max}(k) }{
(\log n)^{\frac{2}{\la+1}} \left(  C_n^{\text{half-edge}}(\Yrn, \Ybn) + C^{\text{paths}}_n( \Yrn, \Ybn)\right)} \le \left( \frac{\Ybnla}{\Yrn}\right)^{\frac{1}{\la+1}}+ o_{\Pv}(1). \]
Now we can finally use $\Yrn\toindis Y_r, \ \Ybn\toindis Y_b$ (by Theorem \ref{thm::davies}
). Hence, the right hand side converges to $(Y_b^\la/Y_r)^{1/(\la+1)}$.

Let us denote
\be\label{def::norm-constant}C_n(\Yrn,\Ybn):= C_n^{\text{half-edge}}(\Yrn, \Ybn) + C^{\text{paths}}_n( \Yrn, \Ybn). \ee Some elementary calculations -- optimizing the fractional parts appearing in the exponents -- yield the following bounds:
\be\label{eq::tight-bounds}  \left(\!\frac{(\tau-2)(3-\tau)}{(\tau-1)^{2}\left(1-(\tau-2)^\la\!\right)}\right)^{\frac{1}{\la+1}}\!\!\! \frac{\la+1}{\la^{\frac{\la}{\la+1}}}\! \le C_n(\Yrn,\Ybn) \le
\frac{(\tau-2)^{-\frac{2\la+1}{\la+1}}}{4^{\frac{1}{\la+1}}}\left(\!1+\frac{3-\tau}{1-(\tau-2)^{\la}}\!\right)\!.  \ee
The bounds given in Theorem \ref{thm::main} are the simplified versions of \eqref{eq::tight-bounds}.
\end{proof}
For the lower bound of Theorem \ref{thm::main} we need to show that most of the optional cluster of blue is actually going to be occupied by blue. For this, let us introduce the notation $\opt(k):=\bigcup_{z\ge0} \opt_{-z}(k) \cup \opt_{\max}(k)$, and set $\opss(k):=|\opt(k)|$,
where $k$ stands for the length of the paths we are counting.
The next lemma shows that essentially all the vertices in $\opt(k)$ for some $k=k_n=o(\log\log n)$ will indeed be painted blue, i.e., red cannot accidentally bite out too much from this set.

\begin{lemma}\label{lem::red-blue-intersection}
Set $k=k(n)=o(\log\log n)$.
The number of vertices in the intersection of $\CR_{t_b+\la k}$ and $\opt(k)$ is small, i.e.,
\[ |\opt(k)\cap  \CR_{t_b+\la k}|= o_{\Pv}\left(\ops_{\max}(k)\right),\]
hence \[ |\opt(k)\setminus(\opt(k)\cap  \CR_{t_b+\la k})|=\ops_{\max}(k)(1+o_{\Pv}(1)).\]
\end{lemma}
The proof of the lemma will follow from the following claim:

\begin{claim} Suppose $S\subset \CMD$ is an arbitrary subset of vertices such that
 $|S|=o(n)$ but $|S|:= s_n \to \infty$. Then with high probability, the total number half-edges that point out of $S$, $H(S)$ is at most
 $|S|^{(\tau-2)/(\tau-1)} n^{1/(\tau-1)}$. Hence, $H(S)=o(n)$ holds also.
\end{claim}

\begin{proof}
In the worst case scenario, $S$ contains the $|S|$ highest degree-vertices. Order the degrees $D_1, D_2, \dots, D_n$ in $\CMD$ of vertices in non-increasing order:  $D^{(1)}\le  D^{(2)}\le \dots D^{(n)}$. Then with an arbitrarily chosen $K_n>0$ the following bound always hold:
\be\label{eq::half-edge-bound} \sum_{v \in S} d_v \le \sum_{j=1}^{|S|} D^{(j)} \ind_{\{D^{(j)}< K_n\}} + \sum_{i=1}^n D_i \ind_{\{D_i>K_n\}}.\ee
By Lemma \ref{claim::Sbound}, the second term is whp at most $n K_n^{2-\tau}$, with the error probability $\exp\{-c n K_n^{1-\tau}\}$ being small as long as $K_n=o(n^{1/(\tau-1)})$.
Hence the rhs of \eqref{eq::half-edge-bound} is whp at most
 \[ \sum_{v \in S} d_v \le |S| K_n + n K_n^{2-\tau}. \]
For the right hand side to be minimal we set $K_n:=((\tau-2)n/|S|)^{1/(\tau-1)}$, which is $o(n^{1/(\tau-1)})$ as long as $|S|\to \infty$ with $n$. With this choice of $K_n$,
 \[ \sum_{v \in S} d_v \le |S|^{\frac{\tau-2}{\tau-1}} n^{\frac{1}{\tau-1}}. \]
Since the exponents sum up to $1$, the rhs is always $o(n)$ if $|S|=o(n)$.
\end{proof}

\begin{proof}[Proof of Lemma \ref{lem::red-blue-intersection}]
Note that we can construct the configuration model by pairing the half-edges in an arbitrarily chosen order. This enables the joint construction of the graph and the spread of the red and blue cluster. Hence, we can assume that if a vertex is not yet colored, its half-edges are still free, and we do not have to take into account the effect that whole paths can be blocked away from one color by the other color by painting one or a few vertices only.

Fix the length of the blue exploration path $k$. For a set $S$ of vertices, we denote by $H(S)$ the total number of half-edges that point out of the set $S$. As a lower bound, we can use the adapted rule that whenever red and blue arrives at a vertex at the same time, it is going to be red deterministically. We can further assume that if this is the case, i.e., there are simultaneous jumps of red and blue, then we always pair the red half-edges first, i.e., when pairing the blue half-edges at time $t_b+\la i$, we consider $\CR_{t_b+\la i}$ as already determined.

Let us consider a path $\pi$ ending in $\CO(k)$ given by the sequence of half-edges and vertices $(\pi_0,  s_0,  t_1,  \pi_1,$
$s_1, \dots, t_k, \pi_k)$, that is, $s_i$ is the half-edge pointing out of vertex $\pi_i$ that we pair to $t_{i+1}$, a half-edge belonging to vertex $\pi_{i+1}$. We call this path \emph{thinned at step $i$} if the half-edge $s_{i-1}$ is paired to the half-edge $t_i$ where $\pi_i$ is already red, i.e. $\pi_i \in \CR_{t_b+\la i}$. We call a path \emph{thinned} if it is thinned at some $i\le k$.

Clearly, each time we pair a blue half-edge at time $t_b+\la i$, it is with probability $H(\CR_{t_b+\lambda i})/\CL_n(1+o(1))$ paired to a red half-edge. Let us denote $\sigma_k:=\{\CR_{t_b+\lambda i}\}_{i=1}^k$. Hence, the probability that a particular path ending in $\CO(k)$ to be thinned can be bounded using a union bound
\[ \Pv\left( (\pi_0,  s_0,  t_1,  \pi_1,s_1, \dots, t_k, \pi_k)\in \CO(k) \text{ thinned }\big| \sigma_k \right) \le  \sum_{i=1}^k\frac{H(\CR_{t_b+\la i})}{\CL_n}=:p_{\text{th,k}},   \]
as long as $k=k(n)$ is so that the quantity on the rhs is less then $1$.
Hence, for any function $\delta_{n,k}$ so that $\delta_{n,k} p_\text{th,k}<1$, the proportion of vertices in $\CO(k)$ that are thinned - denoted by $\CO_{\text{th}}(k)$ - by Markov's inequality is at most
\be\label{eq::thinned-vertices} \Pv\left( \frac{|\CO_{\text{th}}(k)|}{|\CO(k)|} \ge  \delta_{n,k} p_{\text{th},k}   \Big| \sigma_k , \CO(k)\right) \le  \frac{\Ev[\frac{1}{\CO(k)}\sum_{\pi \in \CO(k)} \ind_{\{ \pi \text{ thinned }\}}| \sigma_k, \CO(k)]}{\delta_{n,k} p_{\text{th},k}} \le \frac{1}{\delta_{n,k}}. \ee
Now, note that  we are done with the lower bound if we can pick a $k=k(n)\to \infty$ and an $\delta_{n,k}$ so that
$\delta_{n,k}\to \infty$ and  $\delta_{n,k} p_{\text{th,k}}<1$.
 
 For this, let us temporarily believe that $\widetilde k:=\widetilde k(n)=\log \log \log n$ has the property that  $\CR_{t_b+\la \widetilde k}  =o_{\Pv}(n)$. Then, let us write $\CR_{t_b+\la \widetilde k}:= O(n/\omega_{n,\widetilde k})$ where $\omega_{n,\widetilde k}\to \infty$ with $n\to \infty$.
Set \be\label{eq::kdef} k:=k(n)=\min\{ \log\log\log n, (\omega_{n,\widetilde k})^{\frac{\tau-2}{2(\tau-1)}}\}.\ee
 Clearly, $k\le \widetilde k$ holds, hence, by monotonicity we have $\CR_{t_b+\la k}\le \CR_{t_b+\la \widetilde k}.$ Applying Claim \ref{claim::Sbound} on each term in the sum,
\[ p_{\text{th},k}=\sum_{i=1}^k\frac{H(\CR_{t_b+\la i})}{\CL_n} \le \frac{k\, n}{\CL_n (\omega_{n, \widetilde k})^{(\tau-2)/(\tau-1)} }.  \]
On the event $\{\CL_n \in (1/2 \Ev[D] n, 2 \Ev[D] n)\}$, using \eqref{eq::kdef},  
\[ p_{\text{th},k} \le  \frac{2}{\Ev[D]} \Big(\frac{1}{\omega_{n,\widetilde k}}\Big)^{\frac{\tau-2}{2(\tau-1)}}.\]
This allows us to pick $\delta_{n,k}:=(\omega_{n, \widetilde k})^\frac{\tau-2}{4(\tau-1)}$, and then 
$\delta_{n,k}p_{\text{th},k}\to 0$ as well as $\delta_{n,k}\to \infty$ holds with $n\to\infty$.
As a result, the rhs of \eqref{eq::thinned-vertices} tends to zero, showing that whp, only a negligible fraction of the vertices in $\CO(k)$ will be thinned.

We are left showing that with $\widetilde k= \widetilde k(n)=\log\log \log n$, we have $|\CR_{t_b+\la k} | = o_{\Pv}(n)$. One way to see this is to use \cite[Theorem 1.2]{HHZ07}  about typical distances: typical distances in the graph are $2 \log\log n / |\log (\tau-2)|$ with bounded fluctuations around this value, while $t_b+\la \widetilde k < (1+\ve)2\la/(\la+1) \log\log n/ |\log (\tau-2)|.$ Hence, the number of vertices at most $t_b+\la \widetilde k$ away from the uniformly chosen red source vertex must be $o(n)$.

To keep the paper self-contained, we provide another proof of this fact here.
For this, note that $t_b+\la \widetilde k$, $t_b$ defined in \eqref{def::t_b} is at most $\frac{(1+\ve)\la}{\la+1} \frac{2\log\log n}{|\log(\tau-2)|}$ for some $\ve>0$ whp. To estimate the expected size of the red cluster, we write
\[ \ba  \Ev\Big[|\CR_{t_b+\la \widetilde k}|\Big] &\le 
n\, \Pv\left( \mathcal D(\CR_0,v) \le t_b+\la k  \right)\\
&\le n\, \Pv\left( \mathcal D(\CR_0,v)/2 \le  \frac{(1+\ve)\la}{\la+1} \frac{\log\log n}{|\log(\tau-2)|} \right) \ea\]
Now, using that $\mathcal D(\CR_0,v)/2$ has the distribution ${t_b\big|}_{\la =1}$, we can continue the bound as
\[\ba \Ev\left[|\CR_{t_b+\la \widetilde k}|/n\right] &\le \Pv\left({t_b\big|}_{\la =1} \le \frac{(1+\ve)\la}{\la+1} \frac{\log\log n}{|\log(\tau-2)|} \right)\\
& \le\Pv\left(\log (\Yrn\Ybn) \ge \frac{1-2\ve\la}{\la+1} \log\log n\right).\ea\]
 In the last line, we used \eqref{def::t_b} with $\la=1$ and put bounded terms there in the $(1-2\ve \la)$ factor on the rhs inside the probability sign. Further, note that the random variables $\Yrn\Ybn$ have asymptotically exponential tails. Hence, the probability is tending to zero as $n\to \infty$.
 This ensures that most vertices are further away from the source of the red infection than $t_b+\la k$ and hence $\CR_{t_b+\la \widetilde k} =o(n)$. This finishes the proof of the lemma.
\end{proof}

\begin{proof}[Proof of the lower bound in Theorem \ref{thm::main}]
First, note that the time $T_r$ for red to reach the top of the mountain was a lower bound, i.e., we have shown the existence of a path that reaches the top in time $T_r$ whp in Lemma \ref{lem::gamma_i_connectivity}. Clearly, if red reaches the top earlier, then there is less time for blue to increase its degree, hence, it will occupy fewer vertices. Fortunately, an adaptation of Lemma \ref{lem::badpaths} for red instead of blue shows that this cannot happen. That is, one can define the sequence $\widehat u_i^{\sss{(r)}}$ by the recursion
\[ u_0^{\sss{(r)}}:= (n^{\varrho''} C\log n)^{1/(\tau-2)}, \qquad u_i^{\sss{(r)}}:=(u_0^{\sss{(r)}} C\log n)^{1/(\tau-2)}, \]
and then exchange every superscript $(b)$ to $(r)$ in the definition of $\mathcal B{ad} \mathcal P_k$ (see right before Lemma \ref{lem::badpaths}).  Applying Lemma \ref{lem::badpaths} yields that with high probability, red cannot jump a layer ahead, and hence the time to reach the top remains as defined in $T_r$.

Next, everything from this point on was a concentrated estimate, hence, we only need to check what happens in the last phase, how many vertices blue can actually get from its optional cluster.

Using Lemma \ref{lem::red-blue-intersection}, we see that the log-size of the blue cluster at time $t_b+\la k$ is whp
$\log \ops(k)= \log \ops_{\max}(k)(1+o_{\Pv}(1)).$ Note also that by Corollary \ref{cor::chebisev}, $\log \ops_{\max}(k)$ is concentrated and is equal to $\log \MBN + \sum_{j=1}^k \log \nu_j+o_{\Pv}(1)$ by Lemma \ref{lem::verticeswithmaxdegree}.
Hence, it only remains to give a lower bound on $\sum_{j=1}^k \log\nu_j$.
For this, note that the lower bound on $\nu_j$ is the same as the upper bound, with a factor $C_1$ replaced by $c_1$. This factor becomes an additive term when taking the logarithm, and hence contributing only inside the $o_\Pv(1)$ factor. Hence,
\[ \frac{\log \ops_{\max}(k) }{
(\log n)^{\frac{2}{\la+1}} \left(  C_n^{\text{half-edge}}(\Yrn, \Ybn) + C^{\text{paths}}_n( \Yrn, \Ybn)\right)} \ge \left( \frac{\Ybnla}{\Yrn}\right)^{\frac{1}{\la+1}}+o_{\Pv}(1). \]
The right hand side converges to $(Y_b^\la/Y_r)^{1/(\la+1)}$.
Combining this with the upper bound completes the proof of the Theorem \ref{thm::main}.
\end{proof}
\section{Acknowledgement}
The work of EB, RvdH and JK was supported in part by the Netherlands Organisation for Scientific Research (NWO) through VICI grant 639.033.806.

The work of JK was supported in part by NWO through the STAR cluster, and the work of RvdH was supported in part by NWO through Gravitation grant 024.002.003. JK thanks the Probability Group at The University of British Columbia for their hospitality while completing the project.

\appendix
\section{Path counting methods for restricted paths}\label{app::paths}
\begin{proof}[Proof of Lemma \ref{lem::badpaths}]
Here we follow the notation of \cite[Section 10.4.2]{H10} as much as we can.
 We will use union bound and Markov's inequality to bound the probability of the existence of bad paths:
 \be\label{eq::badpathunion} \Pv(\exists k\le k_0, \CB ad\CP_k \neq \varnothing ) \le \sum_{k=1}^{k_0} \Ev[|\CB ad\CP_i|].\ee
First we give an upper bound on the expected number of bad paths conditioned on the degree sequence, so let us fix the degrees first and write $d_v$ for the degree of the vertex $v$.
A (directed) path of length $k$ from vertex $a$ to some vertex $\pi_k$ can be described as
\be\label{def::path} \left\{(\pi_0,s_0), (\pi_1, t_1,s_1), \dots, (\pi_{k-1}, t_{k-1}, s_{k-1}), (\pi_k, t_k)\right\},\ee
where $\pi_i\in [n]$ is the $i$-th mid-vertex along the path, $s_i\in [d_{\pi_i}]$ denotes the label of the outgoing and $t_i\in [d_{\pi_i}]$ the label of the incoming half-edge of $\pi_i$.
Recall that we call a path \emph{good} if $\deg(\pi_i)\le \widehat u_i^{\sss{(b)}}$ for all $0\le i \le k$, and $\CB ad\CP_k$ is a subset of bad paths with $\pi_0 \in Z_{[t(n^{\vr'})/\la]}$,  $d_{\pi_i}\le \widehat u_{i}^{\sss{(b)}}$ for all $0\le i \le k-1$ but $d_{\pi_k}> \widehat u_k^{\sss{(b)}}$.

Since the number of half-edges out of vertex $\pi_i$ is $d_{\pi_i}$, there are many possible paths via the vertices $(\pi_i)_{i=0}^k$. Thus, the expected number of paths through fixed vertices $\pi_0, \dots, \pi_k$ equals the probability that a given path in \eqref{def::path} is present in $\CMD$ multiplied by the combinatorial factor of picking the possible half-edges for the paths, i.e.,
\be\label{eq::expectedpath1} \prod_{i=1}^k\frac{1}{\CL_n^\ast-2i+1}\cdot d_{\pi_0} \left(\prod_{i=1}^{k-1}d_{\pi_i} (d_{\pi_i}-1)\right) d_{\pi_k},\ee
where $\CL_n^\ast$ is the number of free half-edges when the procedure starts.
Thus, the expected number of all self-avoiding \emph{bad paths in $\CB ad\CP_k$} equals
\be\label{eq::badpath1} \Ev_n[|\CB ad\CP_k|]=\!\!\!\!\!\!\sum_{\pi_0 \in Z_{[t(n^{\vr'})/\la]} }\!\!\!\!\!\!\!d_{\pi_0}\!\!  \left(\prod_{i=1}^{k}\frac{\CL_n^\ast}{\CL_n^\ast-2i+1}\right)\cdot\sideset{}{^\ast}\sum_{\substack{ \pi_1,\dots,\pi_{k-1} \\ \forall i\ \pi_i \le \widehat u_i^{\sss{(b)}}}} \left(\prod_{i=1}^{k-1}\frac{d_{\pi_i} (d_{\pi_i}-1)}{\CL_n^\ast}\right) \Bigg(\sum_{{\substack{ \pi_k \\ d_{\pi_k} \ge \widehat u_k^{\sss{(b)}}}}}\frac{d_{\pi_k}}{\CL_n^\ast}\Bigg),\ee
where $\sideset{}{^\ast}\sum$ means that we sum over distinct vertices.
Allowing non-distinct vertices, we get the upper bound
\be\label{eq::badpath2}\Ev_n[|\CB ad\CP_k|]\le \!\!\!\!\!\sum_{\pi_0 \in Z_{[t(n^{\vr'})/\la]} }\!\!\!\!d_{\pi_0}  e^{\frac{k^2}{\CL_n^\ast}} \prod_{i=1}^{k-1}\Bigg(\sum_{d_{\pi_i}\le \widehat u_i^{\sss{(b)}}}  \frac{d_{\pi_i} (d_{\pi_i}-1)}{\CL_n^\ast}\Bigg) \Bigg(\sum_{{\substack{ \pi_k \\ d_{\pi_k} \ge \widehat u_k^{\sss{(b)}}}}}\frac{d_{\pi_k}}{\CL_n^\ast}\Bigg),\ee
where the factor $e^{\frac{k^2}{\CL_n^\ast}}$ is a bound on the term $\prod_{i=1}^{k}\CL_n^\ast/(\CL_n^\ast-2i+1)$ above.
 Since the path counting starts at time $t(n^{\vr'})=O^\diamond(1)$, and typical distances are $O^\diamond(2)$ in the graph, we have $\CL_n^\ast=\CL_n(1+o(1))$ whp (see the proof of Lemma \ref{lem::red-blue-intersection} for more details). By the Law of Large Numbers, $\CL_n/n\to \Ev[D]$, hence the $i$-th factor on the right hand side is close to
 \be\label{eq::llln-down} \frac{1}{\Ev[D]}\Ev\left[D(D-1)\indicator{D\le\widehat u_i^{\sss{(b)}} }\right]\le C_3 (\widehat u_i^{\sss{(b)}})^{3-\tau}, \ee
 while the last factor in \eqref{eq::badpath2} is close to
 \[ \frac{1}{\Ev[D]} \Ev\left[D \indicator{D\ge\widehat u_k^{\sss{(b)}} }\right] \le C_3 (\widehat u_k^{\sss{(b)}})^{2-\tau}\]
 by the tail behavior \eqref{eq::F} of the distribution function of $D$, for some constant $C_3$.

Since we need to set $k_0=O(\log\log n)$, the error term $\exp\{\frac{k^2}{\CL_n^\ast}\}=1+o(1)$ in \eqref{eq::badpath2} stays close to $1$. Since $\pi_0$ is a vertex that belongs to the last generation of the branching process approximation phase, we get an upper bound on the total number of bad paths by contracting all the vertices that belong to the last generation of the blue branching process. Note that by the coupling to the BP, the degrees in this generation are i.i.d. from distribution $F^\ast$ in \eqref{def::size-biased1}, hence for some large constants $C', C_2>0$ whp
\[  \sum_{\pi_0\in Z_{[t(n^{\vr'})/\la]}} d_{\pi_0}\le  (C' \log n Z_{[t(n^{\vr'})/\la]})^{1/(\tau-2)}\le C_2\widehat u_0^{\sss{(b)}},\]
by the definition of $\widehat u_0^{\sss{(b)}}$ in \eqref{eq::uibar}.
Further note that, with $m:=Z_{[t(n^{\vr'})/\la]}$ and the definition of $\widehat u_0^{\sss{(b)}}$ again, by Lemma \ref{lem::maxdegree},
\be\label{eq::dpi}  \Pv(\exists \pi_0\in Z_{[t(n^{\vr'})/\la]}, d_{\pi_0} > \widehat u_0^{\sss{(b)}} ) \le \frac{1}{m^{c_1C}}.  \ee 
Then with this error probability we can write
\be\label{eq::badpath3} \Ev_n[|\CB ad\CP_k|] \le  C_3 \widehat u_0^{\sss{(b)}}  (\widehat u_k^{\sss{(b)}})^{2-\tau} \prod_{i=1}^{k-1} \big(C_3(\widehat u_i^{\sss{(b)}})^{3-\tau} \big). \ee
The recursion for $\widehat u_i^{\sss{(b)}}$ in \eqref{eq::uibar} gives
\[\widehat u_i^{\sss{(b)}}= (\widehat u_0^{\sss{(b)}} )^{ (\tau-2)^{-i} } ( C\log n)^{ ( (\tau-2)^{-i}-1 ) / (3-\tau) }, \]
and then in \eqref{eq::badpath3}, after elementary calculation, the powers of $(\tau-2)^{-1}$ cancel in the exponent of $\widehat u_0^{\sss{(b)}}$ and $C \log n$, and the formula simplifies to
\[ \Ev_n[|\CB ad\CP_k|] \le (C\log n)^{-k}. \]
This estimate and \eqref{eq::dpi} together implies that the union bound in \eqref{eq::badpathunion} leads to\[ \Pv(\exists k\le k_0, \CB ad\CP_k \neq \varnothing ) \le \sum_{k=1}^\infty (C\log n)^{-k} + \Pv(\exists \pi_0\in Z_{[t(n^{\vr'})/\la]}, d_{\pi_0} > \widehat u_0^{\sss{(b)}} ) \le \frac{2}{ C\log n}.\]
This completes the proof of the upper bound in \eqref{eq::wihZk_expected}.
\end{proof}
\begin{proof}[Proof of Lemma \ref{lem::wihZk}]

Similarly as in the proof of Lemma \ref{lem::badpaths}, first we give upper and lower bound on the expected number of $k$-length paths starting from half-edges in the set $\CMBN$ conditioned on the degree sequence. Then, we relate the number of paths to the total number of vertices found in $\ops_{\max}(k)$ or $\ops_{\max}^{\text{d}}(k)$.

 Let us fix the degrees first and write $d_v$ for the degree of the vertex $v$. Let us also introduce the nested sets $\Lambda_1 \supset \Lambda_2 \supset \dots$-s by
\[ \Lambda_j=\left\{v\in [n]: d_v \le \wit u_{\lfloor t_b+\la j-T_r\rfloor}\right\}.\]
Similarly as in the previous lemma, a (directed) path of length $k$ from vertex $a=\pi_0$ to $b=\pi_k$ can be described as
\be\label{def::path2} \left\{ (\pi_0,s_0), (\pi_1, t_1,s_1), \dots, (\pi_{k-1}, t_{k-1}, s_{k-1}), (\pi_k, t_k)\right\}.\ee
We call a path now \emph{good} if $\pi_i\in \Lambda_i$ and \emph{good-directed} if $\pi_i\in \Lambda_i\setminus \Lambda_{i+1}$ for all $1\le i \le k$.

We write $N_k(a,b)$, $N_k^{\text{d}}(a,b)$ for the number of self-avoiding good paths and good-directed paths going from vertex $a$ to $b$, respectively, and $\CL_n^\ast$ for the total number of half-edges present in $\CMD$ at time $t_b$. Similarly as in \eqref{eq::badpath1}, the expected value of all self-avoiding \emph{good paths} equals
\be\label{eq::nkab} \Ev_n[N_k(a,b)]=\frac{d_a d_b}{\CL_n^\ast - 2k+1} \left(\prod_{i=1}^{k-1}\frac{\CL_n}{\CL_n^\ast-2i+1}\right)\cdot\sideset{}{^\ast}\sum_{\substack{ \pi_1,\dots,\pi_{k-1} \\ \forall i\ \pi_i \in \Lambda_i}} \left(\prod_{i=1}^{k-1}\frac{d_{\pi_i} (d_{\pi_i}-1)}{\CL_n}\right)\ee
where $\sideset{}{^\ast}\sum$ means that we sum over distinct vertices.
Now clearly we have the upper bound
\be\label{eq::nkab_upper1}\Ev_n[N_k(a,b)]\le d_a \frac{d_b}{\CL_n} \left(\frac{\CL_n}{\CL_n^\ast}\right)^{k}\!\!\! \cdot \e^{\frac{k^2}{\CL_n^\ast}}\cdot \prod_{i=1}^{k-1}\left(\sum_{\pi_i\in \Lambda_i} \frac{d_{\pi_i} (d_{\pi_i}-1)}{\CL_n}\right):=\overline{n}_k(a,b).\ee
Note that $\CL_n/\CL_n^\ast \to 1$ by the argument in Lemma \ref{lem::red-blue-intersection}.
The Law of Large Numbers ensures the convergence on the right hand side, so
\be\label{eq::nkab-conv}\Ev_n[N_k(a,b)] \le d_a \frac{d_b}{\CL_n} \Bigg(\prod_{i=1}^{k-1} \nu_i\Bigg) \cdot \e^{\frac{k^2}{\CL_n^\ast}}=:\overline{\Ev[N_k(a,b)]}\ee
 where $\nu_i$ is from \eqref{def::nuj}.
  By contracting all the vertices belonging to the set $\CMBN$, we have $d_a=\MBN$ and by letting $b$ be the contraction of all the vertices with degree less than $K$ for some arbitrary constant $K\ge 2$, we have $d_b/\CL_n^\ast\le 1$ is of constant order again.
  Note that the total number of explored vertices on paths of length $k$ is bounded from below and from above by
  \be\label{eq::obound} N_k(a,b)\le \ops_{\max}(k) \le \sum_{i=1}^k N_i(a,b).\ee
Noting that $\nu_i$ grows super-exponentially, $\sum_{i=1}^k\overline{\Ev[N_i(a,b)]}= \overline{\Ev[N_k(a,b)]} (1+o(1))$. This finishes the proof of the upper bound in \eqref{eq::wihZk_expected}.

We can get a lower bound on \eqref{eq::nkab} if in the sum over distinct vertices, we leave out the $i$ highest degree vertices $V_i:=\{v_{i,1}^{\max}, \dots, v_{i,i}^{\max}\}$ from each set $\Lambda_i$. That is, we have
\be\label{eq::nkab_under} \Ev_n[N_k(a,b)] \ge d_a \frac{d_b}{\CL_n} \prod_{i=1}^{k-1}\left(\sum_{\substack{\pi_i\in \Lambda_i\\ \pi_i\notin V_i}} \frac{d_{\pi_i} (d_{\pi_i}-1)}{\CL_n}\right):=\underline{n}_{k}(a,b). \ee
 Note that since we leave out only finitely many vertices, the $i$th sum within the product  still converges to $\nu_i$. Again contracting all the vertices belonging to the set $\CMBN$, we have $d_a=\MBN$ and by letting $b$ be the contraction of all the vertices with degree less than $K$ for some arbitrary constant $2\le K$, we have $d_b/\CL_n$ is of constant order again. Combining with the lower bound in \eqref{eq::obound} finishes the proof of the lower bound in \eqref{eq::wihZk_expected}.

The proof of the bounds \eqref{eq::wihZk_expected2} for \emph{good-directed} paths are analogous, but now one has to use the restricted sets
\[ \Lambda_j^{\text{d}}:= \Lambda_j\setminus \Lambda_{j+1}=\left\{v\in [n]:  \wit u_{\lfloor t_b+\la (j+1)-T_r\rfloor} < d_v \le \wit u_{\lfloor t_b+\la j-T_r\rfloor}\right\}.\]
Next we prove the variance formula for $\ops_{\max}^{\text{d}}(k)$ following more or less the lines of \cite[Section 10.4.2 and 9.4]{H10}. Note that the major difference between the proof of \cite[Proposition 9.17]{H10} and our case is that here we have the extra restriction $\pi_i, \rho_i \in \Lambda_i^{\text{d}}$, and the $\Lambda_i^{\text{d}}$ sets are disjoint.

First write  $N_k^\text{{d}}(a,b)$ as the sum of indicators that a given good-directed path is present, and write $|\pi\cap\rho|$ for the number of edges the two paths share. Then we have the variance formula
\[\Vv[N_k^\text{{d}}(a,b)]= \sum_{\ell=0}^{k}\sum_{\substack{ \pi, \rho \\ |\pi\cap\rho|=\ell}} \left[\Pv(\pi, \rho \subseteq \CMD )- \Pv(\pi\subseteq \CMD) \Pv(\rho\subseteq \CMD)\right].\]
Consider first the inner sum for $\ell=0$, i.e.\ when the two path have disjoint edge-sets.
Since at the time of pairing the $i$th half-edge, there are $\CL_n^\ast-2i+1$ free half-edges to pick from, the probability that both $\pi,\rho$ are present is exactly $\prod_{i=1}^{2k} (\CL_n^\ast - 2i +1)^{-1}$. On the other hand,  the square of the probability that a path present is $\Pv(\pi\subseteq \CMD)^2=\prod_{i=1}^{k} (\CL_n^\ast - 2i +1)^{-2}$. Comparing the two,  we get  for $\ell=0$, the inner sum is
\be\label{eq::l=0contribution}\ba &\sum_{\substack{\pi, \rho\\ |\pi\cap \rho|=0}}[\Pv(\pi, \rho \subseteq \CMD )-\Pv(\pi\subseteq \CMD) \Pv(\rho\subseteq \CMD)] \\
&= \sum_{\substack{\pi, \rho\\ |\pi\cap \rho|=0}} \left(\prod_{i=1}^{k} \frac{\CL_n^\ast-2i+1}{\CL_n^\ast-2i-2k+1}-1\right) \Pv(\pi\subseteq \CMD) \Pv(\rho\subseteq \CMD),\ea\ee
which, combined with the fact that $\CL_n^\ast=\CL_n(1+o(1))$, gives rise to the first term in the error term $e_{k,n}$ in \eqref{def::ekn}.

For $\ell=k$, that is, the two paths are identical,  $\Pv(\pi \subseteq \CMD)-\Pv(\pi \subseteq \CMD)^2 \le \Pv(\pi \subseteq \CMD)$, hence the inner sum can be bounded by the inequality
\be\label{eq::l=kcontribution} \sum_{\substack{\pi, \rho\\ |\pi\cap \rho|=k}} \Pv(\pi \subseteq \CMD) \le \Ev[N_k^{\text{d}}(a,b)], \ee
explaining the first term on the right hand side of \eqref{eq::wihZk_variance}.

Now we are left with handling the cases $1\le \ell \le k-1$. Note that in these cases we have to evaluate over all possible overlaps between the paths $\pi,\rho$. For this, note that the restriction that $\pi_i, \rho_i \in \Lambda_i^{\text{d}}$ and $\Lambda_i^{\text{d}}$ are disjoint sets implies that for each $i$ there are only two cases: either $\pi_i=\rho_i$ or $\pi_i \neq \rho_i$, but in both cases they are disjoint from all $\pi_j, \rho_j, j \neq i$. We will merge these cases into shapes. Let us call an \emph{excursion} of length $s$ a connected component of $\pi\setminus \rho$, that is, a consecutive sequence of edges where the two paths are not the same. Formally, for some $i$,  $(\pi_i, \pi_{i+1}) \neq (\rho_i, \rho_{i+1}), \dots, (\pi_{i+s-1}, \pi_{i+s}) \neq (\rho_{i+s-1}, \rho_{i+s})$, is an excursion if it is started and ended by the common edges $(\pi_{i-1}, \pi_{i}) =(\rho_{i-1}, \rho_{i})$ and $ (\pi_{i+s}, \pi_{i+s+1})= (\rho_{i+s}, \rho_{i+s+1})$ unless $i-1=-1$ or $i+s+1=k+1$, in which cases there is no edge before/after the excursion, respectively. Due to the property that $\Lambda_i^{\text{d}}$ are disjoint, note that there are exactly the same number of edges on the $\pi$ part of an excursion as on the $\rho$ part of the excursion.

Let us denote by $m$ the number of excursions, and again, we denote by $\ell:=|\pi\cap \rho|=\ell$ the total number of shared edges. For a fix $m$, there can be  $m-1, m$ or $m+1$ many segments of $\pi\cap \rho$, depending on whether none of, only one of, or both $a,b$ are part of an excursion. Let us thus introduce the indicators $\delta_a=1, \delta_b=1$ if vertex $a,b$ are parts of an excursion.

We can now define the class of shapes called
$\mathrm{Shape}_{m,\ell}$ corresponding to pairs of paths for which  $|\pi\cap \rho|=\ell$ and $\pi\setminus\rho$ consists of $m$ excursions.
That is, $\rho$ has $m$ edge-disjoint excursions from $\pi$, and between two consecutive excursions there is at least one edge in $\pi\cap \rho$. Note that the number of excursions $m$ is thus at most $\ell+1$.
Also note that each shape in $\mathrm{Shape}_{m,\ell}$ can be uniquely characterised by a sequence of numbers of the form
\be\label{eq::shape} ( \delta_a, \delta_b, (\ell_1, \ell_2, \dots, \ell_{m+1-\delta_a-\delta_b}), (e_1, e_2 \dots, e_m) ), \ee
where $\delta_a, \delta_b$ are indicators as before; $\ell_i$ is the number of edges in the $i$th connected component of $\pi\cap \rho$, with the property that $\sum_{i=1}^{m+1-\delta_a-\delta_b}\ell_i=\ell$; and $e_i$ is the number of edges on the $i$th excursion with the property that $\sum_{i=1}^m e_i=k-\ell$. See Figure for examples.

\begin{figure}
\centering
\subfigure{\label{fig::path1}
\includegraphics[keepaspectratio,width=10cm]{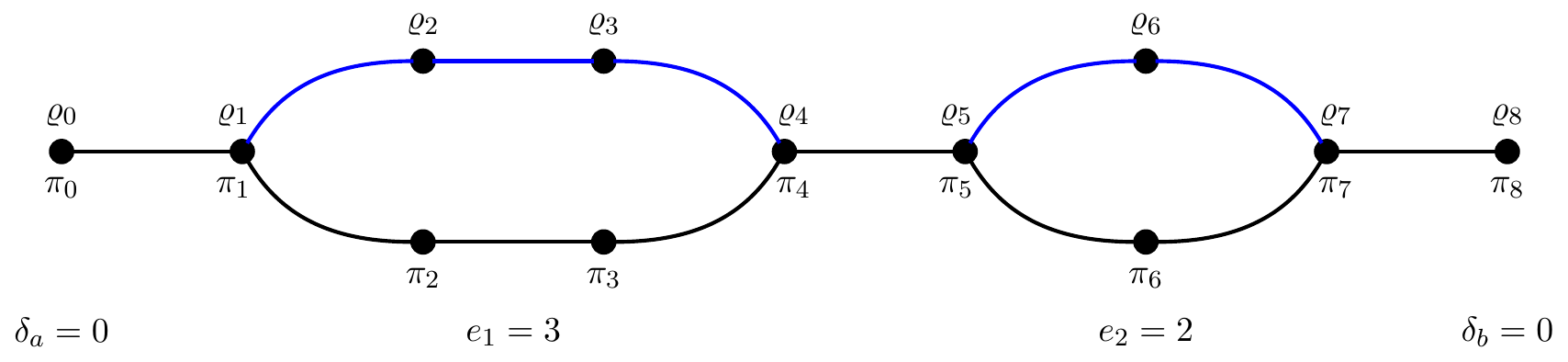}}
\subfigure{\label{fig::path2}
\includegraphics[keepaspectratio,width=10cm]{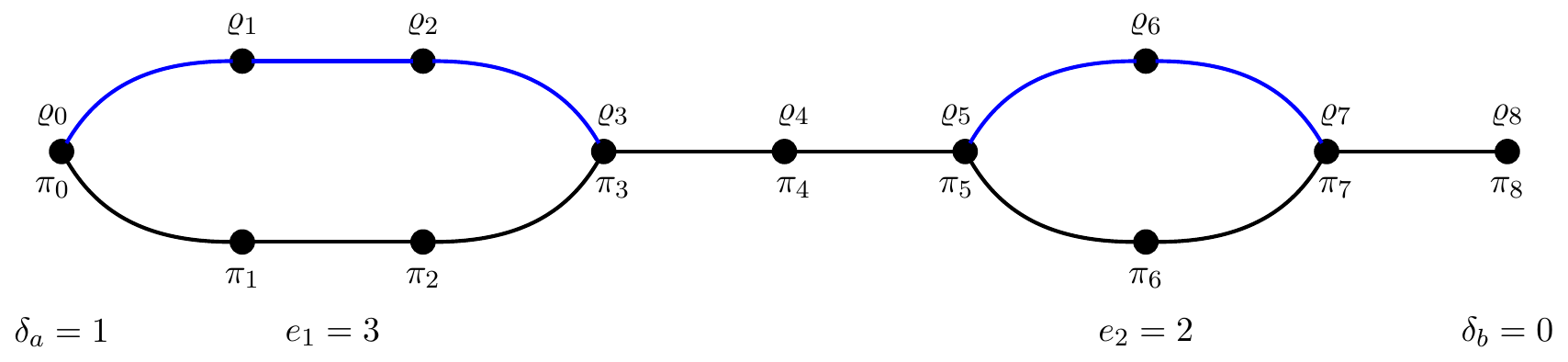}}
\subfigure{\label{fig::path3}
\includegraphics[keepaspectratio,width=10cm]{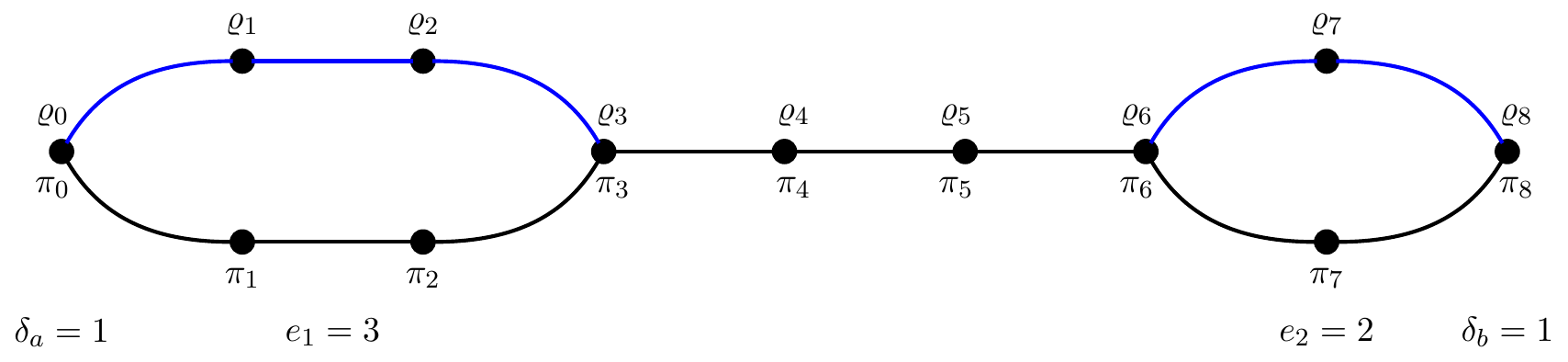}}
\caption{Paths of length $8$ belonging to $\mathrm{Shape}_{2,3}$: $m=2$ indicates that there are two excursions, $\ell=3$ means that the two paths share $3$ edges in total. On the first picture, the excursion do not start at the ends of the path, hence $\delta_a=\delta_b=0$, on the second picture, $\delta_a=1, \delta_b=0$, while on the third picture both excursions start at the ends, hence $\delta_a=\delta_b=1$. Note that in all cases, the number of degree three vertices is $2m-\delta_a-\delta_b$, and the shared edges form $m+1-\delta_a-\delta_b$ many connected components.   }\label{fig::paths}

\end{figure}

The contribution of paths with  $1\le \ell \le k-1$  to the  variance $\Vv[N_k^{\text{d}}(a,b)]$  can be bounded from above as follows:
\be\label{eq::variancecontribution} \sum_{\substack{\pi,\rho\\ 1\le|\pi\cap\rho|\le k-2}} \Pv(\pi,\rho \subseteq \CMD) \le \sum_{\ell=1}^{k-1} \sum_{m=1}^{\ell+1} \sum_{\sigma\in \mathrm{Shape}_{m,\ell}} \sum_{\substack{\pi,\rho\\ \mathrm{Shape}(\pi,\rho)=\sigma}} \Pv(\pi, \rho \subseteq \CMD).\ee
Note that if $|\pi\cap\rho|=\ell$ is fixed, then there are exactly $2k-\ell$ different edges in $\pi\cup \rho$, so that with fixed vertices and fixed half-edges,
\be\label{eq::stubs} \Pv(\pi, \rho\subseteq \CMD) = \prod_{i=1}^{2k-\ell} \frac{1}{\CL_n^\ast-2i+1}.\ee
If we now  fix only the vertices, but not the half-edges, then we have to multiply this with a combinatorial factor similar to that in \eqref{eq::expectedpath1} counting the number of possible variations of half-edges for fixed vertices $(\pi_i, \rho_i)_{1\le i\le  k}$.
Recall again that $\de_a=\ind\{a \in \mbox{ first excursion of } \pi\setminus \rho\}$ and $\de_b=\ind\{b \in \mbox{ last excursion of } \pi \setminus \rho \}$.
Let us write $d_{\sigma}(v)$ for the number of half-edges of $v$ used in the union of paths $\pi, \rho$ of shape $\sigma$, and in text we write degree$_\sigma$ for this degree. At the end of every excursion we have degree$_\sigma$-$3$ vertices, while on the excursions and inside segments of $\pi\cap \rho$ we have degree$_\sigma$-$2$ vertices.
Thus the combinatorial factor to pick half-edges, once fixing the vertices along the path (but not the half-edges) is at most
\be\label{eq::combi_factor}d_a(d_a-1)^{\de_a} d_b(d_b-1)^{\de_b}\!\!\!\!\!\! \prod_{\substack{s:\\d_{\sigma}(\pi_s)=3}}\!\!\!\!\!\! d_{\pi_{i_s}}(d_{\pi_{i_s}}-1)(d_{\pi_{i_s}}-2)\!\!\!\!\!\!\prod_{\substack{t:\\d_{\sigma}(\pi_t)=2}}\!\!\!\!\!\! d_{v_t}(d_{v_t}-1)\!\!\!\!\!\!\prod_{\substack{u:\\d_{\sigma}(\rho_u)=2\\ \rho_u\cap \pi=\varnothing}}\!\!\!\!\!\!  d_{v_u}(d_{v_u}-1). \ee

 Thus, establishing the contribution of a given fixed shape $\sigma$ in \eqref{eq::variancecontribution}, we should sum the product of \eqref{eq::stubs} and \eqref{eq::combi_factor} over all possible configurations $(\pi_i, \rho_i)_{i=1,\dots,k}$ forming the shape $\sigma$.
Note that the number of factors in \eqref{eq::stubs} equals the number of edges while the number of factors in \eqref{eq::combi_factor} equals the number of vertices in $\sigma$. Since the number of edges minus the number of vertices without counting $a$ and $b$ is at least $m + 1$ (equality if all the vertices on excursions are different -- see Remark \ref{rem::4degree} below when this is not true), we get
\be\label{eq::shape_pirho1} \ba &\sum_{\substack{\pi,\rho\\ \mathrm{Shape}(\pi,\rho)=\sigma}} \Pv(\pi, \rho \subseteq \CMD) \le \frac{1}{\CL_n^{m+1}}\prod_{i=1}^{2k-\ell} \frac{\CL_n}{\CL_n^\ast-2i+1} \cdot d_a(d_a-1)^{\de_a} d_b(d_b-1)^{\de_b} \\
&\times\!\!\!\!\!\!\prod_{\substack{s:\\ d_{\sigma}(\pi_s)=3}}\!\!\!\Bigg(\sideset{}{^\ast}\sum_{i \in \Lambda_{i_s}^{\text{d}}} \frac{d_i (d_i-1)(d_i-2)}{\CL_n} \Bigg)
\!\!\!\!\prod_{\substack{t:\\ d_{\sigma}(\pi_t)=2}}\!\!\!\Bigg(  \sideset{}{^\ast}\sum_{i \in \Lambda^{\text{d}}_{i_t}} \frac{d_i (d_i-1)}{\CL_n} \Bigg)
\!\!\!\!\prod_{\substack{u:\\ d_{\sigma}(\rho_u)=2\\ \rho_u\cap \pi=\varnothing} }\!\!\!\Bigg(  \sideset{}{^\ast}\sum_{i \in \Lambda^{\text{d}}_{i_u}} \frac{d_i (d_i-1)}{\CL_n} \Bigg) \ea \ee
By ta similar argument then that in the proof of Claim \ref{claim::Sbound}, the sums in the previous display are converging to $\gamma_{i_s}-\gamma_{i_s+1}$ and $\nu_{i_t}-\nu_{i_t+1}, \nu_{i_u}-\nu_{i_u+1}$, respectively. Thus, we get that the rhs\  of \eqref{eq::shape_pirho1} is at most
\be\label{eq::immed} \frac{1}{\CL_n^{m-1}}\prod_{i=1}^{2k-\ell} \frac{\CL_n}{\CL_n^\ast-2i+1} \cdot d_a^{1+\de_a} \frac{d_b^{1+\de_b}}{\CL_n^2}
\!\!\!\! \prod_{\substack{s:\\d_{\sigma}(\pi_s)=3}}\!\!\!\! \gamma_{i_s} \!\!\!\! \prod_{\substack{t:\\d_{\sigma}(\pi_t)=2}}\!\!\!\!\nu_{i_t}\!\!\!\!
\prod_{\substack{u:\\d_{\sigma}(\rho_u)=2 \\ \rho_u\cap \pi=\varnothing}}\!\!\!\!\nu_{i_u}.\ee
\begin{remark}\label{rem::4degree}\normalfont
 Note that even though the \emph{edges} of an excursion are different by definition, the \emph{vertices} in the excursions might still coincide, i.e. $\pi_i=\rho_i$ can happen even if $\pi_i, \rho_i$ are sitting in the middle of an excursion. In this case, instead of having the factor $d_{\pi_i}(d_{\pi_i}-1) d_{\rho_i}(d_{\rho_i}-1)$ in \eqref{eq::combi_factor} we have $d_{\pi_i}(d_{\pi_i}-1)(d_{\pi_i}-2)(d_{\pi_i}-3)$. Since in \eqref{eq::shape_pirho1}, 
 every original vertex gets a factor $\CL_n^{-1}$ as a normalisation, when $\pi_i=\rho_i$ is merged within an excursion, 
 there is a factor of the form
 \[ \sum_{\pi_i \in \Lambda_i^{\text{d}}} d_{\pi_i}(d_{\pi_i}-1)(d_{\pi_i}-2)(d_{\pi_i}-3)/ \CL_n^2 \]
 in \eqref{eq::shape_pirho1}.
 This converges by the Law of Large Numbers again, and is less than
 $C \wit u_{f(i)}^{5-\tau}/n$ where $f(i)=[t_b+\la i -T_r]$. If the vertices $\pi_i,\rho_i$ are not merged, we have a factor $\nu_i^2=\wit u_{f(i)}^{2(3-\tau)}$ in \eqref{eq::shape_pirho1} instead. Note that
 \[ \wit u_{f(i)}^{5-\tau}/n = o(\wit u_{f(i)}^{2(3-\tau)} )\] as long as $\wit u_{f(i)} = o(n^{1/(\tau-1)})$. Since the maximal degree in the graph is of this order, this is always the case. That is, it in not worth merging vertices on excursions. We can continue analysing formula \eqref{eq::immed}.
\end{remark}

Now we identify the indices $i_s,i_t,i_u$, using the restrictions $\pi_i, \rho_i\in \Lambda_i^{\text{d}}$.
The crucial observation is the following: follow the indices $\pi_1, \pi_2,\dots, \pi_{k-1}$ and $\rho_1, \rho_2,\dots, \rho_{k-1}$ along the two paths. If for some $i$, the vertices $\pi_i\neq\rho_i$ are degree$_\sigma$-$2$ vertices on an excursion, then the corresponding $\nu_i^2$ appears in the product in \eqref{eq::immed}. If $\pi_i=\rho_i$ is a degree$_\sigma$-$3$ vertex, then we have a factor $\gamma_{i}$ \emph{replacing} $\nu_i^2$ in the product. If $\pi_i=\rho_i$ is a degree$_\sigma$-$2$ vertex in $\pi\cap \rho$, then we only have a factor $\nu_{i}$ in the product (instead of $\nu_i^2$) in \eqref{eq::immed}.

Thus, dividing \eqref{eq::immed} by $\prod_{i=1}^{2k-2} \nu_i^2$ yields  that for each degree$_\sigma$-$3$ vertex we have a factor $\gamma_{i}/\nu_i^2$ and for each coinciding degree$_\sigma$-$2$ vertex we have a $1/\nu_{i}\le 1/\nu_{k-1}$ in the product.
Elementary calculation shows that
\be\label{eq::monoton} \gamma_{i}/\nu_i^2 = \left(\wit u_{\lfloor t_b+\la i-T_r\rfloor}\right)^{\tau-2}\ee
and since  the sequence $\wit u_i$ is decreasing in $i$, we get
\[\frac{\gamma_{i}}{\nu_i^2} \le \frac{\gamma_1}{\nu_1^2}.  \]
Thus, we can bound the contribution of every degree$_\sigma$-$3$ vertex by a factor $\gamma_1/\nu_1^2$, and every coinciding degree$_\sigma$-$2$ vertex by a factor $1/\nu_{k-1}.$ Since there are $m-\delta_a-\delta_b$ many degree$_\sigma$-$3$ vertices and $\sum_{j=1}^{m+1-\de_a-\de_b}(\ell_i-1) =\ell-m-1+\delta_a+\delta_b$ many coinciding degree$_\sigma$-$2$ vertices,
if we take out $\overline{\Ev[N_k(a,b)]}^2$ as in \eqref{eq::nkab-conv} from \eqref{eq::immed},
 we are left with the following upper bound
\be\label{eq::sumup} \ba &\sum_{\substack{\pi,\rho\\ \mathrm{Shape}(\pi,\rho)=\sigma}} \Pv(\pi, \rho \subseteq \CMD)
\le \overline{\Ev[N_k(a,b)]}^2 \frac{1}{d_a^{1-\de_a} d_b^{1-\de_b}} \left(\frac{\gamma_1}{\nu_1^2} \right)^{2m-\de_a-\de_b}\!\!\!\!\!\!\! \frac{1}{\nu_{k-1}^{\ell-m-1+\de_a+\de_b}} \frac{1}{\CL_n^{m-1}}\\
 &\quad\quad \quad\le\overline{\Ev[N_k(a,b)]}^2 \cdot \frac{1}{\CL_n^{m-1}} \left(\frac{\gamma_1^{2}\nu_{k-1}}{\nu_1^{4}}\right)^{m-1} \frac{1}{\nu_{k-1}^\ell} \left(\frac{\gamma_1 \nu_{k-1}}{\nu_{1}^2 d_a}\right)^{1-\de_a}  \left(\frac{\gamma_1 \nu_{k-1}}{\nu_{1}^2d_b}\right)^{1-\de_b}. \ea\ee
Recall that we have to sum this formula over all the shapes in $\mathrm{Shape}_{m,\ell}$, and over $m\ge 1$ and $\ell$ as in \eqref{eq::variancecontribution}.

When $m=1$, the number of $\sigma\in \mathrm{Shape}_{1,\ell}$ is $\ell-1$ if $\de_a+\de_b=0$, and $1$ if $\de_a+\de_b=1$, and $0$ if $\de_a=\de_b=1$.
Thus, fixing $m=1$ and summing in $\delta_a, \delta_b, \ell$ gives at most
\[   \left(\frac{\gamma_1 \nu_{k-1}}{\nu_{1}^2 d_a} +  \frac{\gamma_1 \nu_{k-1}}{\nu_{1}^2 d_b}\right) \cdot \sum_{\ell=1}^\infty \frac{1}{\nu_{k-1}^\ell} +  \frac{\gamma_1^2 \nu_{k-1}^2}{\nu_{1}^4 d_a d_b} \sum_{\ell=1}^\infty \frac{\ell-1}{\nu_{k-1}^\ell}\]

Thus the contribution from $m=1$ can be bounded from above by
\be\label{eq::m=1contribution}\ba \overline{\Ev[N_k(a,b)]}^2\left[\frac{\gamma_1 \nu_{k-1}}{\nu_1^2 (\nu_{k-1}-1)} \left(\frac{1}{d_a}+\frac{1}{d_b}\right) + \frac{\gamma_1^2 \nu_{k-1}^2}{\nu_{1}^4(\nu_{k-1}-1)^2} \frac{1}{d_a d_b}\right].\ea\ee
We are left with counting the contribution of shapes with $m\ge 2$. We can bound the number of shapes in $\mathrm{Shape}_{m,l}$ (for fixed $\delta_a, \delta_b$) similarly as in \cite[Lemma 9.18]{H10}
\be\label{eq::shapebound}|\{\sigma: \sigma \in \mathrm{Shape_{m,l}}\}|\le  {{k-\ell-1} \choose {m-1}} {{\ell-1}\choose{m-\de_a-\de_b}}.\ee
To explain this formula, recall the description of the shape in \eqref{eq::shape}. Note that since excursions are separated by at least one common edge, we have to pick $\ell_1, \ell_2, \ell_{m+1-\de_a-\de_b}$ in \eqref{eq::shape} so that for all $i$, $\ell_i\ge1$ and $\sum_{i} \ell_i=\ell$. This can be done in ${{\ell-1}\choose{m-\de_a-\de_b}}$ many ways. Then, we also have to pick the length of excursions so that $\sum_{i=1}^m e_i = k-\ell$ and each $e_i\ge 1.$ This can be done in ${{k-\ell-1} \choose {m-1}}$ many ways.

For $m\ge 2$, the factors in \eqref{eq::shapebound} can be bounded by
\[ {{k-\ell-1} \choose {m-1}}\le \frac{k^{m-1}}{(m-1)!},  \quad \quad
{{\ell-1}\choose{m-\de_a-\de_b}} \le \frac{\ell^{m-\de_a-\de_b}}{(m-\de_a-\de_b)!}\le k^m,  \]
since $\ell\le k$. So, the total number of shapes in $\mathrm{Shape}_{m,\ell}$ can be bounded by
\be\label{eq::shapebound2}|\{\sigma: \sigma \in \mathrm{Shape_{m,\ell}}\}|\le k\frac{k^{2(m-1)}}{(m-1)!},\ee
which is independent of $\ell$. Using this bound in \eqref{eq::sumup} when summing over the number of shapes, summing over $\de_a, \de_b\in \{0,1\}$, then over $\ell\ge 1$,  we get that the contribution of $m\ge 2$ terms can be bounded from above by $\overline{\Ev[N_k(a,b)]}^2$ times
\be\label{eq::m>2contribution}\ba &\left(1+\frac{\gamma_1\nu_{k-1}}{\nu_1^2 d_a}\right)\left(1+\frac{\gamma_1\nu_{k-1}}{\nu_1^2 d_b}\right) \frac{k}{\nu_{k-1}-1} \sum_{m=2}^{\infty}\left( \frac{\gamma_1^2\nu_{k-1}}{\nu_1^4 \CL_n}\right)^{m-1} \frac{k^{2(m-1)}}{(m-1)!} \\
&= \left(1+\frac{\gamma_1\nu_{k-1}}{\nu_1^2 d_a}\right)\left(1+\frac{\gamma_1\nu_{k-1}}{\nu_1^2 d_b}\right) \frac{k}{\nu_{k-1}-1} \left(e^{k^2\gamma_1^2 \nu_{k-1}/(\nu_{1}^4 \CL_n)} -1\right).\ea \ee
Now let us set $d_a:=\MBN$ and in $d_b$ we collect all the vertices that are less then $\nu_k$: these contain vertices with constant degree (say all the degrees smaller than $K$ for $F(K)=1/2$). This implies that $d_b\ge \CL_n/2$ whp. Combining the contribution for $\ell=0$ in \eqref{eq::l=0contribution}, $\ell=k$ in \eqref{eq::l=kcontribution}, and then $m=1$ in \eqref{eq::m=1contribution} and finally $m\ge 2$ in \eqref{eq::m>2contribution} yields \eqref{eq::wihZk_variance}.
\end{proof}

\bibliographystyle{abbrv}
\bibliography{refscompetition}

\end{document}